\let\ORIlabel\label
\let\ORIrefstepcounter\refstepcounter
\AddToHook{package/hyperref/before}{%
   \let\label\ORIlabel
   \let\refstepcounter\ORIrefstepcounter
}

\documentclass[onefignum,onetabnum]{siamart220329} 

\usepackage{amsmath}

\usepackage{amsfonts, psfrag, graphicx, mathtools, stackrel, geometry, amssymb,enumitem,indentfirst,float,color,enumerate,graphicx,subcaption,pifont,mathtools,mathrsfs}
\usepackage{pbox}
\usepackage{booktabs, cellspace, hhline}
\usepackage[numbered]{mcode}
\newcommand{\RG}[1]{{\color{black}{#1}}} 
\newcommand{\RGI}[1]{{\color{black}{#1}}} 

\usepackage[hyperpageref]{backref}

\usepackage{comment,multicol,xspace}
\usepackage{makecell}

 \usepackage{booktabs}
 \usepackage{multirow}
 \usepackage{makecell}
 \usepackage{threeparttable}

\numberwithin{equation}{section}
\numberwithin{figure}{section}

\usepackage{etoolbox}
\usepackage[title]{appendix}
\newtheorem{remark}{Remark}[section]
\newtheorem{asump}{Assumption}[section]

\newcommand{\commentout}[1]{{}} 

\newcommand{\bfb}{{\bf b}}

\newcommand{\bfH}{{\bf H}}

\newcommand{\bfn}{{\bf n}}

\newcommand{\bfv}{{\bf v}}

\newcommand{\bfx}{{\bf x}}

\newcommand{\bfsigma}{\boldsymbol{\sigma}}

\newcommand{\mV}{\mathcal V}

\newcommand{\mM}{{\mathcal M}}

\newcommand{\bV}{{\mathbb V}}
\newcommand{\bQ}{{\mathbb Q}}

\usepackage{comment,multicol,xspace}
\usepackage[all]{xy} 
\usepackage{tikz-cd}
\usepackage{adjustbox}
\usepackage{multirow}
\usepackage{array}

\makeatletter
\MHInternalSyntaxOn
\def\MT_leftarrow_fill:{%
  \arrowfill@\leftarrow\relbar\relbar}
\def\MT_rightarrow_fill:{%
  \arrowfill@\relbar\relbar\rightarrow}
\newcommand{\xrightleftarrows}[2][]{\mathrel{%
  \raise.55ex\hbox{%
    $\ext@arrow 0359\MT_rightarrow_fill:{\phantom{#1}}{#2}$}%
  \setbox0=\hbox{%
    $\ext@arrow 3095\MT_leftarrow_fill:{#1}{\phantom{#2}}$}%
  \kern-\wd0 \lower.55ex\box0}}
\MHInternalSyntaxOff
\makeatother

\newcommand{\dd}{\,{\rm d}}

\newcommand{\vertiii}[1]{{\left\vert\kern-0.25ex\left\vert\kern-0.25ex\left\vert #1
    \right\vert\kern-0.25ex\right\vert\kern-0.25ex\right\vert}}
    \newcommand{\vertii}[1]{{\left\vert\kern-0.25ex\left\vert #1
    \right\vert\kern-0.25ex\right\vert}}

\makeatletter
\def\munderbar#1{\underline{\sbox\tw@{$#1$}\dp\tw@\z@\box\tw@}}
\makeatother

\begin{document}

\headers{Inexact projected gradient descent}{}
\title{
Inexact projected preconditioned gradient methods\\
with variable metrics: 
\RGI{a Lyapunov convergence theory} \\
(extended version)
}
\author{
Ruchi Guo \thanks{School of Mathematics, Sichuan University (ruchiguo@scu.edu.cn).}
\and Jun Zou \thanks{Department of Mathematics, The Chinese University of Hong Kong (zou@math.cuhk.edu.hk).} 
\funding{
The work of the first author was supported by National Key R\&D Program of China 2025YFA1018700, NSFC grant 12571436, 
the Fundamental Research Funds for the Central Universities, NSF grant DMS-2012465.
The work of the second author was substantially supported by the Hong Kong RGC General Research Fund (projects 14306623 and 14310324) 
and NSFC/Hong Kong RGC Joint Research Scheme 2022/23 (project N$\_$CUHK465/22).}
}


\date{}
\maketitle
\begin{abstract}
Projected gradient methods are widely used for constrained optimization. 
A key application is for partial differential equations (PDEs), 
where the objective functional represents physical energy and the linear constraints enforce conservation laws.
However, computing the projections onto constraint sets generally requires solving large-scale ill-conditioned systems.
A common strategy is to relax projection accuracy and apply preconditioners, 
which leads to inexact preconditioned projected gradient descent (IPPGD) methods studied here.
Furthermore, variable preconditioners dynamically incorporating updated nonlinear information often enhance convergence rates.
However, due to the complex interplay between inexactness and adaptive preconditioners,
the theoretical analysis and the dynamic behavior of the IPPGD methods still remain quite open.
We propose an effective strategy for constructing the inexact projection operator 
and develop a gradient-type flow to model the resulting IPPGD methods.
Discretization of this flow not only recovers the original IPPGD method but also yields a potentially faster novel method. 
Furthermore, we apply Lyapunov analysis, designing a delicate Lyapunov function, to prove the exponential convergence at the continuous level and linear convergence at the discrete level  \RGI{under certain assumptions}. 
Finally, we validate our approach through numerical experiments, demonstrating computational efficiency.
\end{abstract}



\section{Introduction}
Given two abstract Hilbert spaces $\bV$ and $\bQ$,
a nonlinear functional $f:\bV \rightarrow \mathbb{R}$, 
and a linear operator $B:\bV\rightarrow \bQ$, we study the constrained optimization problem:
\begin{equation}
\begin{split}
\label{minmax1} 
\min_{u\in \bV} ~ f(u)  ~~~~ \text{subject to} ~~ B u = 0.
\end{split}
\end{equation}
Here $\mathbb{V}$ is equipped with an inner product $\langle \cdot,\cdot \rangle_{\mathbb{V}}$,
but the subscript $\mathbb{V}$ is usually omitted for simplicity.
The gradient $\nabla f(u)$ is formally defined as a linear functional on $\mathbb{V}$, i.e.,
$$
\langle \nabla f(u), v \rangle := \lim_{\epsilon \rightarrow 0} \frac{f(u+\epsilon v) - f(u)}{\epsilon},
$$
given that the limit exists.
As Hilbert spaces are reflexive, $\nabla f$ is identified as an element in $\mathbb{V}$.

Gradient-based methods are widely used in optimization to find critical points, 
but they often converge slowly, 
particularly in ill-conditioned problems such as those arising in numerical PDEs.
To address this issue, two primary approaches have been developed to accelerate convergence.
The first one is to adjust the gradient direction by applying a symmetric positive definite (SPD) operator $M^{-1}$ to
$
\nabla f
$,
where $M$ is a metric and $M^{-1}$ is known as a preconditioner.
With a suitable $M$, the condition number may be significantly reduced; 
see the definition and related discussions around \eqref{CondNum}.
The trivial case $M = I$ simply leads to the standard gradient $\nabla f$, which is computationally straightforward but converges slowly.
Alternatively, $M$ can be chosen as the Hessian matrix of $f$, 
resulting in the projected Newton's method.
This approach is also closely related to the Sobolev gradient method \cite{2020PatrickDaniel},
where $M^{-1}\nabla f$ can be interpreted as the Riesz representative of $\nabla f$ within a subspace of $\mathbb{V}$. 
In the following discussion, we employ the notation of $\langle u, v \rangle_M : = \langle u,  M v \rangle$ and $\| u \|^2_M = \langle u, u \rangle_M$.

Constrained optimization problems are often addressed via Projected Gradient Descent (PGD) methods, 
a class of iterative methods that enforce constraints while descending along the gradient direction
\cite{1981Botsaris,2021GaoSonAbsilStykel,2020PatrickDaniel,2010PARIMAHMICHAEL,2009ShikhmanStein}.
When coupled with preconditioners such as $M^{-1}$,
the projections $P_M:\mathbb{V}\rightarrow \ker(B)$ are typically defined with respect to the metric $M$ to ensure convergence:
\begin{equation}
\label{soblev_grad_3}
\langle P_{M} u,v \rangle_M = \langle u,v \rangle_M, ~~~~~ \forall v \in \ker(B).
\end{equation}
It is well known that the first-order optimality condition tells
\begin{equation}
\label{opt_cond}
\langle \nabla f(u^{\star}), v \rangle = 0, ~~ \forall v \in \ker(B) ~~~~ \Longleftrightarrow ~~~~  P_M M^{-1} \nabla f(u^{\star})= 0.
\end{equation}
To fully capture the nonlinear system information, preconditioners should be updated dynamically.
Then, given a sequence of metrics $\{M_k\}_{k\ge 0}$, the Projected Preconditioned Gradient Descent (PPGD) method reads as
\begin{equation}
\label{PGD0}
u_{k+1} = P_{M_k} ( u_k - \alpha_k M^{-1}_k \nabla f(u_k) ).
\end{equation}
This classical method \eqref{PGD0} can be traced back to the early work in \cite{1964Goldstein,1966LevitinPolyak,1967Polyak}, 
often referred to as the Goldstein-Levitin-Polyak method,
which has been extensively studied in the literature \cite{1976Bertsekas,1960Rosen,1961Rosen,1981Dunn,1974Goldstein}.
Moreover, we highlight that \eqref{PGD0} can be regarded as a special case of the Spectral Projected Gradient methods with variable preconditioners \cite{2009BirginMartnezRaydan,2000BirginMartnezRaydan,2003BirginMartnezRaydan,2014BirginMartnezRaydan} in the sense that the search direction $d_k = -P_{M_k} M^{-1}_k \nabla f(u_k)$ is exactly the minimizer of the following subproblem:
\begin{equation}
\label{subproblem}
\min_{d\in \ker(B)} Q_k(d): = \frac{\| d \|^2_{M_k}}{2} + \langle \nabla f(u_k), d \rangle.
\end{equation}
Additionally, we also refer readers to the Lagrange multiplier methods, such as primal-dual methods or Uzawa-type methods \cite{2006Bacuta,2000BramblePasciakVassilev,uzawa1958iterative}.
In particular, inexact Uzawa methods have also been extensively investigated for accelerating the convergence in the literature; 
see \cite{1997BramblePasciakVassilev,chen1998global,cheng2003inexact,2001HuZou,2002HuZou,2007HuZou} for instance.

Although appropriate preconditioners can significantly reduce the condition number, 
solving several large-scale linear systems at each iteration is still computationally expensive. 
Indeed, it is often impractical and unnecessary to compute the exact projection at every iteration,
which motivates the idea of the inexact preconditioned projected gradient descent (IPPGD) methods.
A closely related concept, the Inexact Spectral Projected Gradient methods, has been explored in the literature \cite{2000BirginMartnezRaydan,2005AndreaniBirginMartnez, 2022DouglasMaxTiago, 2009GomesSantos, 2006WangLiu}.
The inexact oracle method described in \cite{2014DevolderGlineurNesterov} serves as another significant methodology closely related to our current study.
However, the analysis of all these methods can be challenging; see the discussions around \eqref{diffi} below and Section \ref{sec:flow}.

Inexact projections are typically implemented by solving the subproblem \eqref{subproblem} through iterative methods with a limited number of inner iterations;
see the semi-smooth Newton-CG method \cite{2012JiangSunToh} and
Dykstra's algorithm \cite{1983Dykstra,2003BirginMartnezRaydan} for instance.
\RGI{In this study, we develop a convergence theory based on a delicate inexact projection operator $\widetilde{P}_{\mM}$.}
It is constructed via Schur complement approximation:
$\mM$ denotes a metric set $\mM=\{ M, \widetilde{S} \}$ for the projection metric $M$ and the Schur complement approximation $\widetilde{S}$, i.e., $\widetilde{S}\approx S:= B M^{-1}B^T$, with $M$ and $\widetilde{S}$ 
\RGI{both being} symmetric positive definite (SPD) linear operators.
See Section \ref{subsec:ipp_operator} for the detailed discussion.
This operator can collectively integrate the preconditioning and inexactness mechanism.
Given a sequence of \RG{metric sets} $\{\mM_k\}_{k\ge 0}$ and time steps $\{\alpha_k\}_{k\ge 0}$, 
mimicking \eqref{PGD0}, we obtain a natural IPPGD method:
\begin{equation}
\label{PPGD0}
u_{k+1} = \widetilde{P}_{\mM_k} ( u_k - \alpha_k M^{-1}_k \nabla f(u_k) ).
\end{equation}
But our theoretical analysis and numerical experiments both suggest that this choice may not be optimal.
With the tool of  ordinary differential equations (ODEs), by studying the dynamics at the continuous level,
we propose a novel method given in \eqref{PPGD1_disc} that admits faster convergence.



While inexactness can enhance computational efficiency, its analysis often presents substantial challenges.
\RGI{Convergence theory regarding variable preconditioners for inexact projected preconditioned gradient methods remains relatively limited.}
For instance, global linear convergence was achieved in \cite{2018PatrascuNecoara}, but only under the assumption of a fixed trivial identity metric.
The research in \cite{2023AguiarFerreiraPrudente,2022FerreiraLemesPrudente} 
studied a feasible inexact projection, 
proving the sublinear convergence also for the identity metric.
To the best of the authors' knowledge, 
existing theoretical frameworks for these methods are inherently restricted to non-variable metrics. 
For general applications, variable metrics are needed to achieve better preconditioning effects, 
but the analysis would then be much more involved.
In \cite{2003BirginMartnezRaydan,2005AndreaniBirginMartnez}, 
the authors considered an inexact spectral PGD method
with the variable metric approximating the Hessian matrix, i.e., the quasi-Newton method,
but their analysis relies on exact projections.
The authors in \cite{2022DouglasMaxTiago,2006WangLiu} proved the global convergence, 
yet the rate of convergence remains open.
In \cite{2012JiangSunToh},
the optimal rates for non-smooth problems were established under two critical assumptions:
 (i) monotonic Loewner-order decrease of the metric operators and (ii) corresponding decay of projection inexactness.
These conditions are often unattainable for many PDE-related problems.

\begin{comment}
In summary, the IPPGD method introduces several questions that the current theories have yet to adequately address: 

\begin{itemize}

\item The first question concerns the ODE model of the IPPGD method. 
What is the appropriate flow to effectively describe and analyze the algorithms?
How can Lyapunov analysis help in estimating convergence at both the continuous and discrete levels, 
particularly when dealing with variable metrics? 

\item Our numerical experiments suggest that the accuracy of the IPPGD method can be controlled by the step size, 
even while maintaining a substantially large inexactness level $\delta$. 
Can we give a precise mathematical description of the point that the sequence converges to?
 
\end{itemize}
The last item above could be particularly advantageous in numerical solutions of PDEs. 
Here, a larger step size might facilitate quicker convergence while permitting certain error in the solution,
which can be adjusted to align with the mesh size.
\end{comment}

In this work, we resort to ODE models and Lyapunov analysis to show the optimal linear convergence of the IPPGD method.
The approach of analyzing optimization algorithms through ODEs has gained wide attention, 
especially for acceleration methods \cite{2022LuoChen,2022ShiDuJordanSu,2016SuBoydCandes}.
Concurrently, Lyapunov analysis is increasingly recognized as a critical tool in the optimization community,
as it can offer a systematic path to quantify stability and convergence \cite{chen2023transformed, 2025PaulJesusKonstantinos,2023Luo, 2022LuoChen, 2017PolyakShcherbakov, 2021SanzZygalakis,2021WilsonRechtJordan}.
The application of Lyapunov analysis to saddle point systems can be found in \cite{2024ChenGuoWei,2026ZengZhanGuoWei}.
However, to the best of our knowledge, 
little on the potential of these approaches for projected preconditioned methods has been explored in the literature.
Designing an appropriate flow and a suitable Lyapunov function usually remains a significant challenge,
which, together with the difficulties in analysis, is discussed in detail in Section \ref{subsec:prelimin}.

Our contributions in this work are multifold.
\RGI{We first modify an existing ODE model in the literature \cite{1973TANABE,1980Tanabe} such that the resulting flow can capture the dynamics of the IPPGD method in \eqref{PPGD0},
appropriately handling inexactness and variable metrics.
Moreover, discretizing this flow not only recovers the original IPPGD method but also produces a faster novel method benefiting from a relaxation parameter.
By systematically integrating this flow and the inexact projection with variable metrics, 
we rigorously establish a so-called strong Lyapunov property at both the continuous and discrete levels (see Theorems \ref{thm_str_cont_lya} and \ref{thm_lyapunov_disc}). The notion of a strong Lyapunov function has long been established in the dynamical-systems and control-theory literature \cite{1998ClarkeLedyaevStern,2004MazencNesic}, and related strong Lyapunov-type properties have subsequently been used to analyze optimization algorithms \cite{2022LuoChen,2021SanzZygalakis,2022ShiDuJordanSu,2016SuBoydCandes}.
}
Furthermore, our theoretical analysis and numerical experiments suggest that the accuracy of the IPPGD method can be controlled by not only the inexactness level but also the step size.
Specifically, it converges to a solution that retains certain approximation accuracy, even while accommodating a significantly large inexactness level $\delta$. 
It could be particularly advantageous in numerical solutions of Partial Differential Equations (PDEs):
a larger step size may promote faster convergence, 
allowing for a manageable error in the solution that can be tailored to match the mesh size.

We defer the literature review of the related studies regarding the conventional PGD flow to the next section
along with an introduction to the proposed inexact projection operator.
In Section \ref{sec:inexact_est}, we prepare some useful estimates regarding the inexact projection \RG{operators
and show} the existence and uniqueness of the equilibrium.
In Sections \ref{sec:conv_contin} and \ref{sec:conv_disc} respectively, we prove exponential convergence at the continuous level and the linear convergence at the discrete level.
In Section \ref{subsex_exam}, we present some application of IPPGD methods and numerical results.
We conclude in the last section.

\section{The proposed flow and inexact projection}
\label{sec:flow}

In this section, we discuss the existing work for projected gradient flow and propose our flow to deal with the inexactness and variable preconditioners.
Then, we introduce a special inexact projection operator that is suitable for nonlinear PDEs.

\subsection{The flow for the IPPGD method}
 
Continuous flows often provide a deeper understanding of the mechanism and dynamics underlying discrete iterations. 
There is a long history of studying optimization algorithms through the lens of ODEs;
see the early work in \cite{1987Amaya,1978Botsaris}.
\RG{In particular, ODE models have been widely used to analyze projection-type methods, but 
only for the exact PGD methods \cite{1973TANABE,1980Tanabe, 
2020PatrickDaniel,2010PARIMAHMICHAEL}. 
For more practical and more efficient numerical realizations, it would be natural to incorporate 
preconditioning and inexact gradient projections in the PGD-type methods (IPPGDs). 
But then the convergence analysis, especially the rate of convergence, of the IPPGDs would be 
much more technical and challenging, compared to the analysis of exact PGD methods, 
as some effective variable metrics must be introduced 
in the analysis. This may be why there are still no (rigorous) studies in the literature despite the significance 
of the topic in both mathematical understanding and numerical solutions of constrained optimization problems and 
PDEs, which forms the main motivation and focus of this work.
To do so, a natural and fundamental question arises: What is the appropriate and effective ODE 
to model the dynamics of the projected gradient flow when preconditioning, inexact gradient projection 
and variable metrics are incorporated in IPPGDs?
}

\RG{
It is clear that such a flow should meet two requirements: (1) it should effectively model the preconditioning and 
inexactness involved, and (2) it should recover the IPPGD method of interest in the discrete setting.}
For the case of exact projections, the most natural choice is \cite{1973TANABE,1980Tanabe}
\begin{equation}
\label{soblev_grad_2}
u' = - P_{M(t)} M^{-1}(t) \nabla f(u).
\end{equation}
This flow has also been employed in \cite{2020PatrickDaniel,2010PARIMAHMICHAEL}
for solving Gross–Pitaevskii eigenvalue problems, 
where $P_M$ is a \RG{projection onto the surface of a sphere}, admitting a simple closed form for computation.
We also refer readers to the related discussions in \cite{1981Botsaris,2021GaoSonAbsilStykel,2009ShikhmanStein}.
When the trajectory initiates from an infeasible point, 
Tanabe in \cite{1973TANABE,1980Tanabe} slightly modified the flow by adding a term involving the constraint,
making the trajectory gradually move towards the feasible manifold,
resulting in
\begin{equation}
\label{PGD1}
u'(t) + u(t) - P_{M}( u(t) -  M^{-1} \nabla f(u(t))) = 0,
\end{equation}
which was then studied by Yamashita in \cite{1980Hiroshi}, 
Evtushenko-Zhang in \cite{1994EvtushenkoZhadan}
and Schropp-Singer in \cite{2000SchroppSinger}.
We refer readers to a comprehensive review article in \cite{1989BrownBartholomew}.
Among all the aforementioned works, to our best knowledge, 
only \cite{2020PatrickDaniel,1980Hiroshi} include variable projection metrics in their studies, i.e., $M=M(t)$ in \eqref{PGD1}.
However, there appears to be no existing research that uses the ODE \eqref{PGD1} to examine inexact projection methods,
even though Tanabe's work is close to this topic.
In fact, we will see later that \eqref{PGD1} is not very suited for this purpose.

The original flow in \eqref{soblev_grad_2} is certainly not suitable for modeling inexact projection methods.
To see this, let us replace $ P_{M}$ by $\widetilde{P}_{\mM}$.
Notice that $\widetilde{P}_{\mM}$ may be even invertible, provided with only a tiny perturbation to $P_{M}$.
Then the equilibrium point of the flow \eqref{soblev_grad_2} simply satisfies $\nabla f = 0$, certainly not the true minimizer.
Furthermore, \eqref{PGD1} also fails to model inexact projections.
To see this, let us employ a forward Euler method to discretize \eqref{PGD1} and obtain
\begin{equation}
\label{PGD1_disc}
u_{k+1} = (1-\alpha_k)u_k +  \alpha_k\widetilde{P}_{\mM_k} u_k - \alpha_k \widetilde{P}_{\mM_k} M^{-1} \nabla f(u_k)
\end{equation}
with a step size $\alpha_k$,
which unfortunately does not recover \eqref{PPGD0} as $u_k \neq \widetilde{P}_{\mM}u_k$.



\RGI{Inspired by \eqref{PGD1}, one key contribution of the present work is to recast \eqref{soblev_grad_2} into the following ODE} to investigate the dynamics of the IPPGD method in \eqref{PPGD0}:
\begin{equation}
\label{PPGD1}
u'(t) + u(t) - \widetilde{P}_{\mM(t)}( u(t) - \alpha(t) M^{-1}(t) \nabla f(u(t))) = 0,
\end{equation}
where we highlight that the projection metric $\mM(t)$ is time-dependent.
This model integrates the step size $\alpha(t)$ directly at the continuous level, 
setting it distinguished from existing dynamical models.
Comparing \eqref{PPGD1} with \eqref{soblev_grad_2},
there is an extra term $(\widetilde{P}_{\mM} - I ) u(t)$ in \eqref{PPGD1}
precisely designed to accommodate the inexactness. 
Specifically, it allows for a precise definition of the limit of the IPPGD method and facilitates the estimation of its accuracy relative to the true minimizer in terms of the step size and the inexactness level, which will all be discussed in Section \ref{subsec:equilibrium} in detail.

Let us discretize \eqref{PPGD1} by a simple forward Euler method with the step size $\tau_k$:
\begin{equation}
\label{PPGD1_disc}
u_{k+1} = (1-\tau_k)u_k + \tau_k \widetilde{P}_{\mM_k} (u_k -  \alpha_k M^{-1}_k  \nabla f(u_k)).
\end{equation}
Then, it is not hard to see that \eqref{PPGD1_disc} with $\tau_k=1$ recovers the original method in \eqref{PPGD0},
and this is a result not achievable by \eqref{PGD1_disc}.
In fact, \eqref{PPGD1_disc} produces a novel algorithm by incorporating the new parameter $\tau_k$.
Such a flow in \eqref{PPGD1_disc} was originally developed in \cite{1989Antipin} by Antipin
and subsequently analyzed in a series of works \cite{2003Antipin,2019May};
but in all these works, $\widetilde{P}_{\mM(t)}$ is selected as the exact projection $P_{M}$ for a fixed metric $M$.
To the best of our knowledge, its benefits for analyzing inexact projection methods have not been fully recognized by the community.
Surprisingly, we are able to give an explicit bound for $\tau_k$ in terms of $\alpha_k$ and inexactness parameters,
which can recover the case of $\tau_k=1$.
See the main Theorem \ref{thm_lyapunov_disc} for details.
Thus, with the ODE tool, we actually obtain the novel faster method \eqref{PPGD1_disc} while keeping the same computational cost. 
Our findings suggest that ODE is not only a theoretical tool for analysis but also produces more stable and efficient algorithms.

\subsection{The inexact projection operator}
\label{subsec:ipp_operator}
In this subsection, we develop our inexact projection operator for linearly constrained optimization problems 
which may be particularly beneficial for the numerical solution and analysis of nonlinear PDEs.
Recall the exact projection $P_M$:
\begin{equation}
\label{soblev_grad_4}
P_{M} = I - M^{-1}B^TS^{-1}B, ~~~ \text{with} ~ S = B {M}^{-1}B^T ~ \text{being the Schur complement}.
\end{equation}
Computing $P_M$ requires solving at least three large-scale linear systems: two $M^{-1}$ and one $S^{-1}$.
In general, the structure of $S$ could be quite complicated, 
making the computation of $S$, let alone its inverse, significantly expensive.
Many PDE-related problems can be written into constrained optimization where $B$ represents a differential operator;
see the example in Section \ref{subsex_exam},
making $S$ an elliptic-type differential operator. 
In numerical PDEs, there are many approaches to approximate its inverse, known as preconditioners \cite{1990BramblePasciakXu,2006HiptmairWidmerZou,2017XuZikatanov}.
This consideration naturally suggests the development of the following inexact projection operator:
\begin{equation}
\label{soblev_grad_5}
\RG{\widetilde{P}_{\mM} := I - {M}^{-1}B^T\widetilde{S}^{-1}B,} 
\end{equation} 
\RGI{where $\widetilde{S}^{-1}$ is a linear SPD operator approximating $S^{-1}$.
As both $\widetilde{S}$ and $S$ are SPD, $\widetilde{S}$ itself also results in an approximation to $S$, 
which can be explicitly estimated by \eqref{spd_inequa}.
}
\RG{The metric set $\mM=\{M,\widetilde{S}\}$ emphasizes that the inexactness mainly comes from the approximation $\widetilde{S}$ of the Schur complement $S$; 
as illustrated in more detail below.}

\RG{Generally speaking, there are mainly two sources of inexactness: the inexact computation of $M^{-1}$ and $S^{-1}$.
But we shall treat these two approximations in a unified manner, 
namely it suffices to treat the approximation of $S^{-1}$ only. 
Let $\widetilde{S}^{-1}$ be an approximation of $S^{-1}$.
Then we argue that the inexactness in $M^{-1}$ can be subsumed into the approximation of 
$S^{-1}$ by $\widetilde{S}^{-1}$ for analysis. 
To see this, consider an approximation of $M^{-1}$ by $\widetilde{M}^{-1}$, then 
the Schur complement corresponding to $\widetilde{M}^{-1}$ would be 
$B\widetilde{M}^{-1}B^T$, and we readily have the following estimate by the triangle inequality
\begin{equation}
\label{SMApprox}
\| \widetilde{S} - B\widetilde{M}^{-1}B^T \| 
\le \| \widetilde{S} - S \| + \| B (M^{-1} - \widetilde{M}^{-1})B^T \|
\end{equation}
for any matrix norm.
As $\widetilde{S}$ already approximates $S$,
\eqref{SMApprox} indicates that $\widetilde{S}$ serves also as an approximation to $B\widetilde{M}^{-1}B^T$ given that $\widetilde{M}$ approximates ${M}$ well.
Therefore, from the perspective of analysis, 
the algorithm can be recast into using $\widetilde{M}$ as the projection metric, 
with $\widetilde{S}$ serving as an approximation to the exact Schur complement $B\widetilde{M}^{-1}B^T$.
Hence, it is sufficient to use the Schur complement to capture all the inexactness.
Under this setting, the exact projection is achieved when $S= \widetilde{S}$.
In practice, one may employ direct or iterative methods to approximate $M^{-1}$ and $\widetilde{S}^{-1}$ up to certain accuracy;
see Section \ref{subsex_exam_subs2} for more details and an illustration with a specific example.
}

In addition, it is not hard to see the following properties:
\begin{subequations}
  \begin{align}
&P_M \widetilde{P}_{\mM} = \widetilde{P}_{\mM} P_M = P_M, \label{PPi} \\
&M {\widetilde{P}}_{\mM}  = {\widetilde{P}}^T_{\mM} M, ~~~~~ M P_M  = P^T_{M} M. \label{proj_Pi_eq1} 
  \end{align}
\end{subequations}
For simplicity of notation, we introduce the notation 
\RG{\begin{equation}
\label{lem_proj_id1_eq02}
\widetilde{\nabla}_{\mM}:= \widetilde{P}_{\mM} M^{-1}\nabla
\end{equation}
which can be understood as a modified gradient,
and differs from the exact projected gradient owing to the inexact projection $\widetilde{P}_{\mM}$.}
With \eqref{proj_Pi_eq1}, it is not hard to see
\begin{equation}
\label{lem_proj_id1_eq01}
\langle \widetilde{\nabla}_{\mM} f(u), v \rangle_{M}  = \langle \nabla f(u), \widetilde{P}_M v \rangle, ~~~~~ \forall u,v \in \bV.
\end{equation}
Such a way to construct the inexact projection is different from simply solving \eqref{subproblem} by iterative methods to limited accuracy, 
and it also takes advantage of the operator’s structure.

\section{Preliminaries for the analysis}
\label{subsec:prelimin}
In this section, we provide some basic estimates for the subsequent analysis and also outline the main developments 
in the analysis.

\subsection{Convexity and Lipschitz properties with preconditioners}

For any proper closed convex and $C^1$ function $f: \bV \rightarrow \mathbb{R}$, 
we define the Bregman divergence of $f$ as
$$D_{f}(u,v) : = f(u)-f(v)- \langle  \nabla f(v), u-v\rangle.$$
Then, the Lipschitz continuity and convexity with respect to the $M$-norm can be characterized as
\begin{subequations}
\begin{align}
    & \frac{\mu_{f,M}}{2}\left\|u-v\right\|^{2}_{M} \le D_{f}(u, v) \le \frac{L_{f,M}}{2}\left\|u-v\right\|^{2}_{M},
 \label{conv_Lip}  \\
    &  \langle \nabla f(u) - \nabla f(v), u-v \rangle \ge \frac{1}{L_{f,M}+\mu_{f,M}} \| \nabla f(u) - \nabla f(v) \|^2_{M^{-1}} + \frac{L_{f,M}\mu_{f,M}}{L_{f,M}+\mu_{f,M}} \| u - v \|^2_{M},
     \label{conv_Lip_new}
\end{align}
\end{subequations}
$\forall u, v  \in \bV$;
see the detailed discussions in \cite{2003Nesterov,2006Micheal}.
With the convexity and Lipschitz constants $\mu_{f,M}$ and $L_{f,M}$, we can define a condition number for $f$:
\begin{equation}
\label{CondNum}
\kappa_{f,M} = \frac{L_{f,M}}{\mu_{f,M}}.
\end{equation}
Furthermore, we have the following bounds \cite{2003Nesterov,2006Micheal}:
\begin{equation}
\label{eq: Breg bound}
     \frac{1}{2L_{f,M}} \|\nabla f(u) - \nabla f(v)\|^2_{M^{-1}} \leq D_f(u,v) \leq   \frac{1}{2\mu_{f,M}} \|\nabla f(u) - \nabla f(v)\|^2_{M^{-1}}, \quad \forall  u, v \in \bV.
\end{equation}

\RG{
\subsection{Outline of proof}
First, we summarize some key difficulties in the analysis.
The inexactness arising from the approximation 
$\widetilde{S}$ of the Schur complement $S$ diminishes many desired properties of projections, such as 
\begin{equation}
  \label{diffi}
\mathcal{R}(\widetilde{P}_{\mM}) \not\subset \ker(B), ~~~~~~ \widetilde{P}_{\mM}^2 \neq \widetilde{P}_{\mM}, ~~~~~~ \langle \widetilde{P}_{\mM}v -v, w \rangle_M \neq 0, ~~ \forall w\in \mV,
\end{equation}
where $\mathcal{R}$ denotes the image of an operator.
\eqref{diffi} may cause the trajectory of the governing ODE flow to deviate from the constraint set.
Moreover, the inexact projections interplaying with variable preconditioners significantly complicate the dynamics of the iterative algorithm. 
To see this, we note that the metric $M_k = M_k(u_k)$ is usually constructed according to the most recent iterate $u_k$ at each step,
and, consequently the limit of the IPPGD iteration \eqref{PPGD0} inevitably depends on $\lim_{k\rightarrow \infty} \mM_k$,
yet the sequence $\{\mM_k\}$ lacks a priori guarantees of convergence.
This inner dependence between $\{u_k\}$ and $\{M_k\}$ poses significant challenges in analysis,
as their convergence properties are inextricably linked.

To address these technical issues, we resort to Lyapunov analysis by designing a special Lyapunov function.
Given the interplay between inexactness and variable preconditioning metrics involved, 
an effective Lyapunov function for convergence analysis must exhibit two essential characteristics:
(i) independence from the variable metric to avoid some intractable complications during differentiation, 
and (ii) the ability to manage the trajectory's deviation from the constraint set. 

More specifically, let us suppose that the proposed flow \eqref{PPGD1} admits a unique stationary point, $u^{\star}_{\phi}$, which may be different from the true minimizer $u^{\star}$.
One straightforward Lyapunov function would be $\| u - u^{\star}_{\phi} \|^2_M$, but this appears to be undesirable here, 
as its derivative inevitably involves $M'(t)$ that is hard to manipulate while further complicating the analysis process.
Instead, we shall consider the following Lyapunov function
\begin{equation}
  \begin{split}
\label{cont_Lyap}
&\mathcal{E}(t)  =  \lambda \alpha \mathcal{E}^{(1)}(u(t)) + \mathcal{E}^{(2)}(u(t)), \\
\text{with} ~~ & \mathcal{E}^{(1)}(u) := D_f(u,u^{\star}_\phi)~~ \text{and} ~~ \mathcal{E}^{(2)}(u): = \frac{1}{2}\| (I - \widetilde{P}_{\mM_{\star}})(u - u^{\star}_{\phi}) \|^2_{M_{\star}} ,
  \end{split}
\end{equation}
where $\lambda$ is a positive constant to be specified later.
Notice that $\mathcal{E}^{(1)}$ replaces the usual $\| u - u^{\star}_{\phi} \|^2_M$ without involving the variable metric,
while the second term $\mathcal{E}^{(2)}$ exhibits a highly atypical nature, specially designed for handling the inexact projections.
In particular, for exact projections, it is not hard to see that $\mathcal{E}^{(2)}$ vanishes,
which reduces the algorithm and analysis into the standard scenario.
In fact, due to this vanishing property, we may expect that $\mathcal{E}^{(2)}$ is of a small higher-order quantity compared to $\mathcal{E}^{(1)}$.
But to appropriately balance these two terms, 
we introduce two scaling coefficients $\lambda$ and $\alpha$ for $\mathcal{E}^{(1)}$.

The analysis below is built to make the intuitive observations above rigorous.
In Assumption \ref{asump_M}, we shall introduce some quantities $\delta(t)$ to measure the inexactness and $\Theta_m(t)$ to measure how the variable metrics $M(t)$ and $\widetilde{S}(t)$ deviate from the final ones $M_{\star}$ and $\widetilde{S}_{\star}$. Then, the goal is to establish the following inequalities:
\begin{subequations}
\label{cont_Lyap_eq1}
\begin{align}
    & \frac{\dd }{\dd t}\mathcal{E}^{(1)}  \le   C\left( -\alpha \mathcal{E}^{(1)} 
+  \alpha^{-1} \mathcal{E}^{(2)} +  \alpha \delta^2 \Theta^2_m \right),  \label{cont_Lyap_eq11} \\
    &  \frac{\dd }{\dd t}\mathcal{E}^{(2)}  \le  C \left( -\mathcal{E}^{(2)} + (\alpha\delta)^2\mathcal{E}^{(1)} + \alpha^2 \delta^2 \Theta^2_m \right),
\end{align}
\end{subequations}
where $C$ represents some generic positive constants and its dependence on Lipschitz and convexity constants will be specified clearly during detailed analysis.
Given \eqref{cont_Lyap_eq1} with $\lambda$ and $\delta$ small enough, 
it is not hard to see that the Lyapunov function \eqref{cont_Lyap_eq1} admits the decay property
\begin{equation}
\label{cont_Lyap_eq2}
\frac{\dd}{\dd t}\mathcal{E}  \le   - \omega  \mathcal{E}
\end{equation}
for a suitable positive constant $\omega$ depending on the generic constants in \eqref{cont_Lyap_eq1}.

With this framework in place, our analysis concentrates on the following questions:
\begin{itemize}

\item Does $u^{\star}_\phi$ uniquely exist, and how does it differ from the true minimizer $u^{\star}$?

\item How can we establish the key inequalities in \eqref{cont_Lyap_eq1}?

\item How can the continuous Lyapunov analysis be extended to the discrete setting?

\end{itemize}
These questions will be addressed in the next three sections.
Through the analysis, we shall obtain explicit relations among $\delta$, $\Theta_m$, $\lambda$ and $\alpha$ for ensuring the desirable convergence.

}



\section{Inexactness estimates and equilibrium}
\label{sec:inexact_est}
This section provides fundamental estimates for the inexact projection, 
then shows a unique equilibrium point for the ODE model \eqref{PPGD1}.

\subsection{Inexactness estimates}
To begin with, for two symmetric linear operators $Q$ and $R$, we denote $Q \preccurlyeq R$ by $R - Q$ being positive semidefinite.
If $Q$ and $R$ are SPD, we have the well-known property: 
\begin{equation}
\label{spd_inequa}
c_1 Q \preccurlyeq R \preccurlyeq c_2 Q ~~~~ \Longleftrightarrow ~~~~~ \lambda(Q^{-1}R) \in [c_1,c_2].
\end{equation}
The following result will be frequently used throughout this work.
\begin{lemma}
\label{lem_SPD_general}
For two linear SPD operators $Q$ and $R$ with $c_1 Q \preccurlyeq R \preccurlyeq c_2 Q$, $c_1,c_2>0$, we have 
\begin{equation}
\label{ds}
(Q^{-1} - R^{-1}) R (Q^{-1} - R^{-1}) \preccurlyeq \max\{(1-c_1)^2,(1-c_2)^2\} R^{-1}.
\end{equation}
\end{lemma}
\begin{proof}
Let $A = (Q^{-1} - R^{-1}) R (Q^{-1} - R^{-1})$, and thus $RA = (RQ^{-1} - I)^2$. 
As $c_1 Q \preccurlyeq R \preccurlyeq c_2 Q$, the property \eqref{spd_inequa} yields $\lambda(RQ^{-1})\in [c_1,c_2]$.
Then, we obtain $\lambda(RA) \in \max\{(1-c_1)^2,(1-c_2)^2\}$ which leads to the desired result by the property \eqref{spd_inequa} again.
\end{proof}

The next lemma provides comprehensive estimates regarding the inexact and exact projections.
\begin{lemma}
\label{lem_Pi_est}
Given a set of SPD metrics $\mM=\{M,\widetilde{S}\}$, there uniformly holds
\begin{subequations}
\label{lem_Pi_est_eq0}
\begin{align}
  &   \| P_M  u \|^2_M \le \| \widetilde{P}_{\mM}  u \|^2_M, ~~~~ \forall u\in \bV,  \label{lem_Pi_est_eq01} \\ 
  &   \langle u, \widetilde{P}_{\mM} u \rangle_M \le \| u \|^2_M, ~~~~ \forall u\in \bV.  \label{lem_Pi_est_eq011}
\end{align}
If $ {S} \preccurlyeq  \widetilde{S}$ is assumed, 
then $\langle \cdot, \widetilde{P}_{\mM} \cdot \rangle_M$ forms an inner product and
\begin{align}
\label{lem_Pi_est_eq02}
\| \widetilde{P}_{\mM}  u \|^2_M \le \langle u, \widetilde{P}_{\mM} u \rangle_M , ~~~~ \forall u\in \bV.
\end{align}
Furthermore, if $(1-\epsilon)\widetilde{S} \preccurlyeq  {S} \preccurlyeq  \widetilde{S}$, 
with $\epsilon \in (0,1)$, then
\begin{align}
&(1-\epsilon)  \| ( I - P_M) u \|_M \le \| (I - \widetilde{P}_{\mM} )u \|_M \le  \| ( I - P_M) u \|_M , ~~~~ \forall u\in \bV,  \label{lem_Pi_est_eq03}\\
& \| (\widetilde{P}_{\mM} - P_M )^T u \|_{M^{-1}} \le \epsilon \| u \|_{M^{-1}}, ~~~~ \forall u\in \bV. \label{lem_Pi_est_eq04}
\end{align}
In addition, given two metric sets $\mM_1=\{M_1, \widetilde{S}_1 \}$ and $\mM_2=\{M_2, \widetilde{S}_2 \}$, 
assume 
$M_2 \preccurlyeq c M_1$,
and $(1-\epsilon_i) \widetilde{S}_i \preccurlyeq {S}_i \preccurlyeq  \widetilde{S}_i$, $i=1,2$, with $\epsilon_i \in (0,1)$. 
Then, there hold for any $ u\in \bV$ that
\begin{align}
& \| ( \widetilde{P}_{\mM_1} - \widetilde{P}_{\mM_2} \widetilde{P}_{\mM_1} ) u \|_{M_2} \le \min\{ \sqrt{c}\epsilon_1 \| u \|_{M_1},
c\epsilon_1 \| u \|_{M_2} \},  \label{lem_Dproj_new_eq0} \\
& \| ( \widetilde{P}_{\mM_1} - \widetilde{P}_{\mM_2} \widetilde{P}_{\mM_1} ) u \|_{M_2} 
\le \min\Big\{ \sqrt{c}\frac{\epsilon_1}{1-\epsilon_1} \| (I - \widetilde{P}_{\mM_1}) u \|_{M_1},
c\frac{\epsilon_1}{1-\epsilon_2} \| (I - \widetilde{P}_{\mM_2})u \|_{M_2}\Big\}. \label{lem_Dproj_new_eq01}
\end{align}
\end{subequations}
\end{lemma}
\begin{proof}
As the proof is a little technical, we put it in Appendix \ref{append_lem_Pi_est}.
\end{proof}



Next, we present the convexity and Lipschitz properties of the inexact projected gradient operator.
\begin{lemma}
\label{lem_convLip}
Under \eqref{conv_Lip}, there holds
\begin{subequations}
\label{lem_convLip2_eq0}
\begin{align}
     D_f(u,v) \leq   \frac{1}{2\mu_{f,M}} \|  \widetilde{\nabla}_{\mM}( f(u) -  f(v)) \|^2_{M}  , \quad \forall  u, v \in \ker(B),  \label{lem_convLip2_eq01} 
\end{align}
\RG{and if $\widetilde{S} \succcurlyeq  {S}$, then it holds that}
\begin{align}
D_f(u,v) \geq   \frac{1}{2L_{f,M}} \|  \widetilde{\nabla}_{\mM}( f(u) -  f(v)) \|^2_{M}  , \quad \forall  u, v \in \bV. \label{lem_convLip2_eq02}
\end{align}
\end{subequations}
\end{lemma} 
\begin{proof}
The proof follows from the techniques in \cite{2003Nesterov} with the properties in \eqref{lem_Pi_est_eq01}-\eqref{lem_Pi_est_eq02}.
Fixing a $v\in \ker(B)$, we introduce an auxiliary function: $\phi(u)= f(u) - \langle \nabla f(v),u \rangle$ satisfying
\begin{equation}
\label{lem_convLip_eq4}
D_{\phi}(w,u) = \phi(w) - \phi(u) - \langle \nabla \phi(u), w-u \rangle =  D_{f}(w,u). 
\end{equation}
Then, \eqref{conv_Lip} leads to $D_{\phi}(w,u) \geq \frac{\mu_{f,M}}{2}\| w - u \|^2_M$.
Thus, as a strongly-convex function, $\phi$ achieves the minimum at $v$ where $\nabla \phi(v) = 0$. 
Then, we \RG{obtain from \eqref{lem_convLip_eq4}} that
\begin{equation}
\label{lem_convLip_eq5}
\begin{split}
\phi(v) &= \min_{w\in \ker(B)} \phi(w) \geq \min_{w \in \ker(B)} \left[  \phi(u) + \langle \nabla \phi(u), w-u \rangle + \frac{\mu_{f,M}}{2}\| w - u \|^2_M \right] .
\end{split}
\end{equation}
To minimize the right-hand side of \eqref{lem_convLip_eq5} over $w \in \ker(B)$, 
let us take $u\in \ker(B)$ and $w = u + P_M \xi $ for any $\xi\in \bV$. Then, the direct computation yields
$$
\text{the right-hand side of \eqref{lem_convLip_eq5}} =: g(\xi) = \phi(u) + \langle \nabla \phi(u), P_M \xi \rangle + \frac{\mu_{f,M}}{2}\| P_M \xi \|^2_M.
$$
Establishing the equation for the critical point: 
$$
\nabla g(\xi) = \mu_{f,M} P^T_M M P_M \xi + P^T_M \nabla \phi(u) = \mu_{f,M} M  P_M \xi + P^T_M \nabla \phi(u) = 0
$$ 
where \eqref{proj_Pi_eq1} is used. 
Then, we have $ P_M \xi = - M^{-1 }P^T_M \nabla \phi(u) /  \mu_{f,M} $ 
which leads to the minimizer $w = u - P_M M^{-1 } \nabla \phi(u) / \mu_{f,M}$. 
We also note that $P_M M^{-1}\nabla \phi(u) = P_M M^{-1} \nabla(f(u) - f(v))$.
Putting this into \eqref{lem_convLip_eq5}, we then obtain
\begin{equation*}
\label{lem_convLip_eq5_1}
\frac{1}{2\mu_{f,M}} \| P_M M^{-1} \nabla(f(u) - f(v)) \|^2_M \geq \phi(u) - \phi(v)  = D_f(u,v).
\end{equation*}
Then, \eqref{lem_convLip2_eq01} follows from \eqref{lem_Pi_est_eq01} in Lemma \ref{lem_Pi_est}.
Next, we proceed to show \eqref{lem_convLip2_eq02}. 
By \eqref{lem_convLip_eq4} and \eqref{conv_Lip} we have $D_{\phi}(w,u) \le \frac{L_{f,M}}{2} \| w - u \|^2_M$.
Inputting $u - \widetilde{\nabla}_{\mM} \phi(u)/L_{f,M}$ into $w$ in \eqref{lem_convLip_eq5} and using that $v$ is the minimizer of $\phi$, we obtain
\begin{equation*}
\begin{split}
\label{lem_convLip_eq6}
\phi(v) 
  \leq \phi(u) - \langle \nabla \phi(u) , \widetilde{\nabla}_{\mM} \phi(u) \rangle/L_{f,M} + \frac{1}{2L_{f,M}}\| \widetilde{\nabla}_{\mM} \phi(u) \|^2_M 
 \le  \phi(u) -  \frac{1}{2L_{f,M}}\| \widetilde{\nabla}_{\mM} \phi(u) \|^2_M,
\end{split}
\end{equation*}
where we have used $(\nabla \phi(u) , \widetilde{\nabla}_{\mM} \phi(u)) \ge \| \widetilde{\nabla}_{\mM} \phi(u) \|^2_M$ from \eqref{lem_Pi_est_eq02} in the last inequality.
The proof is finished by using $\phi(u) - \phi(v)  = D_f(u,v)$.
\end{proof}

Notably, \eqref{lem_convLip2_eq01} above basically states that the modified gradient $\widetilde{\nabla}_{\mM}$ preserves the convexity property of $ f$ on $\ker(B)$. 
But the trajectory produced by \eqref{PPGD1} may not lie in the constraint set, 
and thus the estimate may not hold either. 
These differences will make the analysis more involved.
So we generalize \eqref{lem_convLip2_eq01} to the case of $u,v \notin \ker(B)$ in the next result.

\begin{lemma}
\label{lem_convLip3}
If $ \widetilde{S} \succcurlyeq  {S}$, then $\forall  u, v \in \bV$, there holds
\begin{equation}
\label{lem_convLip3_eq0}
 \langle  \nabla ( f(u) -  f(v)) , \widetilde{\nabla}_{\mM}  ( f(u) -  f(v)) \rangle \ge \mu^2_{f,M} /2 \| u-v \|^2_M -  L^2_{f,M} \| (\widetilde{P}_{\mM}- I) (u-v)  \|^2_M.
\end{equation}
\end{lemma} 
\begin{proof}
Let us denote $w:=u-v$. 
By \eqref{lem_proj_id1_eq01}, \eqref{conv_Lip} and \eqref{eq: Breg bound}, we have
\begin{equation}
\label{lem_convLip3_eq1}
\begin{split}
\langle w,  \widetilde{\nabla}_{\mM}  ( f(u) -  f(v)) \rangle_M & =  \langle \widetilde{P}_{\mM} w,  \nabla ( f(u) -  f(v)) \rangle  \\
& = \langle w,  \nabla ( f(u) -  f(v)) \rangle + \langle \widetilde{P}_{\mM} w - w ,  \nabla ( f(u) -  f(v)) \rangle  \\
&\ge \mu_{f,M} \| w \|^2_M -  L_{f,M} \| \widetilde{P}_{\mM} w - w \|_M \| w \|_M.
\end{split}
\end{equation}
Next, by H\"older's inequality, we obtain $(  w , \widetilde{\nabla}_{\mM}  ( f(u) -  f(v)) )_M \le \| w \|_M \|  \widetilde{\nabla}_{\mM}  ( f(u) -  f(v))  \|_{M}$
which yields, with \eqref{lem_convLip3_eq1}, that
 $$
  \|  \widetilde{\nabla}_{\mM}  ( f(u) -  f(v))  \|_{M} \ge \mu_{f,M}  \| w \|_M -  L_{f,M} \| \widetilde{P}_{\mM} w - w \|_M .
 $$
 Then, the desired result is concluded by \eqref{lem_Pi_est_eq02} in Lemma \ref{lem_Pi_est}
\end{proof}

\begin{remark}
    \label{rem_grad0}
    A direct corollary of \eqref{opt_cond} and \eqref{lem_convLip2_eq02} in Lemma \ref{lem_convLip} with the exact projection is 
    \begin{equation}
    \label{rem_grad0_eq1}
    f(u) - f(u^{\star}) \ge \frac{1}{2L_{f,M}} \| P_{M} M^{-1} \nabla  ( f(u) -  f(u^{\star})) \|^2_{M}, ~~~ \forall u\in \ker(B).
    \end{equation}
    When applying Lyapunov analysis to PGD with exact projections, 
    one natural choice of Lyapunov functions is $f(u)-f(u^{\star}) = D_f(u,u^{\star})$ due to \eqref{opt_cond} if $u^{\star}\in \ker(B)$. 
    Notice that \eqref{rem_grad0_eq1} makes it a positive function.
    However, the corresponding optimality condition ${\nabla}_{\mM} f(u^{\star})=0$ is generally not true if $u^{\star}\notin \ker(B)$.
    We shall design a delicate and effective Lyapunov function in \eqref{cont_Lyap} below.
    One key motivation actually comes from the term $\| (\widetilde{P}_{\mM}- I) (u-v)  \|^2_M$ in \eqref{lem_convLip3_eq0} above
    which is precisely attributed to $u,v \notin \ker(B)$;
otherwise it will vanish.
\end{remark}

According to \eqref{lem_Pi_est_eq02} above,
we need $\widetilde{S} \succcurlyeq  S$ to ensure that $\langle \cdot, \widetilde{P}_{\mM} \cdot \rangle_M$ qualifies as an inner product.
This condition is also needed in \eqref{lem_convLip2_eq02} of Lemma \ref{lem_convLip} for \RG{making the Bregman divergence positive}.
So, it will be consistently assumed in subsequent discussions. 
Scaling $\widetilde{S}$ can achieve this requirement.


\subsection{Existence, uniqueness and estimates of equilibrium solutions}
\label{subsec:equilibrium}
In this subsection, \RG{we consider the equilibrium of \eqref{PPGD1}, i.e., the stationary point of this ODE.} 
It can be identified as a fixed point of the following function:
\begin{equation}
\label{phi_rhs}
\phi(u;\mM_{\star},\alpha^{\star}) := \widetilde{P}_{\mM_{\star}}( u - \alpha^{\star} M^{-1}_{\star} \nabla f(u)),
\end{equation}
where $\mM_{\star} = \{ M_{\star}, \widetilde{S}_{\star} \}$ is a given metric set and $\alpha^{\star}>0$.
We will show the existence and uniqueness of the fixed point $u^{\star}_{\phi}$ of $\phi$ in Lemma \ref{lem_fixpt} below.
Apparently, different $\mM_{\star}$ and $\alpha^{\star}$ lead to different $u^{\star}_{\phi}$.
We then give in Lemma \ref{lem_u_diff} its error to the true minimizer 
which can be effectively controlled by the step size $\alpha$ and inexactness $\delta$.
At this stage we have not made any assumptions on the relation between the metric sequence $\{\mM(t)\}_{t\ge 0}$ and $\mM_{\star}$.
If $\mM(t)$ is assumed to be convergent to $\mM_{\star}$, 
it becomes both intuitive to see and more straightforward to prove that $u(t)$ also converges to $u^{\star}_{\phi}$.
However, this assumption may result in a circular argument as
constructing $M(t)$ usually relies on $u(t)$ in practice, i.e., $M(t) = M(u(t))$. 
Assuming convergence for the former without established convergence for the latter presents a substantial risk 
and dilemma; also see Remark \ref{rem_assump} for the related discussion.

The following result shows that $\phi$ in \eqref{phi_rhs} is a contraction.
\begin{lemma}
\label{lem_fixpt}
Assume ${S}_{\star} \preccurlyeq \widetilde{S}_{\star}$,
then the function $\phi$ defined in \eqref{phi_rhs} satisfies
$$
\| \phi(u) - \phi(v) \|_{M_{\star}} \le \max\{ |1- \alpha^{\star}L_{f,M_{\star}}|,  |1- \alpha^{\star}\mu_{f,M_{\star}}| \} \| u - v \|_{M_{\star}}.
$$ 
\RG{Therefore, for each $\alpha^{\star}\in (0, 2L^{-1}_{f,M_{\star}})$, there exists a unique fixed point of $\phi$,
which depends on $\alpha^{\star}$.}
\end{lemma}
\begin{proof}
By the assumption with \eqref{lem_Pi_est_eq02} and \eqref{lem_Pi_est_eq011},
we have  
\begin{equation*}
\begin{split}
\label{lem_fixpt_eq1}
& \| \phi(u) - \phi(v) \|^2_{M_{\star}}  \le \| (u - v) - \alpha^{\star} M^{-1}_{\star} ( \nabla f(u) -  \nabla f(v) )  \|^2_{M_{\star}} \\
 = & \| u - v \|^2_{M_{\star}} - 2\alpha^{\star} \langle u-v, \nabla f(u) -  \nabla^{\star} f(v) \rangle + (\alpha^{\star})^2 \| \nabla f(u) -  \nabla f(v) \|^2_{M^{-1}_{\star}} \\
 \le & \left(1 - \frac{2L_{f,M_{\star}}\mu_{f,M_{\star}}}{L_{f,M_{\star}}+\mu_{f,M_{\star}}}  \alpha^{\star} \right) \| u - v \|^2_{M_{\star}}
 - \left( \frac{2 }{L_{f,M_{\star}}+\mu_{f,M_{\star}}} - \alpha^{\star}  \right) \alpha^{\star}  \| \nabla f(u) -  \nabla f(v) \|^2_{M^{-1}_{\star}} \\
\le &   \left(1 - \frac{2L_{f,M_{\star}}\mu_{f,M_{\star}}}{L_{f,M_{\star}}+\mu_{f,M_{\star}}}  \alpha^{\star} \right) \| u - v \|^2_{M_{\star}} \\
 &   - \min\Bigg\{ L_{f,M_{\star}}^2 \left( \frac{2 }{L_{f,M_{\star}}+\mu_{f,M_{\star}}} - \alpha^{\star} \right), \mu_{f,M_{\star}}^2 \left( \frac{2 }{L_{f,M_{\star}}+\mu_{f,M_{\star}}} - \alpha^{\star}  \right) \Bigg\} \alpha^{\star} \| u - v \|^2_{M_{\star}},\\
 & =  \max\{ |1- \alpha^{\star}L_{f,M_{\star}}|^2,  |1- \alpha^{\star}\mu_{f,M_{\star}}|^2 \} \| u - v \|^2_{M_{\star}},
\end{split}
\end{equation*}
where in the last inequality, we have used \eqref{conv_Lip} and \eqref{eq: Breg bound}.
It yields the desired estimate by the assumption of $\alpha^{\star}$.
By Banach Fixed Point Theorem, the fixed point of $\phi$ exists.
\end{proof}
By Lemma \ref{lem_fixpt},
we know that $u^{\star}_{\phi}$ is well-defined, hence we obtain the existence and uniqueness of the equilibrium point.
Remarkably, Lemma \ref{lem_fixpt} is independent of the inexactness level $\delta$.


\begin{remark}
\label{rem_PMustar}
\eqref{opt_cond} shows $P_M M^{-1} \nabla f(u^{\star})=0$ for any SPD operator $M$.
However, the first-order optimality condition $P_{M_{\star}} M^{-1}_{\star} \nabla f(u^{\star}_{\phi})= 0$ only holds for $M_{\star}$.
To see this, we only need to apply $P_{M_{\star}}$ to each side of $u^{\star}_{\phi} = \widetilde{P}_{\mM_{\star}}(u^{\star}_{\phi} - \alpha^{\star} M_{\star}^{-1} \nabla f(u^{\star}_{\phi}) )$ with $P_{M_{\star}} \widetilde{P}_{\mM_{\star}} = P_{M_{\star}}$. 
\end{remark}


\begin{lemma}
\label{lem_u_diff}
\begin{subequations}
\label{lem_u_diff_eq0}
Assume $\alpha^{\star} \le 2 L_{f,M_{\star}}^{-1}$ such that the fixed point $u^{\star}_{\phi}$ uniquely exists,
and assume $(1-\delta^{\star}) \widetilde{S}_{\star} \preccurlyeq {S}_{\star} \preccurlyeq \widetilde{S}_{\star}$ with an inexactness level $\delta^{\star}$,
then there holds
\begin{align}
\label{lem_u_diff_eq00}
\| (I - \widetilde{P}_{\mM_{\star}}) (u^{\star} - u^{\star}_{\phi} ) \|_{M_{\star}}   \le 2\alpha^{\star} \delta^{\star}  \| \nabla f(u^{\star}) \|_{M^{-1}_{\star}} .
\end{align}
Additionally, if $\alpha^{\star} \le L_{f,M_{\star}}^{-1}$ and $\delta^{\star} \le ( 4 \kappa_{f,M_{\star}})^{-1}$, then there holds
\begin{align}
& \| u^{\star}_{\phi} - u^{\star} \|_{M_{\star}} \le 3\sqrt{\kappa_{f,M_{\star}}} \mu_{f,M_{\star}}^{-1/2} \delta^{\star} \sqrt{\alpha^{\star}} \| \nabla f(u^{\star}) \|_{M^{-1}_{\star}},\label{lem_u_diff_eq01}    \\
&   \| \nabla f(u^{\star}_{\phi}) - \nabla f(u^{\star}) \|_{M^{-1}_{\star}} \le 3 \sqrt{ L_{f,M_{\star}}} \kappa_{f,M_{\star}}  \delta^{\star} \sqrt{\alpha^{\star}} \| \nabla f(u^{\star}) \|_{M^{-1}_{\star}},
\label{lem_u_diff_eq02}  \\
&  \| \nabla f(u^{\star}_{\phi})  \|_{M^{-1}_{\star}} \le 2 \| \nabla f(u^{\star}) \|_{M^{-1}_{\star}}. \label{lem_u_diff_eq03} 
\end{align}
\end{subequations}
\end{lemma}
\begin{proof}
For simplicity, we denote $\eta = u^{\star} - u^{\star}_{\phi}$ and $\eta_f = \nabla f(u^{\star} ) - \nabla f(u^{\star}_{\phi})$.
It follows from $P_{M_{\star}} M_{\star}^{-1} \nabla f(u^{\star}_{\phi}) = 0$ and Remark \ref{rem_PMustar} that
\begin{equation}
\label{lem_u_diff_eq1}
\eta = \widetilde{P}_{\mM_{\star}} \eta +  \alpha^{\star} (\widetilde{P}_{\mM_{\star}}M^{-1}_{\star} - P_{M_{\star}} M^{-1}_{\star} )\nabla f(u^{\star}_{\phi}).
\end{equation}
Then, from \eqref{lem_Pi_est_eq04} in Lemma \ref{lem_Pi_est}, we obtain 
\begin{equation}
\label{lem_u_diff_eq1_new}
\| (I - \widetilde{P}_{\mM_{\star}}) \eta \|_{M_{\star}}   \le \alpha^{\star} \delta^{\star}  \| \nabla f(u^{\star}_{\phi}) \|_{M^{-1}_{\star}} .
\end{equation}
Next, we rewrite \eqref{lem_u_diff_eq1} to the following identity by using the definition of $\eta_f$:
\begin{equation}
\label{lem_u_diff_eq2}
\eta = \widetilde{P}_{\mM_{\star}} \eta - \alpha^{\star} \widetilde{P}_{\mM_{\star}}M^{-1}_{\star}\eta_f +  \alpha^{\star} (\widetilde{P}_{\mM_{\star}}M^{-1}_{\star} - P_{M_{\star}} M^{-1}_{\star} )\nabla f(u^{\star}).
\end{equation}
Noticing that $\langle (I - \widetilde{P}_{\mM_{\star}})\eta, \eta \rangle_{M_{\star}} \ge 0 $ by \eqref{lem_Pi_est_eq011}, and taking the $M_{\star}$-inner product of \eqref{lem_u_diff_eq2} with $\eta$,
we obtain
$\langle \eta_f , \widetilde{P}_{\mM_{\star}} \eta \rangle \le \langle (\widetilde{P}_{\mM_{\star}}M^{-1}_{\star} - P_{M_{\star}} M^{-1}_{\star} )\nabla f(u^{\star}), \eta \rangle_{M_{\star}} $,
which implies
\begin{equation}
\begin{split}
\label{lem_u_diff_eq3}
\mu_{f,M_{\star}} \| \eta \|^2_{M_{\star}} \le & \langle \eta_f, \eta \rangle 
\le  \underbrace{ \langle  \eta_f , \eta -  \widetilde{P}_{\mM_{\star}} \eta \rangle }_{R_1} 
+ \underbrace{ \langle (  \widetilde{P}_{\mM_{\star}} - P_{M_{\star}}  )M^{-1}_{\star} \nabla f(u^{\star}), \eta \rangle_{M_{\star}} }_{R_2},
\end{split}
\end{equation}
where we have also used \eqref{conv_Lip}.
Employing \eqref{lem_u_diff_eq1_new} with \eqref{conv_Lip} and \eqref{eq: Breg bound}, we have 
$$
R_1 \le  \alpha^{\star} \delta^{\star} L_{f,M_{\star}} \| \nabla f(u^{\star}_{\phi}) \|_{M^{-1}_{\star}} \| \eta \|_{M_{\star}}.
$$
As for $R_2$, we use \eqref{lem_Pi_est_eq03} in Lemma \ref{lem_Pi_est} and $\delta^{\star}\le 1/4$ to conclude 
$\| (I - P_{M_{\star}})\eta \|_{M_{\star}} \le \frac{4}{3} \| (I - \widetilde{P}_{\mM_{\star}})\eta \|_{M_{\star}}$.
Then, using \eqref{lem_Pi_est_eq04} with \eqref{lem_u_diff_eq1_new}, we have
\begin{equation}
\begin{split}
\label{lem_u_diff_eq4}
R_2 
& = \langle (  \widetilde{P}_{\mM_{\star}} - P_{M_{\star}}  )M^{-1}_{\star} \nabla f(u^{\star}), \eta - P_{M_{\star}} \eta \rangle_{M_{\star}} \\
& \le \| ( \widetilde{P}_{\mM_{\star}} - P_{M_{\star}}  )M^{-1}_{\star} \nabla f(u^{\star}) \|_{M_{\star}} \| \eta - P_{M_{\star}} \eta  \|_{M_{\star}} \\
& \le \frac{4(\delta^{\star})^2 \alpha^{\star}}{3} \| \nabla f(u^{\star}) \|_{M^{-1}_{\star}} \| \nabla f(u^{\star}_{\phi}) \|_{M^{-1}_{\star}}.
\end{split}
\end{equation}
Putting these estimates into \eqref{lem_u_diff_eq3} and using Young's inequality, we obtain
\begin{equation}
\begin{split}
\label{lem_u_diff_eq5}
\mu_{f,M_{\star}} \| \eta \|^2_{M_{\star}} & \le 
(\delta^{\star}\alpha^{\star} )^2 L^2_{f,M_{\star}} \mu^{-1}_{f,M_{\star}}   \| \nabla f(u^{\star}_{\phi}) \|^2_{M^{-1}_{\star}} 
 + \frac{\mu_{f,M_{\star}}}{4} \| \eta \|^2_{M_{\star}} \\
& + \frac{2(\delta^{\star})^2 \alpha^{\star}}{3} \left( \| \nabla f(u^{\star}_{\phi}) \|^2_{M^{-1}_{\star}} +  \| \nabla f(u^{\star}) \|^2_{M^{-1}_{\star}} \right).
\end{split}
\end{equation}
Notice 
$$
\| \nabla f(u^{\star}_{\phi}) \|_{M^{-1}_{\star}} \le \| \eta_f \|_{M^{-1}_{\star}} + \| \nabla f(u^{\star}) \|_{M^{-1}_{\star}} \le L_{f,M_{\star}} \| \eta \|_{M_{\star}} + \| \nabla f(u^{\star}) \|_{M^{-1}_{\star}}
$$
Using this estimate in \eqref{lem_u_diff_eq5} with $\kappa_{f,M_{\star}} \ge 1$ and $\alpha^{\star} \le L^{-1}_{f,M_{\star}}$, we have 
\begin{equation}
\begin{split}
\label{lem_u_diff_eq6}
\frac{3\mu_{f,M_{\star}}}{4}  \| \eta \|^2_{M_{\star}} &
 \le \frac{ (\delta^{\star})^2 \alpha^{\star}(3\kappa_{f,M_{\star}} + 2 ) }{3}  \| \nabla f(u^{\star}_{\phi}) \|^2_{M^{-1}_{\star}} 
 + \frac{2(\delta^{\star})^2 \alpha^{\star}}{3}  \| \nabla f(u^{\star}) \|^2_{M^{-1}_{\star}} \\
& \le \frac{10(\delta^{\star})^2 \alpha^{\star} \kappa_{f,M_{\star}} L^2_{f,M_{\star}}}{3} \| \eta \|^2_{M_{\star}} 
+ 4( \delta^{\star})^2 \alpha^{\star} \kappa_{f,M_{\star}}\| \nabla f(u^{\star}) \|^2_{M^{-1}_{\star}}.
\end{split}
\end{equation}
We conclude \eqref{lem_u_diff_eq01} from \eqref{lem_u_diff_eq6}
with $10(\delta^{\star})^2 \alpha^{\star} \kappa_{f,M_{\star}} L^2_{f,M_{\star}}/3 \le 5\mu_{f,M_{\star}}/24 \le \mu_{f,M_{\star}}/4$ 
by $\delta^{\star} \le (4\kappa_{f,M_{\star}})^{-1}$ and $\alpha^{\star} \le L^{-1}_{f,M_{\star}}$.
At last, \RG{\eqref{lem_u_diff_eq02} follows from} \eqref{eq: Breg bound} and \eqref{conv_Lip}, which yields \eqref{lem_u_diff_eq03}.
\end{proof}


In the next two sections, we shall proceed to prove the convergence of the IPPGD method \eqref{PPGD0}.
As $u^{\star}_{\phi}$ relies on the final step size $\alpha^{\star}$, it is reasonable that $\alpha$ cannot keep oscillating to the end.
In fact, our proof can handle the case of variable step size by assuming $\alpha$ exponentially converging to $\alpha^{\star}$;
namely, for some positive constants $r_1$ and $r_2$, there holds
\begin{equation*}
\label{asump_lam_eq1}
|\alpha(t) - \alpha^{\star}| \le r_1e^{-r_2 t}.
\end{equation*}
However, to facilitate the ease of exposition but without loss of generality, 
we only consider the fixed step size in the subsequent convergence analysis.


\section{The convergence analysis at the continuous level}
\label{sec:conv_contin}

In this section, we address the exponential convergence at the continuous level.




\subsection{Assumptions on the metric sequence}
\label{assum_metric}


The following assumptions are introduced to establish the behavior of the metric sequence and the impact of inexactness levels. 

\begin{asump}
\label{asump_M}
Given a time-dependent sequence of metrics $\mM(t) = \{ M(t), \widetilde{S}(t) \}$, assume:
\begin{itemize}
\item[\textbf{(H1)}] \label{asump_M_h1} 
\RG{There exists a metric set $\mM_{\star} = \{ M_{\star}, \widetilde{S}_{\star} \}$ and upper-bounded functions $\Theta(t)$ and $\widetilde{\Theta}(t)$,
i.e., $\Theta(t) \in [0,\theta]$ and $\widetilde{\Theta}(t) \in [0,\tilde{\theta}]$, $\forall t\in [0,\infty)$, for positive constants $\theta$ and $\tilde{\theta}$, such that }
\begin{subequations}
\label{asump_M_eq2}
\begin{align}
 & M(t) - M_{\star} \preccurlyeq  \Theta(t) M_{\star}, ~~~~~~ M_{\star} - M(t) \preccurlyeq  \Theta(t) M(t), ~~~ \forall t\in [0,\infty], \label{asump_M_eq21} \\
& \widetilde{S}^{-1}(t) - \widetilde{S}^{-1}_{\star} \preccurlyeq \widetilde{\Theta}(t) \widetilde{S}^{-1}_{\star}, ~~~ 
\widetilde{S}^{-1}_{\star} - \widetilde{S}^{-1}(t) \preccurlyeq \widetilde{\Theta}(t) \widetilde{S}^{-1}(t), ~~~~  \forall t\in [0,\infty].\label{asump_M_eq22}
\end{align} 
\end{subequations}
Denote $\Theta_m(t): = \max\{ \Theta(t), \widetilde{\Theta}(t) \}$ and $\theta_m = \max\{ \theta, \tilde{\theta} \}$.
\item[\textbf{(H2)}] \label{asump_M_h2} 
There is a time-dependent sequence $\delta(t)$ to describe the inexactness level with a uniform upper bound $\delta_{\max}$, i.e., $\delta(t)\le \delta_{\max}$, $\forall t\ge 0$, such that
\begin{equation}
\label{asump_M_eq3}
(1-\delta(t)) \widetilde{S}(t) \preccurlyeq S(t) \preccurlyeq \widetilde{S}(t), ~~~ \forall t \in[0,\infty],
\end{equation}
where $t=\infty$ corresponds to the case of $\widetilde{S}_{\star}$, $S_{\star}$, and $\delta^{\star}$.
\item[\textbf{(H3)}] \label{asump_M_h2_new}  
There exists a uniform constant $K_S$ independent of $t$ such that
\begin{equation}
\label{asump_M_eq4}
-K_S \Theta_m(t) \delta_{\max} S^{-1}_{\star} \preccurlyeq (\widetilde{S}^{-1}-S^{-1}) - (\widetilde{S}^{-1}_{\star} - S^{-1}_{\star})   \preccurlyeq K_S \Theta_m(t) \delta_{\max} S^{-1}_{\star}.
\end{equation}
\item[\textbf{(H4)}] \label{asump_M_h3} 
Let $u^{\star}_{\phi}$ be the fixed point of $\phi$ in \eqref{phi_rhs}. 
\RG{The} sequence $\Theta(t)$ in \hyperref[asump_M_h1]{\textbf{(H1)}} is assumed to satisfy
\begin{equation}
\label{Thetat_bound}
\RG{\Theta_m(t) \le K_{\theta} \sqrt{ D_f(u(t),u^{\star}_{\phi}) },}
\end{equation}
where $K_{\theta}$ is a uniform constant independent of $t$.
\end{itemize}
\end{asump}

\begin{remark}
\label{rem_assump}
\RGI{
We shall see that the convergence analysis below highly relies on these assumptions.
It is still unclear whether the convergence still holds without such assumptions.
Indeed, their verification is problem-dependent; see Section \ref{sec:num} for some examples.
Here, let us mention several notable remarks regarding these assumptions.}
\begin{itemize}

\item
\RGI{
Assumptions \hyperref[asump_M_h1]{\textbf{(H1)}} and \hyperref[asump_M_h3]{\textbf{(H4)}} together appear slightly restrictive,
as they yield
\begin{equation}
  \label{eq_Mu}
M(t) - M_{\star} \preccurlyeq   K_{\theta} \sqrt{ D_f(u(t),u^{\star}_{\phi}) } M_{\star}.
\end{equation}
This inequality trivially shows that if $u$ converges, then $M$ converges. 
The non-trivial aspect, however, is that it cannot directly imply the convergence of $u$ itself. 
The convergence of $u$ certainly represents one of the main challenges in analysis.
In addition, Assumption \hyperref[asump_M_h2]{\textbf{(H2)}} is standard in the sense that it just describes the preconditioning effectiveness of $\widetilde{S}$ and $S$.
Moreover, Assumption \hyperref[asump_M_h3]{\textbf{(H3)}} is a technical Lipschitz-type condition on the difference between
the inverse approximate and exact Schur complements. 
Its verification is also the most difficult one for the PDE case and relies on certain critical properties of specific preconditioners,
which is discussed in Appendix \ref{append_assump_verify}.
}


\item Notice that these conditions \textbf{DO NOT} require $M$ and $\widetilde{M}$ to converge to $M_{\star}$ and $\widetilde{M}_{\star}$,
i.e., $\Theta(t)$ and $\widetilde{\Theta}(t)$ are not assumed to converge to zero.
Notice that we \textbf{DO NOT} assume $\lim_{t\rightarrow \infty} \delta(t) = \delta^{\star}$.
In fact, we do not assume any continuity for $\delta$.

\item All these assumptions are scale-invariant; 
namely all the constants in those inequalities stay unchanged if the $\{\mM(t)\}$ 
is replaced by $\{\beta \mM(t)\}$ with a scaling factor $\beta$. 

\item Constructing $\widetilde{S}(t)$ should rely on $u(t)$ in practice, i.e., $\widetilde{S}(t) = \widetilde{S}(u(t))$. 
Thus, Assumption \hyperref[asump_M_h3]{\textbf{(H3)}} actually states Lipschitz continuity of $(\widetilde{S}^{-1}-S^{-1})$ in a certain sense.

\item \RG{All the estimates are made carefully so that the key constants $\theta_m$, $\delta$, $K_S$ and $K_{\theta}$ 
will appear explicitly in the main Theorems \ref{thm_lyapunov_disc}-\ref{thm_str_cont_lya2_disc}, indicating 
how they affect the convergence rate.}
 
\end{itemize}
\end{remark}

With Assumption \ref{asump_M}, we prepare the following results.
\begin{lemma}
\label{lem_spect}
Under \hyperref[asump_M_h1]{\textbf{(H1)}} in Assumption \ref{asump_M}, there holds
\begin{subequations}
\label{lem_spect_eq0}
\begin{align}
    &  \lambda(M(t) M_{\star}^{-1}) \in \left[ (1+\Theta(t))^{-1}, (1+\Theta(t)) \right],  ~~~~\lambda(\widetilde{S}(t) \widetilde{S}_{\star}^{-1})   \in \left[ (1+\widetilde{\Theta}(t))^{-1}, (1+\widetilde{\Theta}(t)) \right], \label{lem_spect_eq02}  \\
   &   M(t) \preccurlyeq  (1+ \theta) M_{\star}, ~~~~~~~~\widetilde{S}(t) \preccurlyeq  (1+ \tilde{\theta}) \widetilde{S}_{\star},
   ~~~~~~~~ \forall t \ge 0. 
   \label{lem_pi_est_eq01} 
 \end{align}
\end{subequations}
Further assume $ {S}(t) \preccurlyeq \widetilde{S}(t)$, $\forall t \ge 0$, then 
\begin{equation}
\label{lem_pi_est_eq02}
\| (\widetilde{P}_{\mM(t)} - \widetilde{P}_{\mM_{\star}})u  \|_{M(t)} \le 2\sqrt{1+\theta_m} \Theta_m(t) \| (I-\widetilde{P}_{\mM_{\star}}) u \|_{M_{\star}}.
\end{equation}
Under Assumption \hyperref[asump_M_h2_new]{\textbf{(H3)}}, there holds
\begin{equation}
\label{lem_pi_est_eq03}
\| \left[ (\widetilde{P}_{\mM(t)} - P_{M(t)}) - (\widetilde{P}_{\mM_{\star}} - P_{M_{\star}}) \right]u \|_{M(t)} \le 2\sqrt{1+\theta_m} K_S \delta_{\max} \Theta_m(t) \| u \|_{M_{\star}}.
\end{equation}
\end{lemma}
\begin{proof}
The first one in \eqref{lem_spect_eq02} follows from \eqref{asump_M_eq21} and \eqref{spd_inequa}.
The second one follows from a similar argument, and \eqref{lem_pi_est_eq01} is trivial.

We then proceed to estimate \eqref{lem_pi_est_eq02}. 
To simplify the notations, we shall ignore the dependence of those quantities on $t$.
We write down
\begin{equation}
 \label{lem_pi_est_eq1}
 \begin{split}
\widetilde{P}_{\mM} -  \widetilde{P}_{\mM_{\star}}  =  -\underbrace{M^{-1} (B^T \widetilde{S}^{-1} B - B^T \widetilde{S}_{\star}^{-1} B) }_{=: R_1} 
- \underbrace{ (M^{-1} - M_{\star}^{-1}) B^T \widetilde{S}_{\star}^{-1} B }_{=:R_2}.
\end{split}
\end{equation}
We first estimate $R_1$ above. 
Using Lemma \ref{lem_SPD_general}, we obtain
\begin{equation}
\begin{split}
\label{lem_pi_est_eq2}
R^T_1 M R_1 & =  B^T( \widetilde{S}^{-1} - \widetilde{S}_{\star}^{-1} ) \widetilde{S} ( \widetilde{S}^{-1} - \widetilde{S}_{\star}^{-1} ) B 
\preccurlyeq  \widetilde{\Theta}^2 B^T \widetilde{S}^{-1} B \preccurlyeq  (1+\tilde{\theta}) \widetilde{\Theta}^2 B^T \widetilde{S}^{-1}_{\star} B,
\end{split}
\end{equation}
which gives the estimates of $R_1$.
As for $R_2$ in \eqref{lem_pi_est_eq1}, by Lemma \ref{lem_SPD_general} with \eqref{lem_spect_eq02}, we have
\begin{equation}
\label{lem_pi_est_eq5}
(M^{-1} - M_{\star}^{-1}) M (M^{-1} - M_{\star}^{-1}) \preccurlyeq  \Theta^2  M^{-1}. 
\end{equation}
Then, due to \eqref{lem_pi_est_eq01} and $S_{\star} \preccurlyeq  \widetilde{S}_{\star}$, we obtain
\begin{equation}
\begin{split}
\label{lem_pi_est_eq6}
R^T_2 M R_2 
& \preccurlyeq \Theta^2  B^T \widetilde{S}_{\star}^{-1} B  M^{-1} B^T \widetilde{S}_{\star}^{-1} B
=  \Theta^2  B^T \widetilde{S}_{\star}^{-1} S \widetilde{S}_{\star}^{-1} B 
 \preccurlyeq   \Theta^2  (1+\theta) B^T \widetilde{S}_{\star}^{-1}  B .
\end{split}
\end{equation}
Combining \eqref{lem_pi_est_eq2} and \eqref{lem_pi_est_eq6} finishes the proof.

At last, for \eqref{lem_pi_est_eq03} we notice 
\begin{equation}
\begin{split}
\label{lem_pi_est_eq7}
 \left[ (\widetilde{P}_{\mM} - P_{M}) - (\widetilde{P}_{\mM_{\star}} - P_{M_{\star}}) \right] & =  - (M^{-1} - M^{-1}_{\star})B^T(\widetilde{S}^{-1}-S^{-1})B \\
 - M^{-1}_{\star}B^T & \left( (\widetilde{S}^{-1}-S^{-1}) - (\widetilde{S}^{-1}_{\star} - S^{-1}_{\star}) \right) B 
 =: - R_3 - R_4.
 \end{split}
\end{equation}
Using Lemma \ref{lem_SPD_general} with a similar argument to \eqref{lem_pi_est_eq5}, we have
\begin{equation}
\label{lem_pi_est_eq8}
R^T_3 M R_3 \preccurlyeq \Theta^2 B^T(\widetilde{S}^{-1}-S^{-1}) S (\widetilde{S}^{-1}-S^{-1})B 
 \preccurlyeq \delta^2 \Theta^2 M  \preccurlyeq (1+\theta) \delta^2 \Theta^2 M_{\star}.
\end{equation}
In addition, using Assumption \hyperref[asump_M_h2_new]{\textbf{(H3)}} with a similar argument to Lemma \ref{lem_SPD_general}, 
we obtain 
\begin{equation}
\begin{split}
\label{lem_pi_est_eq9}
R^T_4 M R_4  & \preccurlyeq (1+\theta) B^T \left( (\widetilde{S}^{-1}-S^{-1}) - (\widetilde{S}^{-1}_{\star} - S^{-1}_{\star}) \right) S_{\star} \left( (\widetilde{S}^{-1}-S^{-1}) - (\widetilde{S}^{-1}_{\star} - S^{-1}_{\star}) \right) B \\
& \preccurlyeq (1+\theta) K_S^2 \delta^2_{\max} \Theta^2_m B^T S_{\star}^{-1} B  \preccurlyeq (1+\theta) K_S^2 \delta^2_{\max} \Theta^2_m M_{\star}.
\end{split}
\end{equation}
\end{proof}

\subsection{Lyapunov analysis}
\label{subsec:main_results}

To facilitate the discussion, we also introduce the following notation
\begin{equation}
\begin{split}
\label{xi_notat}
& \xi(t) = u(t) - \alpha M^{-1} \nabla f(u(t)), 
 ~~~~ \xi^{\star} = u^{\star}_{\phi} - \alpha M^{-1}_{\star} \nabla f(u^{\star}_{\phi}) \\
& \zeta(t) = u(t) - u^{\star}_{\phi}, 
~~~~~ \zeta_f(t) = \nabla f(u(t)) - \nabla f(u^{\star}_{\phi}) 
\end{split}
\end{equation}
which will be frequently used.
Notice that $\xi(t) \rightarrow \xi^{\star}$ and $ \zeta(t) \rightarrow 0$ if $u(t) \rightarrow u^{\star}_{\phi}$. 
In the following discussion, for simplicity, 
we shall drop ``$(t)$" if there is no danger of causing confusion.

Let us recall the following trivial result: 
given two linear symmetric operators $Q$ and $R$ satisfying $c_1 Q \preccurlyeq  R \preccurlyeq c_2 Q$, there holds
\begin{equation} 
\label{Lmu_M_ineq}
L_{f,Q} \le c_2 L_{f,R}, ~~~~  \mu_{f,R}  \le c_1^{-1} \mu_{f,Q}, ~~~~ \kappa_{f,Q} \le c_2/c_1 \kappa_{f,R}.
\end{equation}
Assumption  \hyperref[asump_M_h1]{\textbf{(H1)}} with \eqref{Lmu_M_ineq} enables us to unify the potential metrics to be $M_{\star}$ up to a constant depending only on $\theta_m$:
\begin{equation}
\label{Lmu_M_ineq2}
\mu_{f,M} \ge (1+\theta_m)^{-1} \mu_{f,M_{\star}} ,~~~ L_{f,M}\le (1+\theta_m) L_{f,M_{\star}}, ~~~  \kappa_{f,M(t)} \le (1+\theta_m)^2 \kappa_{f,M_{\star}}.
\end{equation}


With these preparations, 
we then present some useful estimates.
\begin{lemma}
\label{lem_uxi}
Under \hyperref[asump_M_h1]{\textbf{(H1)}}, \hyperref[asump_M_h2]{\textbf{(H2)}} and \hyperref[asump_M_h2_new]{\textbf{(H3)}} in Assumption \ref{asump_M}, 
and $\alpha^{\star} \le L_{f,M_{\star}}^{-1}$ and $\delta^{\star}< (4\kappa_{f,M_{\star}})^{-1}$,
there hold
\begin{subequations}
\label{lem_uxi_eq0}
\begin{align}
  &    \| u^{\star}_{\phi} - \widetilde{P}_{\mM} ( u^{\star}_{\phi} - \alpha M^{-1}(t) \nabla f(u^{\star}_\phi) ) \|_{M(t)}  
   \le  \sqrt{1+\theta_m} \alpha p  \| \nabla f(u^{\star}) \|_{M^{-1}_{\star}} \Theta_m , \label{lem_uxi_eq01}  \\
  & \| (  \widetilde{P}_{\mM_{\star}} M_{\star}^{-1} -   \widetilde{P}_{\mM} M^{-1} ) \nabla f(u^{\star}_\phi)  \|_{M_{\star}} \le  
 \sqrt{1+\theta_m}  p  \| \nabla f(u^{\star}) \|_{M^{-1}_{\star}} \Theta_m;   \label{lem_uxi_eq03}
\end{align}
where the function $p$ is given by
\begin{align}
\label{lem_uxi_eq04}
   p(\delta_{\max},\delta^{\star};\kappa_{f,M_{\star}},K_S) := (9\kappa_{f,M_{\star}}+4)\delta^{\star} 
   + (1 + 2K_S) \delta_{\max} .
\end{align}
\end{subequations}
\end{lemma}
\begin{proof}
For simplicity, we drop ``$(t)$''. Note that $u^{\star}_\phi = \widetilde{P}_{\mM_{\star}} (u^{\star}_{\phi} - \alpha M_{\star}^{-1} \nabla f(u^{\star}_\phi)) $. 
Then, we obtain
\begin{equation}
\begin{split}
\label{lem_uxi_eq1}
  u^{\star}_{\phi} - \widetilde{P}_{\mM} ( u^{\star}_{\phi} - \alpha M^{-1} \nabla f(u^{\star}_\phi) ) 
 =  \underbrace{ (\widetilde{P}_{\mM_{\star}} - \widetilde{P}_{\mM})u^{\star}_\phi }_{R_1} -   \alpha \underbrace{ (  \widetilde{P}_{\mM_{\star}} M_{\star}^{-1} -   \widetilde{P}_{\mM} M^{-1} ) \nabla f(u^{\star}_\phi)  }_{R_2}. \\
 \end{split}
\end{equation}
For $R_1$, as $\widetilde{P}_{\mM_{\star}} u^{\star} = \widetilde{P}_{\mM} u^{\star}= u^{\star}$, 
using \eqref{lem_u_diff_eq00} in Lemma \ref{lem_u_diff} and \eqref{lem_pi_est_eq02} in Lemma \ref{lem_spect}, we obtain
\begin{equation*}
\begin{split}
\label{lem_uxi_eq2_new}
\| R_1 \|_{M}  &= \| (\widetilde{P}_{\mM_{\star}} - \widetilde{P}_{\mM})( u^{\star}_\phi - u^{\star} ) \|_{M} \\
& \le 2 \sqrt{1+\theta_m} \Theta_m \| ( I - \widetilde{P}_{\mM_{\star}} )( u^{\star}_\phi - u^{\star} ) \|_{M_{\star}} 
  \le 4 \alpha \delta^{\star} \sqrt{1+\theta_m} \Theta_m  \| \nabla f(u^{\star}) \|_{M^{-1}_{\star}}.
\end{split}
\end{equation*}
For $R_2$, we notice the following decomposition
\begin{equation*}
\begin{split}
R_2  =  \underbrace{  ( \widetilde{P}_{\mM_{\star}} M_{\star}^{-1} -  \widetilde{P}_{\mM}  M^{-1} ) ( \nabla f(u^{\star}_\phi) - \nabla f(u^{\star}))  }_{R_{21}}  
 +  \underbrace{  ( \widetilde{P}_{\mM_{\star}} M_{\star}^{-1} -  \widetilde{P}_{\mM}  M^{-1} )  \nabla f(u^{\star})  }_{R_{22}} .
\end{split}
\end{equation*}
We then need to estimate each term above.
For $R_{21}$, noticing the decomposition 
$$
\widetilde{P}_{\mM_{\star}} M_{\star}^{-1} -  \widetilde{P}_{\mM}  M^{-1} = (\widetilde{P}_{\mM_{\star}} -  \widetilde{P}_{\mM} )  M_{\star}^{-1}  + \widetilde{P}_{\mM}  ( M_{\star}^{-1} - M^{-1} ),
$$ 
and using \eqref{lem_pi_est_eq02} in Lemma \ref{lem_spect}, 
\eqref{lem_Pi_est_eq02} and \eqref{lem_Pi_est_eq011} in Lemma \ref{lem_Pi_est},
the similar argument to \eqref{lem_pi_est_eq5},
we have
\begin{equation*}
\begin{split}
\label{lem_uxi_eq3_2}
\| R_{21} \|_M   \le & \| (\widetilde{P}_{\mM_{\star}} -  \widetilde{P}_{\mM} )  M_{\star}^{-1} ( \nabla f(u^{\star}_\phi) - \nabla f(u^{\star}))    \|_{M }  +  \|  \widetilde{P}_{\mM}  ( M_{\star}^{-1} - M^{-1} ) ( \nabla f(u^{\star}_\phi) - \nabla f(u^{\star}))   \|_{M }  \\
 \le & 3  \sqrt{1+\theta_m} \Theta_m  \|   \nabla f(u^{\star}_\phi) - \nabla f(u^{\star})   \|_{M_{\star}^{-1} } 
 \le 9 \sqrt{1+\theta_m} \kappa_{f,M_{\star}}  \Theta_m \delta^{\star}  \|  \nabla f(u^{\star})  \|_{M_{\star}^{-1} } ,
\end{split}
\end{equation*}
where in the last inequality we have also used \eqref{lem_u_diff_eq01} in Lemma \ref{lem_u_diff} with $\alpha^{\star} \le  L_{f,M_{\star}}^{-1}$.
Furthermore, for $R_{22}$, as $P_M M^{-1} \nabla f(u^{\star}) = 0$ for any SPD $M$, we can write down
\begin{equation}
\begin{split}
\label{lem_uxi_eq3_3}
 R_{22}   = &  (\widetilde{P}_{\mM_{\star}}  - P_{M_{\star}}) M_{\star}^{-1} \nabla f(u^{\star})  - ( \widetilde{P}_{\mM} - P_M ) M^{-1} \nabla f(u^{\star})\\
  = & \underbrace{ \left[ (\widetilde{P}_{\mM_{\star}}  - P_{M_{\star}}) -  (  \widetilde{P}_{\mM} - P_M ) \right] M_{\star}^{-1} \nabla f(u^{\star}) }_{R_{221}}
  - \underbrace{ ( \widetilde{P}_{\mM} - P_M )( M^{-1} - M_{\star}^{-1} ) \nabla f(u^{\star}) }_{R_{222}}.
\end{split}
\end{equation}
Applying \eqref{lem_pi_est_eq03} in Lemma \ref{lem_spect}, we have
\begin{equation*}
\label{lem_uxi_eq3_4}
\| R_{221} \|_M \le 2\sqrt{1+\theta_m} K_S \delta \Theta_m \| \nabla f(u^{\star}) \|_{M^{-1}_{\star}}.
\end{equation*}
In addition, employing \eqref{lem_Pi_est_eq04} in Lemma \ref{lem_SPD_general}, 
\eqref{lem_pi_est_eq01} in Lemma \ref{lem_spect}, and \eqref{lem_pi_est_eq5} yields
\begin{equation*}
\label{lem_uxi_eq3_5}
\| R_{222} \|_M \le \delta \| ( M^{-1} - M_{\star}^{-1} ) \nabla f(u^{\star}) \|_M \le \sqrt{1+\theta_m} \delta \Theta_m  \| \nabla f(u^{\star}) \|_{M^{-1}_{\star}}.
\end{equation*}
Substituting these estimates into \eqref{lem_uxi_eq3_3}, we have the estimate of $R_{22}$.
It then leads to $R_2$ together with the estimate of $R_{21}$ and $\kappa_{f,M_{\star}}\ge1$.
Notice that $R_2$ readily gives \eqref{lem_uxi_eq03}.
\end{proof}

In the next two lemmas, we shall proceed to establish the dynamics for $\mathcal{E}^{(1)}$ and $\mathcal{E}^{(2)}$, respectively.
\begin{lemma}[Dynamics for $\mathcal{E}^{(1)}$]
\label{lemma_error_ineq}
Under the conditions of Lemma \ref{lem_uxi},
there holds
\begin{equation}
\label{lemma_error_ineq_eq0}
\frac{\dd }{\dd t}\mathcal{E}^{(1)}  \le  - \frac{ \mu_{f,M} \kappa^{-1}_{f,M}}{2} \alpha \mathcal{E}^{(1)} 
+ K_1 \alpha^{-1} \mathcal{E}^{(2)} + K_2 \alpha  p^2 \Theta^2_m,
\end{equation}
where $K_1 =  2(2\kappa^{2}_{f,M} + \alpha^2 L^2_{f,M})$, 
$K_2 = 2 (1+\theta_m)  \| \nabla f(u^{\star}) \|^2_{M^{-1}_{\star}} \kappa^2_{f,M}$,
and $p$ is given in \eqref{lem_uxi_eq04}.
\end{lemma}
\begin{proof}
To begin with, we notice the following identity:
\begin{equation}
\begin{split}
\label{lemma_error_ineq_eq1}
\frac{\dd }{\dd t} \mathcal{E}^{(1)} = & - \langle \nabla f(u) - \nabla f(u^{\star}_{\phi}), \zeta - \widetilde{P}_{\mM} \zeta \rangle  - \alpha  \langle \nabla f(u) - \nabla f(u^{\star}_{\phi}), \widetilde{P}_{\mM} M^{-1} (\nabla f(u) - \nabla f(u^{\star}_{\phi})) \rangle \\
& -  \langle \nabla f(u) - \nabla f(u^{\star}_{\phi}), u^{\star}_{\phi} - \widetilde{P}_{\mM} ( u^{\star}_{\phi} - \alpha M^{-1} \nabla f(u^{\star}_\phi) ) \rangle : =  R_1 + R_2 + R_3.
\end{split}
\end{equation}
The estimate of $R_1$ follows from a simple Young's inequality:
\begin{equation*}
\begin{split}
\label{lemma_error_ineq_eq2} 
R_1 & \le \| \nabla f(u) - \nabla f(u^{\star}_{\phi}) \|_{M^{-1}} \| \zeta - \widetilde{P}_{\mM} \zeta \|_M 
 \le (2L_{f,M})^{1/2} \mathcal{E}^{1/2}  \| \zeta - \widetilde{P}_{\mM} \zeta \|_M \\
 & \le  \frac{1}{4} \alpha \mu_{f,M} \kappa^{-1}_{f,M} \mathcal{E}^{(1)} + 2 \alpha^{-1}  \kappa^{2}_{f,M} \| \zeta - \widetilde{P}_{\mM} \zeta \|^2_M.
\end{split}
\end{equation*}
As for $R_2$, by Lemma \ref{lem_convLip3} and $\| \zeta \|^2_M \ge 2 L^{-1}_{f,M} \mathcal{E}$ from \eqref{conv_Lip}, we have
\begin{equation*}
\label{lemma_error_ineq_eq3}
R_2  \le   -\alpha \left( \mu^2_{f,M} /2  \| \zeta \|^2_M - L^2_{f,M} \| \widetilde{P}_{\mM}\zeta - \zeta \|^2_M \right) 
\le  -\alpha \mu_{f,M} \kappa^{-1}_{f,M} \mathcal{E}^{(1)} + \alpha L^2_{f,M}  \| \zeta - \widetilde{P}_{\mM} \zeta \|^2_M.
\end{equation*}
At last, for $R_3$, by \eqref{lem_uxi_eq01} in Lemma \ref{lem_uxi} with Young's inequality, we have
\begin{equation*}
\begin{split}
\label{lemma_error_ineq_eq4}
R_3  & \le (2L_{f,M})^{1/2} (\mathcal{E}^{(1)})^{1/2}  \left( \sqrt{1+\theta_m} p  \alpha \| \nabla f(u^{\star}) \|_{M^{-1}_{\star}} \Theta_m \right) \\
& \le \frac{1}{4} \alpha \mu_{f,M}  \kappa_{f,M}^{-1} \mathcal{E}^{(1)} 
+ 2 (1+\theta_m)  \| \nabla f(u^{\star}) \|^2_{M^{-1}_{\star}} \kappa^2_{f,M} \alpha  p^2 \Theta^2_m.
\end{split}
\end{equation*}
Combining these estimates into \eqref{lemma_error_ineq_eq1} yields \eqref{lemma_error_ineq_eq0}.
\end{proof}

\begin{lemma}[Dynamics for $\mathcal{E}^{(2)}$]
\label{lem_est_uPiu}
Under the conditions of Lemma \ref{lem_uxi} and the extra assumption $\delta_{\max} = \sup_{t\ge 0}\{\delta(t)\} \le (8\theta_m+9)^{-1}$, 
there holds for any $\epsilon\ge 0$
\begin{equation}
\begin{split}
\label{lem_est_uPiu_eq02}
\frac{\dd}{\dd t}\mathcal{E}^{(2)}
& \le - \left( \frac{7}{4} - \epsilon \right) \mathcal{E}^{(2)}
+ K_3 \epsilon^{-1}  (\alpha \delta)^2  \mathcal{E}^{(1)}
+ K_4  \left(16(\delta^{\star})^2 +  p^2 \right) \alpha^2 \Theta^2_m ,
\end{split}
\end{equation}
where $K_3=4 (1+\theta_m)^2 L_{f,M}$, $K_{4}=\frac{3}{2}(1+\theta_m) \| \nabla f(u^{\star}) \|^2_{M^{-1}_{\star}}$,
and $p$ is given by \eqref{lem_uxi_eq04}.
\end{lemma}
\begin{proof}
Using the notation in \eqref{xi_notat}, we can write down
\begin{equation*}
\label{lem_est_uPiu_eq1}
\frac{\dd}{\dd t}\zeta  = -(I - \widetilde{P}_{\mM})\zeta  + (\widetilde{P}_{\mM}-\widetilde{P}_{\mM_{\star}})u^{\star}_{\phi} 
+  \alpha\widetilde{P}_{\mM}M^{-1}\zeta_f + \alpha (\widetilde{P}_{\mM}M^{-1} - \widetilde{P}_{\mM_{\star}}M^{-1}_{\star}) \nabla f(u^{\star}_{\phi}) .
\end{equation*}
We then get
\begin{equation}
\begin{split}
\label{lem_est_uPiu_eq2}
\frac{\dd}{\dd t}\mathcal{E}_2  = & \langle (I - \widetilde{P}_{\mM_{\star}})\zeta, (I - \widetilde{P}_{\mM_{\star}})\zeta'  \rangle_{M_{\star}} \\
 = & -2 \mathcal{E}^{(2)} + \langle (I - \widetilde{P}_{\mM_{\star}})\zeta,  (I - \widetilde{P}_{\mM_{\star}}) \widetilde{P}_{\mM}\zeta \rangle_{M_{\star}}
  +  \langle (I - \widetilde{P}_{\mM_{\star}})\zeta, (\widetilde{P}_{\mM}-\widetilde{P}_{\mM_{\star}})u^{\star}_{\phi}   \rangle_{M_{\star}}  \\
 + &\alpha \langle (I - \widetilde{P}_{\mM_{\star}})\zeta,  (I - \widetilde{P}_{\mM_{\star}})\widetilde{P}_{\mM} M^{-1} \zeta_f   \rangle_{M_{\star}} \\
 + & \alpha  \langle (I - \widetilde{P}_{\mM_{\star}})\zeta,  (I - \widetilde{P}_{\mM_{\star}}) (\widetilde{P}_{\mM}M^{-1} - \widetilde{P}_{\mM_{\star}}M^{-1}_{\star}) \nabla f(u^{\star}_{\phi})  \rangle_{M_{\star}}.  
\end{split}
\end{equation}
We denote $R_1$-$R_4$ by the second to fifth terms above and proceed to estimate each one.
First, by \eqref{lem_Dproj_new_eq01} in Lemma \ref{lem_Pi_est}, we have
\begin{equation*}
\begin{split}
\label{lem_est_uPiu_eq3}
& R_1 \le \frac{ 2(1+\theta_m)\delta}{1-\delta^{\star}}   \mathcal{E}^{(2)} \le \frac{1}{4} \mathcal{E}^{(2)}
~~~\text{and}~~~\\
& R_3 \le  \frac{ (1+\theta_m)\alpha\delta}{1-\delta} \sqrt{2\mathcal{E}^{(2)}} \| \zeta_{f} \|_{M_{\star}}  
\le \frac{\epsilon}{3} \mathcal{E}^{(2)} + 4\epsilon^{-1} (1+\theta_m)^2(\alpha \delta)^2 L_{f,M_{\star}} \mathcal{E}^{(1)},
\end{split}
\end{equation*}
where the first one follows from $\delta \le \frac{1}{8\theta_m +9}$ 
and the second follows from $\delta \le \frac{1}{9} \le \frac{1}{8\theta_m +9}$.
It yields the term associated with $K_3$.
As for $R_2$, by \eqref{lem_pi_est_eq02} in Lemma \ref{lem_spect}, $\widetilde{P}_{\mM}u^{\star} = u^{\star}$, and \eqref{lem_u_diff_eq00} in Lemma \ref{lem_u_diff}, we have
\begin{equation*}
\begin{split}
& \| (\widetilde{P}_{\mM} - \widetilde{P}_{\mM_{\star}} ) u^{\star}_{\phi} \|_{M_{\star}} 
 \le 2\sqrt{1+\theta_m} \Theta_m  \| (I - \widetilde{P}_{\mM_{\star}} ) u^{\star}_{\phi} \|_{M_{\star}} \\
 = &2\sqrt{1+\theta_m} \Theta_m  \| (I - \widetilde{P}_{\mM})(u^{\star}_{\phi} - u^{\star}) \|_{M(t)}  
\le 4\sqrt{1+\theta_m}\alpha \delta^{\star} \Theta_m  \| \nabla f(u^{\star}) \|_{M_{\star}^{-1}} .
\end{split}
\end{equation*}
Then, with Young's inequality, we conclude
\begin{equation*}
\label{lem_est_uPiu_eq5}
R_2 \le 4\sqrt{1+\theta_m}\alpha \delta^{\star} \sqrt{2\mathcal{E}^{(2)}} \Theta_m  \| \nabla f(u^{\star}) \|_{M_{\star}^{-1}}  
\le \frac{\epsilon}{3} \mathcal{E}^{(2)} + 24(1+\theta_m)(\alpha \delta^{\star})^2  \| \nabla f(u^{\star}) \|^2_{M_{\star}^{-1}} \Theta^2_m.
\end{equation*}
Next, using \eqref{lem_uxi_eq03} in Lemma \ref{lem_uxi}, we achieve
\begin{equation*}
\label{lem_est_uPiu_eq6}
R_4 \le \sqrt{1+\theta_m} \alpha  p  \sqrt{2\mathcal{E}^{(2)}} \Theta_m \| \nabla f(u^{\star}) \|_{M_{\star}^{-1}}
\le  \frac{\epsilon}{3} \mathcal{E}^{(2)} + \frac{3}{2}(1+\theta_m)(\alpha p)^2  \| \nabla f(u^{\star}) \|^2_{M_{\star}^{-1}} \Theta^2_m.
\end{equation*}
Combining these estimates into \eqref{lem_est_uPiu_eq2} yields \eqref{lem_est_uPiu_eq02}.
\end{proof}

\begin{theorem}[\RGI{Strong Lyapunov property}]
\label{thm_str_cont_lya}
Under Assumption \ref{asump_M},
assume $\alpha \le L_{f,M_{\star}}^{-1}$ and $\delta(t)$ is sufficiently small such that $\forall t \ge 0$
\begin{equation}
\begin{split}
\label{thm_str_cont_lya_eq01}
& \delta \le \min\Big\{ \frac{\sqrt{\lambda}}{4\sqrt{2}(1+\theta_m)\kappa_{f,M} }, ~~ \frac{1}{8\theta_m+9} \Big\}, \\
& \delta^{\star} \le \min\Big\{ \frac{\sqrt{\lambda}}{ 8 \sqrt{6} K_{\theta} (1+\theta_m) \kappa_{f,M}^{1/2} C^{\star} }, ~~ \frac{1}{4 \kappa_{f,M_{\star}}}  \Big\}, \\
& p \le \frac{\min\Big\{ \sqrt{6\lambda}, ~~ \sqrt{2} \kappa_{f,M}^{-1} \Big\}}{8(1+\theta_m) K_{\theta} \kappa_{f,M}^{1/2} C^{\star}},
\end{split}
\end{equation}
\RGI{
where $p = p(\delta_{\max},\delta^{\star};\kappa_{f,M_{\star}},K_S)= (9\kappa_{f,M_{\star}}+4)\delta^{\star} 
   + (1 + 2K_S) \delta_{\max}$ is given in \eqref{lem_uxi_eq04},
   $C^{\star} = \mu^{1/2}_{f,M_{\star}} \| \nabla f(u^{\star}) \|_{M_{\star}^{-1}}$,
   and $\delta^{\star}$, $\delta_{\max}$, $\theta_m$, and $K_S$ are all defined in Assumption \ref{asump_M}.
We further take $\lambda$ to be sufficiently small 
such that $\lambda \le (4K_1)^{-1}$ 
with $K_1$ given in Lemma \ref{lemma_error_ineq}.}
Then, there holds
\begin{equation}
\begin{split}
\label{thm_str_cont_lya_eq0}
\frac{\dd}{\dd t}\mathcal{E}  \le   - \omega  \mathcal{E}~~~~ \text{with} ~~ 
\omega = \min\Big\{ \frac{  \mu_{f,M} \kappa^{-1}_{f,M}\alpha}{8} , \frac{3}{2} \Big\}.
\end{split}
\end{equation}
\end{theorem}
\begin{proof}
Employing Lemmas \ref{lemma_error_ineq} and \ref{lem_est_uPiu}, we have
\begin{equation*}
\begin{split}
\label{thm_str_cont_lya_eq1}
\frac{\dd}{\dd t}\mathcal{E} 
& \le - \left(\frac{ \lambda \mu_{f,M} \kappa^{-1}_{f,M}}{2} - K_3 \epsilon^{-1} \delta^2 - \lambda K_2 K_{\theta}^2 p^2  - K_4 K_{\theta}^2  \left(24(\delta^{\star})^2 +  3p^2/2 \right)  \right) \alpha^2 \mathcal{E}^{(1)} \\
& ~~~~~ - \left( \frac{7}{4} - \epsilon - \lambda K_1 \right) \mathcal{E}^{(2)} .
\end{split}
\end{equation*}
Take $\epsilon = 1/4$.
By the assumptions in \eqref{thm_str_cont_lya_eq01},
we have $4 K_3 \delta^2, ~ \lambda K_2 K_{\theta}^2 p^2, ~ K_4 K_{\theta}^2  \left(24(\delta^{\star})^2 +  3p^2/2 \right) \le \lambda \mu_{f,M} \kappa^{-1}_{f,M}/8$,
which together yield
\begin{equation*}
  \frac{\dd}{\dd t}\mathcal{E} \le - \frac{ \lambda \mu_{f,M} \kappa^{-1}_{f,M}}{8} \alpha^2 \mathcal{E}^{(1)} - \frac{3}{2} \mathcal{E}^{(2)}
\le - \min\Big\{ \frac{  \mu_{f,M} \kappa^{-1}_{f,M}\alpha}{8} , \frac{3}{2} \Big\} \mathcal{E},
\end{equation*}
where we have also used the assumption $\lambda \le (4K_1)^{-1}$.
\end{proof}


\begin{theorem}[Exponential convergence]
\label{thm_str_cont_lya2}
Under the assumptions of Theorem \ref{thm_str_cont_lya}, there holds 
\begin{subequations}
\begin{align}
\mathcal{E}  \le e^{-\int^t_0\omega\dd s} \left( \lambda \alpha \mathcal{E}^{(1)}(0) +   \mathcal{E}^{(2)}(0) \right), 
\label{thm_str_cont_lya2_eq01}
\end{align}
where \RGI{$\omega = \min\Big\{ \frac{  \mu_{f,M} \kappa^{-1}_{f,M}\alpha}{8} , \frac{3}{2} \Big\}$ is given in \eqref{thm_str_cont_lya_eq0}}.
If $u_0 \in \ker(B)$, there further holds 
\begin{align}
\label{thm_str_cont_lya2_eq02}
\mathcal{E}  \le e^{-\int^t_0\omega\dd s}  \left( \lambda \alpha \mathcal{E}^{(1)}(0) 
+ 4 (\alpha\delta^{\star})^2  \| \nabla f (u^{\star}) \|^2_{M^{-1}_{\star}} 
\right). 
\end{align}
\end{subequations}
\end{theorem}
\begin{proof}
Notice that \eqref{thm_str_cont_lya2_eq01} is trivial from the strong Lyapunov property in Theorem \ref{thm_str_cont_lya}.
In addition, \eqref{thm_str_cont_lya2_eq02} follows from Lemma \ref{lem_u_diff} and $(I - \widetilde{P}_{\mM_{\star}})(u_0 - u^{\star}_{\phi}) = (I - \widetilde{P}_{\mM_{\star}})(u^{\star} - u^{\star}_{\phi})$.
\end{proof}

\begin{remark}
  There are several notable remarks for the main theorem above.
\begin{itemize}
\item The ``sufficiently small'' condition in \eqref{thm_str_cont_lya_eq01} only imposes restrictions on $\delta$ and $\delta^{\star}$,
      as $p$ is only a linear function of $\delta$ and $\delta^{\star}$.
\item The error bound in \eqref{thm_str_cont_lya2_eq02} is independent of $\lambda$.
      This is important since $\lambda$ appears in \eqref{thm_str_cont_lya2_eq01} as a denominator whose smallness may slow down the convergence.
      It shows that $\lambda$, though critical for theoretical analysis, does not directly influence the convergence.
\end{itemize}
\end{remark}

\section{The convergence at the discrete level}
\label{sec:conv_disc}

In this section, we present the discrete linear convergence analysis for the IPPGD method in \eqref{PPGD0}.
However, special attention should be paid to ensuring the admissible range for the pseudo step size $\tau_k$.
It should be large enough to include $1$ to recover the classical projection methods.
Let us construct the discrete Lyapunov sequences:
\begin{equation}
  \begin{split}
\label{cont_Lyap_disc}
&\mathcal{E}_k  =  \lambda \alpha \mathcal{E}^{(1)}_{k} + \mathcal{E}^{(2)}_{k}, \\ 
\text{with} ~~ & \mathcal{E}^{(1)}_{k}:=\mathcal{E}^{(1)}(u_k) = D_f(u_k,u^{\star}_\phi)~~ \text{and} ~~ 
\mathcal{E}^{(2)}_{k}:= \mathcal{E}^{(2)}(u_k) = \frac{1}{2}\| (I - \widetilde{P}_{\mM_{\star}})(u_k - u^{\star}_{\phi}) \|^2_{M_{\star}}.
  \end{split}
\end{equation}

We then generalize the assumptions given in Subsection \ref{assum_metric} to the discrete case.
\begin{asump}
\label{asump_M_disc}
Given a time series of metrics $\mM_k$, assume:
\begin{itemize}
\item[\textbf{(H1')}] \label{asump_M_h1_disc} 
There exist $\mM_{\star} = \{ M_{\star}, \widetilde{S}_{\star} \}$ and functions $\Theta_{k} \in [0,\theta]$ and $\widetilde{\Theta}_{k} \in [0,\tilde{\theta}]$ such that 
\begin{subequations}
\label{asump_M_eq2_disc}
\begin{align}
 & M_k - M_{\star} \preccurlyeq  \Theta_{k} M_{\star}, ~~~~~~ M_{\star} - M_k \preccurlyeq  \Theta_{k} M_k, \label{asump_M_eq21_disc} \\
& \widetilde{S}^{-1}_k - \widetilde{S}^{-1}_{\star} \preccurlyeq \widetilde{\Theta}_{k} \widetilde{S}^{-1}_{\star}, ~~~ 
\widetilde{S}^{-1}_{\star} - \widetilde{S}^{-1}_k \preccurlyeq \widetilde{\Theta}_{k} \widetilde{S}^{-1}_k.  \label{asump_M_eq22_disc}
\end{align} 
\end{subequations}
Denote $\Theta_{k,m}: = \max\{ \Theta_k, \widetilde{\Theta}_k \}$ 
and recall $\theta_m = \max\{ \theta, \tilde{\theta} \}$.
\item[\textbf{(H2')}] \label{asump_M_h2_disc} 
There is a time sequence $\delta_k$ known as the inexactness level with a uniform upper bound $\delta_{\max}$, i.e., $\delta_k \le \delta_{\max}$, $\forall k$, such that
\begin{equation}
\label{asump_M_eq3_disc}
(1-\delta_k) \widetilde{S}_k \preccurlyeq S_k \preccurlyeq \widetilde{S}_k, ~~~ \forall k \ge 0.
\end{equation}
\item[\textbf{(H3')}] \label{asump_M_h2_new_disc} Under \hyperref[asump_M_h1]{\textbf{(H1')}} and \hyperref[asump_M_h2]{\textbf{(H2')}}, 
there exists a constant $K_S$
\begin{equation}
\label{asump_M_eq4_disc}
-K_S \Theta_{k,m} \delta_k S^{-1}_{\star} \preccurlyeq (\widetilde{S}^{-1}_k - S^{-1}_k ) - (\widetilde{S}^{-1}_{\star} - S^{-1}_{\star})   
\preccurlyeq K_S \Theta_{k,m} \delta_k S^{-1}_{\star}.
\end{equation}
\item[\textbf{(H4')}] \label{asump_M_h3_disc} 
Let $u^{\star}_{\phi}$ be the fixed point of the function $\phi$ in \eqref{phi_rhs}.
The sequence $\{\mM_k\}$ and $\mM_{\star}$ satisfy
\begin{equation}
\label{Thetat_bound_disc}
\Theta_{k,m} \le  K_{\theta} \sqrt{ D_f(u_k , u^{\star}_{\phi})  },
\end{equation}
where $K_{\theta}$ is a uniform constant independent of $t$.
\end{itemize}
\end{asump}

Now, we proceed to establish the discrete versions of Lemmas \ref{lemma_error_ineq} and \ref{lem_est_uPiu}.
It should be noted that in these two lemmas, we intentionally avoid imposing any restrictions on 
$\tau_k$ to maintain their universality. 
Instead, the conditions regarding $\tau_k$ are deferred to the forthcoming main theorems.
\begin{lemma}[Error inequality for $\mathcal{E}^{(1)}_k$]
\label{lemma_error_ineq_disc}
Under \hyperref[asump_M_h1_disc]{\textbf{(H1')}}, \hyperref[asump_M_h2_disc]{\textbf{(H2')}} and \hyperref[asump_M_h2_new_disc]{\textbf{(H3')}} in Assumption \ref{asump_M_disc}, and $\alpha \le L_{f,M_{\star}}^{-1}$ and $\delta^{\star}< (4\kappa_{f,M_{\star}})^{-1}$, 
there holds
\begin{equation}
\begin{split}
\label{lemma_error_ineq_disc_eq0}
\frac{\mathcal{E}^{(1)}_{k+1} - \mathcal{E}^{(1)}_{k}}{\tau_k} 
\le   -  \alpha \left( \frac{ \mu_{f,M_k} \kappa^{-1}_{f,M_k}}{2} - 3 L^2_{f,M_k} \tau_k \alpha  \right)  \mathcal{E}^{(1)}_{k}
+  K_{1,k} \alpha^{-1} \mathcal{E}^{(2)}_{k}  + \alpha p^2 K_{2,k} \Theta^2_{k,m},
\end{split}
\end{equation}
where
\begin{equation*}
\begin{split}
&K_{1,k} = 2(2\kappa^{2}_{f,M_k} + \alpha^2 L^2_{f,M_k})  +  \frac{3L_{f,M_k}}{2} \tau_k\alpha, \\
&K_{2,k} =  \left( \frac{3L_{f,M_k}}{2}   \tau_k \alpha   + 2 \kappa^2_{f,M_k} \right) (1+\theta_m)  \| \nabla f(u^{\star}) \|^2_{M^{-1}_{\star}},
\end{split}
\end{equation*}
and $p = p(\delta_{\max},\delta^{\star};\kappa_{f,M_{\star}},K_S)$ is given in \eqref{lem_uxi_eq04}.
\end{lemma}
\begin{proof}
Let us first notice that
\begin{equation}
\label{lemma_error_ineq_disc_eq1}
\mathcal{E}^{(1)}_{k+1} - \mathcal{E}^{(1)}_{k} = 
\underbrace{ D_f(u_{k+1},u_k) }_{R_1} + \underbrace{ \langle \nabla f(u_k) - \nabla f(u^{\star}_{\phi}), u_{k+1} - u_k \rangle }_{R_2}.
\end{equation}
For the first term in the right-hand side above, employing a similar decomposition to \eqref{lemma_error_ineq_eq1} with \eqref{lem_uxi_eq01} in Lemma \ref{lem_uxi} and \eqref{conv_Lip}, we obtain
\begin{equation*}
\begin{split}
 R_1  \le & \frac{L_{f,M_k}}{2} \| u_{k+1} - u_k \|^2_{M_k} 
= \frac{L_{f,M_k}}{2} \tau^2_k \| u_k - \widetilde{P}_{\mM_k}(u_k - \alpha M^{-1}_k \nabla f(u_k)) \|^2_{M_k} \\
 \le &  \frac{3L_{f,M_k}}{2} \tau^2_k 
 \left( 
  \| \zeta_k - \widetilde{P}_{\mM_k} \zeta_k \|^2_{M_k} 
  + \alpha^2 \|  \widetilde{P}_{\mM_k} M^{-1}_k \zeta_{f,k} \|^2_{M_k} 
+ \| u^{\star}_{\phi} - \widetilde{P}_{\mM_k} ( u^{\star}_{\phi} - \alpha M^{-1}_k \nabla f(u^{\star}_\phi) ) \|^2_{M_k} \right) \\
\le &  \frac{3L_{f,M_k}}{2} \tau^2_k
\left( \mathcal{E}^{(2)}_{k}  + 2 \alpha^2 L_{f,M_k} \mathcal{E}^{(1)}_{k} +  (1+\theta_m) \alpha^2 p^2 \| \nabla f(u^{\star}) \|^2_{M^{-1}_{\star}} \Theta^2_{k,m} \right) .
\end{split}
\end{equation*}
Additionally, for the second term in \eqref{lemma_error_ineq_disc_eq1}, 
by Lemma \ref{lemma_error_ineq}, we have
\begin{equation*}
\begin{split}
 R_2 &
 = \tau_k \langle \nabla f(u_k) - \nabla f(u^{\star}_{\phi}), -u_k + \widetilde{P}_{\mM_k}(u_k - \alpha M^{-1}_k \nabla f(u_k)) \rangle \\
& =  \tau_k \frac{\dd}{\dd t} \mathcal{E}^{(1)}(u_k)  
  \le \tau_k \left( - \frac{ \mu_{f,M_k} \kappa^{-1}_{f,M_k}}{2} \alpha \mathcal{E}^{(1)}_{k} 
+ \widetilde{K}_{1,k} \alpha^{-1}  \mathcal{E}^{(2)}_{k} + \widetilde{K}_{2,k} \alpha  p^2 \Theta^2_{k,m} \right),
\end{split}
\end{equation*}
where $\widetilde{K}_{1,k}$ and $\widetilde{K}_{2,k}$ are the constants $K_1$ and $K_2$ in Lemma \ref{lemma_error_ineq}
being evaluated at $\kappa_{f,M_k}$ and $L_{f,M_k}$.
Then, combining all these estimates above, we have the desired result.
\end{proof}


\begin{lemma}[Error inequality for $\mathcal{E}^{(2)}_k$]
\label{lem_est_uPiu_disc}
Under the conditions of Lemma \ref{lemma_error_ineq_disc} and the extra assumption of
$\delta_{\max} := \max_k\{ \delta_k \}_{k\ge 0} \le \frac{1}{8\theta_m+9}$,
there holds for any $\epsilon\ge 0$
\begin{equation}
\begin{split}
\label{lem_est_uPiu_eq0_disc}
\frac{\mathcal{E}^{(2)}_{k+1} - \mathcal{E}^{(2)}_{k} }{\tau_k}
\le & - \left( \frac{7}{4} - \epsilon -  K_{5,k}\tau_k  \right) \mathcal{E}^{(2)}_{k}
+  \alpha^2 \delta^2_k K_{3,k} \epsilon^{-1} \mathcal{E}^{(1)}_{k} 
+  \alpha^2  \left( (1+\tau_k) p^2 + 16(\delta^{\star})^2 \right)  K_{4,k} \Theta^2_{k,m},
\end{split}
\end{equation}
where $p = p(\delta_{\max},\delta^{\star};\kappa_{f,M_{\star}},K_S)$ is given by \eqref{lem_uxi_eq04}, 
$K_{3,k} =  \left( \frac{3}{2} \tau_k\epsilon +  4 \right) (1+\theta_m)^2 L_{f,M_k}$, 
$K_{4,k} =  \frac{3}{2}(1+\theta_m)  \| \nabla f(u^{\star}) \|^2_{M_{\star}^{-1}}$,
and 
$K_{5,k} = \frac{3}{2}(1 + \frac{3}{2}(1+\theta_m)^2\delta_k^2)$.
\end{lemma}
\begin{proof}
Notice 
\begin{equation}
\label{lem_est_uPiu_eq1_disc}
\mathcal{E}^{(2)}_{k+1} - \mathcal{E}^{(2)}_{k} 
=\underbrace{ \frac{1}{2} \| (I - \widetilde{P}_{\mM_{\star}}) (\zeta_{k+1} - \zeta_k) \|^2_{M_{\star}} }_{R_1}
+ \underbrace{ \langle (I - \widetilde{P}_{\mM_{\star}}) (\zeta_{k+1} - \zeta_k),  (I - \widetilde{P}_{\mM_{\star}})  \zeta_k   \rangle_{M_{\star}} }_{R_2}.
\end{equation}
For $R_1$, the same argument as Lemma \ref{lemma_error_ineq_disc} leads to
\begin{equation*}
\begin{split}
\label{lem_est_uPiu_eq2_disc_new}
2R_1 = & \| (I - \widetilde{P}_{\mM_{\star}}) (u_{k+1}-u_k) \|^2_{M_{\star}}
=  \tau^2_k \| (I - \widetilde{P}_{\mM_{\star}}) (u_k - \widetilde{P}_{\mM_k} (u_k - \alpha M^{-1}_k \nabla f(u_k)) ) \|^2_{M_{\star}} \\
 \le & 3\tau^2_k  \left( \| (I - \widetilde{P}_{\mM_{\star}}) \zeta_k  \|^2_{M_{\star}} 
  +  \| (I - \widetilde{P}_{\mM_{\star}}) \widetilde{P}_{\mM_k} \zeta_k \|^2_{M_{\star}}  \right. \\
   &  \left. + \alpha^2 \| (I - \widetilde{P}_{\mM_{\star}}) \widetilde{P}_{\mM_k} M^{-1}_k \zeta_{f,k} \|^2_{M_{\star}}
    + \| (I - \widetilde{P}_{\mM_{\star}}) (u^{\star}_{\phi} - \widetilde{P}_{\mM_k} ( u^{\star}_{\phi} - \alpha M^{-1}_k \nabla f(u^{\star}_\phi) )) \|^2_{M_{\star}} \right) \\
 \le & 3\tau^2_k (1+\theta_m)
 \left( \frac{\mathcal{E}^{(2)}_{k}}{(1+\theta_m)} + \frac{(1+\theta_m) \delta^2_k }{(1-\delta^{\star})^{2}}  \mathcal{E}^{(2)}_{k}  +   L_{f,M_k} (\alpha\delta_k)^2 \mathcal{E}^{(1)}_{k} 
 + \alpha^2 p^2 \| \nabla f(u^{\star}) \|^2_{M^{-1}_{\star}} \Theta^2_{k,m} \right).
\end{split}
\end{equation*}
As for $R_2$, by Lemma \ref{lem_est_uPiu}, we have
\begin{equation*}
R_2 = \tau_k \frac{\dd}{\dd t} \mathcal{E}^{(2)}(u_k) 
\le - \left( \frac{7}{4} - \epsilon \right)\tau_k \mathcal{E}^{(2)}_{k}
+ \tilde{K}_{3,k} \epsilon^{-1} \tau_k (\alpha \delta_k)^2  \mathcal{E}^{(1)}_{k} 
+ {K}_{4,k} \tau_k \left(16(\delta^{\star})^2 +  p^2 \right) \alpha^2 \Theta^2_{k,m},
\end{equation*}
where $\tilde{K}_{3,k}$ and ${K}_{4,k}$ are the constants in Lemma \ref{lem_est_uPiu} 
being evaluated at $\kappa_{f,M_k}$ and $L_{f,M_k}$.
Putting these estimates into \eqref{lem_est_uPiu_eq1_disc} yields the desired estimate.
\end{proof}


\begin{theorem}[\RG{Discrete strong Lyapunov property}]
\label{thm_lyapunov_disc}
Under Assumption \ref{asump_M_disc}, 
assume $\lambda \le (16 K_{1,k})^{-1}$, \RGI{with $K_{1,k}$ given in Lemma \ref{lemma_error_ineq_disc}}, and $\alpha \le L_{f,M_{\star}}^{-1}$ and $\delta$ is small enough such that
\begin{equation}
\begin{split}
\label{thm_lyapunov_disc_eq01}
& \delta_k \le 
\min\left\{ \frac{\sqrt{\lambda} }{21(1+\theta_m)\kappa_{f,M_k}}, ~ \frac{1}{8\theta_m+9} \right\}, \\
~~~
& \delta^{\star} \le \min\bigg\{  \frac{\sqrt{\lambda} }{ 12\sqrt{2} K_{\theta} \sqrt{(1+\theta_m)\kappa_{f,M_k}} C^{\star}} ,\frac{1}{4\kappa_{f,M_{\star}}} \bigg\} , \\
~~~
& p \le \frac{\sqrt{\lambda}}{ 9  K_{\theta} \sqrt{(1+\theta_m)\kappa_{f,M_k}} C^{\star}},
\end{split}
\end{equation}
and $\tau_k$ is sufficiently small such that
\begin{equation}
\label{thm_lyapunov_disc_eq011}
 \tau_k\alpha \le  \frac{1}{36\kappa^2_{f,M_k} L_{f,M_k} }, 
~~~
\tau_k \le \frac{49}{48(1 + \frac{3}{2}(1+\theta_m)\delta^2_k)},
\end{equation}
\RGI{
where $p = p(\delta_{\max},\delta^{\star};\kappa_{f,M_{\star}},K_S)= (9\kappa_{f,M_{\star}}+4)\delta^{\star} 
   + (1 + 2K_S) \delta_{\max}$ is given in \eqref{lem_uxi_eq04},
   $C^{\star} = \mu^{1/2}_{f,M_{\star}} \| \nabla f(u^{\star}) \|_{M_{\star}^{-1}}$,
   and $\delta^{\star}$, $\delta_{\max}$, $\theta_m$, and $K_S$ are all defined in Assumption \ref{asump_M}.
}
Then, there holds
\begin{equation}
\label{thm_lyapunov_disc_eq02}
\frac{\mathcal{E}_{k+1} - \mathcal{E}_k }{\tau_k}  \le - \omega_k \mathcal{E}_k, 
~~~~~~ \omega_k = \min\{ \alpha \frac{  \mu_{f,M_k} \kappa^{-1}_{f,M_k}}{4},  \frac{1}{32} \}.
\end{equation}
\end{theorem}
\begin{proof}
Applying Lemmas \ref{lemma_error_ineq_disc} and \ref{lem_est_uPiu_disc}, 
and setting $\epsilon=1/8$ and using $\lambda \le (16 K_{1,k})^{-1}$, 
employing Assumption  \hyperref[asump_M_h3_disc]{\textbf{(H4')}} we arrive at
\begin{equation*}
\begin{split}
\label{thm_lyapunov_disc_eq1}
\frac{\mathcal{E}_{k+1} - \mathcal{E}_k }{\tau_k}
\le & -  \alpha^2 \left( \frac{ \lambda \mu_{f,M_k} \kappa^{-1}_{f,M_k}}{2} - 3 \lambda L^2_{f,M_k} \tau_k \alpha - \delta^2_k K_{3,k} \epsilon^{-1}  \right)  \mathcal{E}^{(1)}_{k} \\
&  - \left( \frac{7}{4} - \epsilon - K_{5,k}\tau_k - \lambda K_{1,k}  \right)  \mathcal{E}^{(2)}_{k} \\
& + \left( \lambda p^2 K_{2,k}  +    \left( 16(\delta^{\star})^2 + (1+\tau_k) p^2  \right)  K_{4,k} \right) \alpha^2   \Theta^2_{k,m} \\
\le &  -  \alpha^2 \left( \frac{ \lambda \mu_{f,M_k} \kappa^{-1}_{f,M_k}}{2} - 3 \lambda L^2_{f,M_k} \tau_k \alpha 
- 8 \delta^2_k K_{3,k}  \right. \\
& \left.  -  \left( \lambda p^2 K_{2,k}  +    \left( 16(\delta^{\star})^2 + (1+\tau_k)p^2  \right)  K_{4,k} \right) K^2_{\theta} \right)  \mathcal{E}^{(1)}_{k} \\
& - \left( \frac{25}{16}  - \tau_k K_{5,k}  \right)  \mathcal{E}^{(2)}_{k}.
\end{split}
\end{equation*}
We proceed to use the assumptions to estimate the following four terms 
\begin{equation*}
\begin{split}
& 3 \lambda L^2_{f,M_k} \tau_k \alpha, ~ 8 \delta^2_k K_{3,k}  \le \frac{ \lambda \mu_{f,M_k} \kappa^{-1}_{f,M_k}}{12},\\
& p^2 (\lambda K_{2,k}  +  (1+\tau_k)  K_{4,k} ) K^2_{\theta}, ~ 16(\delta^{\star})^2 K_{4,k} K^2_{\theta} \le \frac{ \lambda \mu_{f,M_k} \kappa^{-1}_{f,M_k}}{24}.
\end{split}
\end{equation*}
The first inequality in \eqref{thm_lyapunov_disc_eq011} is equivalent to 
$3 \lambda L^2_{f,M_k} \tau_k \alpha \le \frac{ \lambda \mu_{f,M_k} \kappa^{-1}_{f,M_k}}{12}$,
and the second inequality implies 
$\frac{25}{16}  - \tau_k K_{5,k} \ge \frac{1}{32}$.
It also implies $\tau_k\le \frac{25}{24}$, 
and thus $96(\frac{3}{16}\tau_k + 4) \le 402.3750 \le 21^2$.
Using the first bound of $\delta_k$ in \eqref{thm_lyapunov_disc_eq01}, we obtain
$$
8 \delta^2_k 
\le \frac{\lambda}{12\left( \frac{3}{16}\tau_k +4 \right) \kappa_{f,M_k}^2(1+\theta_m)^2}
= \frac{ \lambda \mu_{f,M_k} \kappa^{-1}_{f,M_k}}{12 K_{3,k}}.
$$
Next, noticing $K_{2,k}/K_{1,k}\le (1+\theta_m) \| \nabla f(u^{\star}) \|^2_{M^{-1}_{\star}}$,
using the upper bound of $p$ in \eqref{thm_lyapunov_disc_eq01},
and noticing $9 > \sqrt{74.25} = \sqrt{24 \left( \frac{1}{16} + (1 + \frac{49}{48} )\frac{3}{2} \right)}$,
we get
\begin{equation}
  \begin{split}
p^2 & \le \frac{\lambda \mu_{f,M_k} \kappa^{-1}_{f,M_k} }{24 \left( \frac{1}{16} + (1 + \frac{49}{48} )\frac{3}{2} \right)K^2_{\theta}  (1+\theta_m) \| \nabla f(u^{\star}) \|^2_{M^{-1}_{\star}} } \\
 & \le \frac{\lambda \mu_{f,M_k} \kappa^{-1}_{f,M_k}}{24 \left( \frac{K_{2,k}}{16 K_{1,k}} + (1+\tau_k)K_{4,k} \right) K^2_{\theta}   }
 \le \frac{\lambda \mu_{f,M_k} \kappa^{-1}_{f,M_k}}{24 \left(\lambda K_{2,k} + (1+\tau_k) K_{4,k} \right) K^2_{\theta}  }.
  \end{split}
\end{equation}
Moreover, the first upper bound of $\delta^{\star}$ in \eqref{thm_lyapunov_disc_eq01} yields the bound of $16 (\delta^{\star})^2 K^2_{\theta} K_{4,k}$.
Then these estimates together yield the desired estimate.
\end{proof}

\begin{theorem}[Optimal linear convergence]
\label{thm_str_cont_lya2_disc}
Under the assumptions of Theorem \ref{thm_lyapunov_disc}, suppose $\tau_k\omega_k \le 1$ for all $k \ge 1$,
\RGI{where $\omega_k = \min\{ \alpha \frac{  \mu_{f,M_k} \kappa^{-1}_{f,M_k}}{4},  \frac{1}{32} \}$ is given in \eqref{thm_lyapunov_disc_eq02}},
then there holds 
\begin{subequations}
\begin{align}
\mathcal{E}_{k}  \le \prod_{l=1}^k \left( 1 - \frac{\min\{ \kappa^{-4}_{f,M_l}/9, ~ \tau_l/2 \}}{16} \right) 
\left( \lambda \alpha \mathcal{E}^{(1)}_0 +  \mathcal{E}^{(2)}_0 \right). 
\label{thm_str_cont_lya2_eq01_disc}
\end{align}
\end{subequations}
If $u_0 \in \ker(B)$, then $ \mathcal{E}^{(2)}_0 \le 4 (\alpha\delta^{\star})^2  \| \nabla f (u^{\star}) \|^2_{M^{-1}_{\star}}$.  
\end{theorem}
\begin{proof}
Notice that
$
\tau_k \alpha  \frac{  \mu_{f,M_k} \kappa^{-1}_{f,M_k}}{4}
\le ( 12 \kappa^2_{f,M_k} )^{-2}
$ from the upper bound of $\tau_k\alpha$ in \eqref{thm_lyapunov_disc_eq011}.
This leads to \eqref{thm_str_cont_lya2_eq01_disc}.
The estimate for $ \mathcal{E}^{(2)}_0$ is similar to \eqref{thm_str_cont_lya2_eq02}.
\end{proof}

\begin{remark}
\label{rem_tau}
The condition \eqref{thm_lyapunov_disc_eq011} indicates that $\tau_k$ can be selected as $1$ 
to recover the standard IPPGD method in \eqref{PPGD0} given sufficiently-small $\delta_k$ and $\alpha$.
However, when computing the inexact projections, it is usually not easy to control the smallness of $\delta_k$.
Then, it would be more desirable to use smaller $\tau_k$.
In fact, our numerical results also suggest that smaller $\tau_k$ can significantly improve the convergence speed in some cases.
\end{remark}
\begin{remark}
\label{const_explic}
All those intermediate constants in Assumption \ref{asump_M_disc} such as $\theta_m$, $K_{\theta}$ and $K_S$, 
as well as the convexity and Lipschitz constants, 
all explicitly appear in Theorems \ref{thm_lyapunov_disc} and \ref{thm_str_cont_lya2_disc},
such that one can explicitly see how they affect the inexactness level and step sizes.
For example, $\theta_m$ measures how far $\mM_0$ is different from $\mM_{\star}$.
The conditions in \eqref{thm_lyapunov_disc_eq01} and \eqref{thm_lyapunov_disc_eq011} show that
for larger $\theta_m$ we should use smaller inexactness level and step sizes.
\end{remark}

\begin{remark}
  \label{rem_conv}
The convergence rate in Theorem \ref{thm_str_cont_lya2_disc} only depends on $\tau_k$ and $\kappa_{f,M_k}$.
For the standard IPPGD method in \eqref{PPGD0} where $\tau_k=1$, the convergence rate is determined by $\kappa_{f,M_k}$ alone.
Notice that $\kappa_{f,M_k} \le (1+\theta_m)^2 \kappa_{f,M_{\star}}$.
Hence, if $\kappa_{f,M_{\star}}$ is independent of the discretized system size, we can conclude that the convergence rate also has this property.
This can be verified by numerical results in the following section for solving nonlinear PDEs.
\end{remark}

\section{Numerical Experiments}
\label{sec:num}

\RGI{
In this section, we apply the IPPGD methods to two nonlinear optimization problems.
\subsection{A resource-allocation problem}
The first is a separable nonlinear resource-allocation problem \cite{1998MuralidharanHanan,2015PatrikssonStrmberg}
\begin{equation}
\label{eq:resource-allocation}
    \min_{\bfx\in \mathbb{R}^n} f(x)
    :=
    \sum_{i=1}^n \phi_i(x_i)
    \qquad
    \text{subject to}\qquad
    \mathbf{1}^\top \bfx = R,
\end{equation}
where
\begin{equation}
\label{eq:phi}
    \phi_i(x_i)
    =
    \log\bigl(1+\exp(a_i x_i+c_i)\bigr)
    +
    \frac{\gamma_i}{2}(x_i-r_i)^2,
    \qquad
    \gamma_i>0.
\end{equation}
Here $x_i$ is the amount of resource assigned to activity $i$, and $R$ is the total available resource.

Let
$
\sigma(t)=\frac{1}{1+e^{-t}}.
$
Compute the derivatives of each scalar component:
\begin{align*}
    \phi_i'(x_i)
    =
    a_i\sigma(a_i x_i+c_i)+\gamma_i(x_i-r_i), 
    ~~~~
    \phi_i''(x_i)
    =
    a_i^2\sigma(a_i x_i+c_i)\bigl(1-\sigma(a_i x_i+c_i)\bigr)+\gamma_i.
\end{align*}
Since
$
    0\leq \sigma(t)(1-\sigma(t))\leq \frac14,
$
we have the uniform scalar bounds
\begin{equation*}
\label{eq:scalar-curvature-bound}
    \gamma_i
    \leq
    \phi_i''(x_i)
    \leq
    \gamma_i+\frac{a_i^2}{4}.
\end{equation*}
Hence, the convexity and Lipschitz constants can be defined as
\begin{equation}
\label{eq:mu-L-euclidean}
    \mu=\min_{1\leq i\leq n}\gamma_i,
    \qquad
    L=\max_{1\leq i\leq n}\left(\gamma_i+\frac{a_i^2}{4}\right).
\end{equation}
 The associated condition number is $\kappa_f = L/\mu$.

Consider a diagonal preconditioner:
\begin{equation}
\label{eq:fixed-diagonal-M}
    M=\text{diag}(m_1,\ldots,m_n),
    \qquad m_i>0,
\end{equation}
Then, the strong-convexity and Lipschitz constants in the $M$-metric are
\begin{equation*}
\label{eq:relative-constants-general-M}
    \mu_{f,M}
    =
    \min_i \frac{\gamma_i}{m_i},
    \qquad
    L_{f,M}
    =
    \max_i \frac{\gamma_i+a_i^2/4}{m_i}.
\end{equation*}
The condition number can be computed accordingly.
For the diagonal metric \eqref{eq:fixed-diagonal-M}, the Schur complement becomes a simple scalar
\begin{equation}
\label{eq:schur-diagonal}
    S=\sum_{i=1}^n\frac{1}{m_i}.
\end{equation}

We now consider the construction of the diagonal preconditioners.
A particularly simple one is the componentwise scaling preconditioner:
\begin{equation}
\label{eq:natural-diagonal-preconditioner}
    M=\text{diag}( \{ \gamma_i+\frac{a_i^2}{4} \}_{i=1,\cdots,n} ).
\end{equation}
For this choice,
\begin{equation}
\label{eq:natural-preconditioner-constants}
    L_{f,M}=1,
    \qquad
    \mu_{f,M}=\min_i \gamma_i(\gamma_i+\frac{a_i^2}{4})^{-1},
    \qquad
    \kappa_{f,M}
    =
    \max_i\left(1+\frac{a_i^2}{4\gamma_i}\right).
\end{equation}
Such a simple fixed scaling preconditioner normalizes the largest possible coordinate curvature.

We may also consider a variable preconditioner given by the regularized Hessian diagonal:
\begin{equation}
\label{eq_diag_hessian}
    M(\bfx)=\text{diag}(m_1(x_1),\ldots,m_n(x_n)),
\end{equation}
where
    $m_i(s)
    = \phi''_i(s) + a^2_i\epsilon =
    \gamma_i
    +
    a_i^2\left( \sigma(a_is+c_i)
    \bigl(1-\sigma(a_is+c_i)\bigr) + \epsilon \right)$.
Let $M_\star:=M(\bfx_{\star})=\text{diag}(m_{\star,1},\ldots,m_{\star,n})$. Then, noticing 
$
m_i(s)-m_{\star,i} = \left( m_i(s)/m_{\star,i} - 1 \right)m_{\star,i}
$,
we obtain
\begin{equation}
\label{eq:H1-M-part}
    M(\bfx)-M_\star\preceq \Theta(\bfx)M_\star,
    \qquad
    M_\star-M(\bfx)\preceq \Theta(\bfx)M(\bfx),
\end{equation}
where
\begin{equation}
\label{eq:Theta-definition}
    \Theta(\bfx)
    :=
    \max_{1\leq i\leq n}
    \left\{
    \frac{m_i(x_i)}{m_{\star,i}}-1,
    \frac{m_{\star,i}}{m_i(x_i)}-1
    \right\}. 
\end{equation}
In addition, the scalar Schur complement can be computed exactly, i.e., $\delta = 0$ and $\widetilde{S}=S$.
This already gives Assumption \hyperref[asump_M_h1]{\textbf{(H1)}},
and Assumptions \hyperref[asump_M_h2]{\textbf{(H2)}} and \hyperref[asump_M_h3]{\textbf{(H3)}} trivially hold.

Next, noticing
$
    \frac{d}{dt}\{\sigma(t)(1-\sigma(t))\}
    =
    \sigma(t)(1-\sigma(t))(1-2\sigma(t)),
$
we get
\begin{equation}
\label{eq:mi-prime}
    m_i'(s)
    =
    a_i^3\sigma(\xi_i(s))(1-\sigma(\xi_i(s)))(1-2\sigma(\xi_i(s))),
    \qquad
    \xi_i(s):=a_is+c_i.
\end{equation}
The sharp scalar bound
    $
    \left|\sigma(t)(1-\sigma(t))(1-2\sigma(t))\right|
    \leq
    \frac{1}{6\sqrt 3}
    $
gives
\begin{equation}
\label{eq:LM}
    |m_i'(s)|\leq \frac{|a_i|^3}{6\sqrt 3},
    \qquad
    K_m:=\max_i\frac{|a_i|^3}{6\sqrt 3}.
\end{equation}
This Lipschitz constant leads to 
$
|m_i(x_i) -m_{\star,i}| \le K_m |x_i-x_{\star,i}|.
$
Therefore, we obtain from \eqref{eq:Theta-definition} that
$$
\Theta(\bfx) \le \frac{K_m}{\mu} \max_{1\le i \le n} |x_i-x_{\star,i}| \le \frac{K_m}{\mu} \| \bfx - \bfx_{\star} \|_2
$$
In addition, since $f$ is $\mu$-strongly convex, there holds
$
    D_f(\bfx,\bfx_{\star})
    \geq
    \frac{\mu}{2}\|\bfx - \bfx_{\star}\|_2^2.
$
Assumption \hyperref[asump_M_h3]{\textbf{(H4)}} immediately holds with $K_{\theta} = \sqrt{2}K_m/\mu^{3/2}$.

Next, we apply the IPPGD method to this problem and present the numerical results.
We let $\left[ a_1,a_2,\cdots,a_n \right] = (0.1+[1/n,2/n,\cdots,1])\times10^{\kappa}$, 
$\left[ \gamma_1,\gamma_2,\cdots,\gamma_n \right] = (0.1+[1/n,2/n,\cdots,1])$,
and $\left[ c_1,c_2,\cdots,c_n \right] = 0.5\left[\sin(2\pi/n), \sin(4\pi/n), \cdots, \sin(2\pi) \right]$.
Additionally, we choose $r_i$ and make $R=0$ such that the minimizer is shifted to $\mathbf{ 0}$.
We select $\kappa=1,2,3$ to adjust the condition numbers.
We consider the aforementioned three types of preconditioners: the identity $M_I$, the scaling one $M_s$ in \eqref{eq:natural-diagonal-preconditioner} and the diagonal Hessian $M_H$ in \eqref{eq_diag_hessian} with $\epsilon = 10^{-3}$. Notice that the last one is a variable preconditioner. 
The initial guess is generated from a standard normal distribution.
Using the relaxed IPPGD method in \eqref{PPGD1_disc} with $\tau_k = 0.8$,
we present all the iteration numbers in Table \ref{tab:resource_allocation_iterations}.
Clearly, the IPPGD method with $M_H$ is highly stable with respect to condition numbers and problem size compared with the other two cases. In addition, it admits the smallest iteration numbers in most challenging situations.
However, we also mention that for the relatively easy situation, the scaling preconditioner $M_s$ is good enough for acceleration.
}

\begin{table}[t]
\centering
\RGI{
\caption{Iteration numbers for the centered resource-allocation problem.}
\label{tab:resource_allocation_iterations}
\begin{threeparttable}
\setlength{\tabcolsep}{5pt}
\renewcommand{\arraystretch}{1.15}
\begin{tabular}{    c    r    r    r    r}
\toprule

\multirow{2}{*}{$\kappa_f$}
&
\multirow{2}{*}{$n$}
&
\multicolumn{3}{c}{IPPGD outer iterations} \\
\cmidrule(lr){3-5}
&
&
\makecell{$M=I$\\ Identity}
&
\makecell{$M=M_{s}$\\ Scaling}
&
\makecell{$M=M_{H}(x_k)$\\ Diagonal Hessian}
\\
\midrule

\multirow{3}{*}{$\approx 310$}
&
$10^3$ & $1526$ & $43$ & $391$ \\
&
$10^4$ & $1708$ & $50$ & $420$ \\
&
$10^5$ & $1854$ & $51$ & $448$ \\
\midrule

\multirow{3}{*}{$\approx 30000$}
&
$10^3$ & $3841$ & $417$ & $331$ \\
&
$10^4$ & $4755$ & $586$ & $359$ \\
&
$10^5$ & $6238$ & $598$ & $387$ \\
\midrule

\multirow{3}{*}{$\approx 3\times 10^6$}
&
$10^3$ & $29116$ & $4129$ & $466$ \\
&
$10^4$ & $39688$ & $5996$ & $600$ \\
&
$10^5$ & $52834$ & $6171$ & $602$ \\

\bottomrule
\end{tabular}
\end{threeparttable}
}
\end{table}


\subsection{Nonlinear elliptic PDEs}
\label{subsex_exam}

Next, we demonstrate that the IPPGD method \eqref{PPGD0} and the modified method in \eqref{PPGD1_disc} can benefit the numerical solution of nonlinear PDEs.

\subsubsection{Problem set-up}
\label{subsex_exam_subs1}
Here we focus on one type of quasilinear elliptic equations \cite{2016Han,1977Scheurer} whose diffusion coefficient depends on the gradients nonlinearly.
It can also be applied to other nonlinear PDEs \cite{2020XuYouseptZou,2018HuangChenRui}.
On a domain $\Omega \subseteq \mathbb{R}^3$,
we aim to find $u\in H^1_0(\Omega)$ such that
\begin{equation}
\label{proj_exam_1_eq1}
\nabla\cdot( \nu(|\nabla u|) \nabla u) = g, ~~~~ \text{in} ~ \Omega, ~~~~~~ u = g_D ~~~ \text{on} ~ \partial \Omega.
\end{equation}
Here, the function $\nu:  \mathbb{R}^+_0 \rightarrow \mathbb{R}^+$ is a continuous function satisfying the following properties \cite{2020XuYouseptZou}.
\begin{itemize}

\item[{(A1)}]  \label{asump_A1}
$\nu$ is continuously differentiable on $\mathbb{R}^+_0$.

\item[(A2)] \label{asump_A2}
There exist positive constants $\nu_m$ and $\nu_M$ such that $\nu_M \ge \nu(t) \ge \nu_m$, $\forall t \ge 0$, and $\lim_{t\rightarrow \infty} \nu(t) = \nu_M$.

\item[(A3)] \label{asump_A3}
$\tilde{\nu}(s):= \nu(s)s$ is invertible on $ \mathbb{R}^+_0$ and is Lipschitz continuous with a Lipschitz constant $\bar{\nu}_M> \nu_M$.
Additionally, $\tilde{\nu}$ is strongly monotone with a monotonicity constant $\nu_m$, i.e.,
\begin{equation}
\label{proj_exam_1_eq01}
(\tilde{\nu}(s) - \tilde{\nu}(t))(s-t) \ge \nu_m |s-t|^2, ~~~~~ \forall s,t \ge 0.
\end{equation}


\end{itemize}

One popular approach for solving \eqref{proj_exam_1_eq1} is to recast it in a mixed formulation \cite{2013BoffiBrezziFortin,1994ChenZhangxin}.
It can also be equivalently written as a constrained optimization problem.
Let us introduce the new variable $\bfsigma = \nu(|\nabla u|) \nabla u$, which yields $\nabla u = \frac{\bfsigma}{\nu(\tilde{\nu}^{-1}(|\bfsigma|))} $. 
Define the function $\Psi(t) = \int^t_0 \tilde{\nu}^{-1}(s) \dd s$,
the Hilbert spaces: $V(\Omega) = \bfH(\text{div};\Omega)$, 
and $W(\Omega)=L^2(\Omega)$.
Then, \eqref{proj_exam_1_eq1} can be formulated as solving the following optimization problem
\begin{equation}
\label{proj_exam_1_eq2}
\min_{\bfsigma \in V(\Omega)} f(\bfsigma):= \int_{\Omega} \Psi(|\bfsigma|) \dd x - \int_{\partial \Omega} g_D \bfsigma\cdot\bfn \dd s, 
~~~~~ \text{subject to} ~ (\text{div}\bfsigma, w) = (g,w), ~~~ \forall w \in W(\Omega).
\end{equation} 
Then, the calculus of variation gives the gradient:
\begin{equation}
\label{proj_exam_1_eq3}
\langle \nabla f(\bfsigma), \bfv \rangle = \int_{\Omega}  \frac{\bfsigma \cdot \bfv}{\nu(\tilde{\nu}^{-1}(|\bfsigma|))}  \dd x
 - \int_{\partial\Omega} g_D \bfv\cdot \bfn \dd s, 
~~~~~ \forall \bfv \in V.
\end{equation}
Certainly, \eqref{proj_exam_1_eq3} implies $\nabla f$ is a nonlinear function on $\bfsigma$.
Using the properties (A1)-(A3) of $\nu$ outlined above, 
we can show that the Lipschitz continuity and convexity properties of the energy functional: 
\begin{subequations}
\label{proj_exam_1_eq5}
\begin{align}
    &   \langle \nabla f(\bfsigma_{1}) - \nabla f(\bfsigma_{2}), \bfv \rangle \le   2 \nu_m^{-1} \| \bfv \|_{L^2(\Omega)} \| \bfsigma_{1} - \bfsigma_{2} \|_{L^2(\Omega)}, \label{proj_exam_1_eq51}   \\
    &    \langle \nabla f(\bfsigma_{1}) - \nabla f(\bfsigma_{2}), \bfsigma_{1} - \bfsigma_{2} \rangle \ge (\bar{\nu}_M)^{-1} \| \bfsigma_{1} - \bfsigma_{2} \|_{L^2(\Omega)}^2 . \label{proj_exam_1_eq52}
\end{align}
\end{subequations}
The proof is standard and given in Appendix \ref{append_nonlin_ell} for completeness.
However, the resulting nonlinear saddle point system 
or the constraint optimization problem from these formulations can be difficult to solve.
Here, we shall apply the proposed IPPGD method.


\subsubsection{Discretization and inexactness}
\label{subsex_exam_subs2}
As for discretization,
we let $V_h$ be the Raviart-Thomas space and $W_h$ be the piecewise constant finite element space.
Now, let us denote by $\bfv_{h,i}$ the basis functions of the discrete space $V_h$
and by $\mathcal{N}(V_h)$ the number of degrees of freedom of $V_h$.
\RG{
Then, the discretized problem is to solve
\begin{equation}
\begin{split}
\label{disc_prob_eq1}
&\min_{\bfsigma_h \in V_h} f(\bfsigma_h):= \int_{\Omega} \Psi(|\bfsigma_h|) \dd x - \int_{\partial \Omega} g_D \bfsigma_h\cdot\bfn \dd s, \\
~~~~ \text{subject to} & ~  (\text{div}\bfsigma_h, w_h) = (g,w_h), ~~ \forall w_h \in W_h(\Omega).
\end{split}
\end{equation}
From now on, denote $\bar{\bfv}_h$ as the vector representation of each FE function $\bfv_h$ in $V_h$.
From \eqref{proj_exam_1_eq3}, the discretized gradient of $f$ can be expressed in terms of a weighted mass matrix capturing the nonlinear coefficient $ \left[ \nu(\tilde{\nu}^{-1}(|\bfsigma_h|)) \right]^{-1}$:
\begin{equation}
\label{proj_exam_1_eq4}
A(\bfsigma_h) = \left[ \int_{\Omega} \left[ \nu(\tilde{\nu}^{-1}(|\bfsigma_h|)) \right]^{-1} \bfv_{h,i}\cdot \bfv_{h,j} \dd x \right]_{i,j=1}^{\mathcal{N}(V_h)},
~~~ \text{and} ~~~
\langle \nabla f(\bfsigma_h), \bar{\bfv}_h \rangle = \langle A(\bfsigma_h)\bar{\bfsigma}_h, \bar{\bfv}_h \rangle.
\end{equation}
For simplicity, we also let $A$ be the usual mass matrix.
Then, it is not hard to see that the constraint in \eqref{proj_exam_1_eq2} becomes
\begin{equation}
\label{eq_constr_B}
B \bar{\bfsigma}_h = \bar{g}_h, ~~~ \text{with} ~~ B = GA,
\end{equation}
where $G$ is the matrix representation of the gradient operator and $\bar{g}_h$ is a certain FE function approximation to $g$.
Then, with a selected preconditioner $M(\bfsigma_h)$, 
the exact Schur complement is 
$$
S(\bfsigma_h) = GA \left( M(\bfsigma_h) \right)^{-1} AG^T
$$
that is close to a matrix representation of a negative Laplacian with variable coefficients.
Thus, to compute the exact projection, one needs to invert these matrices, 
which is quite expensive especially for large-scale problems.
Following the strategy in Section \ref{subsec:ipp_operator},
we introduce an operator $\widetilde{S}(\bfsigma_h)$ approximating $S(\bfsigma_h)$,
such that its inverse $\widetilde{S}^{-1}$ is much easier to compute. 
To show how this inexactness approach works in the present problem, we discuss the following two cases.
\begin{itemize}
\item[Case 1.] Let $M(\bfsigma_h):=\text{diag}(A(\bfsigma_h))$ be the diagonal of the matrix $A(\bfsigma_h)$. Such a choice will make the computation of $M^{-1}$ trivial. Then, the exact Schur complement $S=GA M^{-1}(\bfsigma_h) AG^T$ can also be computed trivially. As it is close to a negative Laplacian,
a multigrid (MG) method can be used to construct $\widetilde{S}^{-1}$; 
namely, $\widetilde{S}^{-1}\bfb$ is defined as the approximated solution to the linear system $S\bfx = \bfb$ by MG with only a few iterations.
More specifically, denote $\mathcal{G}(S(\bfsigma_h),n_{mg})$ as the MG approximation to $S(\bfsigma_h)$ with $n_{mg}$ inner W-cycle iterations and define $\widetilde{S}^{-1}=\mathcal{G}(S(\bfsigma_h),n_{mg})$.
\item[Case 2.] Let $M(\bfsigma_h)$ be a conjugate gradient (CG) approximation of $A(\bfsigma_h)$; 
namely, $(M(\bfsigma_h))^{-1}\bfb$ is defined as the approximated solution to the linear system $A(\bfsigma_h)\bfx = \bfb$ by CG with only a few iterations.
In this case, the inverse of the exact Schur complement, i.e., $S^{-1}$, is very expensive to compute, 
and thus we can use the Schur complement with its inverse in Case 1 as the approximations.
Notice that, even though one can solve the MG iteration precisely, the resulting projection is still inexact,
as the Schur complement itself involves approximation.
\end{itemize}
Meanwhile, in both cases, $M(\bfsigma_h)$ can also be viewed as approximations to $A(\bfsigma_h)$,
introducing the inexactness into both $M^{-1}$ and $\widetilde{S}^{-1}$,
where Case 1 provides a direct approximation method by using the diagonal,
and Case 2 offers an iterative approximation approach.
This observation highlights the broad applicability and flexibility of the proposed framework,
as all the inexactness can be simply subsumed into the Schur complement.

For the present example, our numerical experiments suggest that Case 1 is faster than Case 2,
as the latter one involves more complicated solvers.
Hence, we shall only present the numerical results of Case 1 below.
In computation, we let
$
\widetilde{S}^{-1}_k =  \mathcal{G}(S(\bfsigma_{h,k}),n^{(k)}_{mg}),
$
where $n^{(k)}_{mg}$ can vary in outer iterations to adjust the inexactness
and $\bfsigma_{h,k}$ is the solution obtained at the current step.
Moreover, we can verify such choices of $M$ and $\widetilde{S}$ satisfy Assumption \ref{asump_M_disc} with the involved constants to be estimated explicitly.
As the proof is a little lengthy, we put it in Appendix \ref{append_assump_verify}.
}



The case of linear equations (linear elliptic PDEs) has been well-studied in the literature,
where MG can achieve the iteration number (complexity) independent of the mesh size.
But the nonlinear case is still not well understood.
The proposed algorithm coupled with MG can obtain the mesh-independent convergence rates.

The condition number $\kappa_f$ (measured relative to the $L_2$ metric) is $2\nu_0/\nu_m$ independent of mesh size.
But these problems are essentially ill-conditioned due to the differential operators in the constraint. 
After discretization with a mesh size $h$, the Schur complement in the projection computation will have a condition number $O(h^{-2})$,
being ill-conditioned especially when the mesh size is small.
Based on Remark \ref{rem_conv}, it is not hard to see that the outer iteration number should be independent of the discretization mesh size,
demonstrated by the numerical results below.

To show the effectiveness of the proposed method, 
we consider three types of methods: 
the classical exact PGD method and the IPPGD method with a fixed metric (denoted by PG and IPPGD in Table \ref{table:num_v1}),
and the IPPGD method with variable metric (IPPGDv).
For these three types of methods, we simply fix $\tau_k = 1$ which leads to original algorithm \eqref{PPGD0},
where the difference is just the inexactness and preconditioning metric.
In addition, we also consider the case $\tau_k<1$ corresponding to the new algorithm \eqref{PPGD1_disc},
referred to as IPPGDv-$\tau$.
Theorem \ref{thm_lyapunov_disc} tells us that large $\alpha_k$ can be compensated by small $\tau_k$.
With this strategy, we can achieve faster convergence.


\subsubsection{Numerical results}

Now, let us consider the following specific coefficient function and the true solution:
$$
\nu(s) = a_0 + a_1 e^{-a_2 s} ~~~~~ \text{and} ~~~~~ u(x_1,x_2,x_3) = \sin(x_1)\sin(x_2)\sin(x_3),
$$
where the Dirichlet boundary condition and the source term are computed accordingly.
For such a nonlinear scenario, $\tilde{\nu}(s) = \nu(s)s$ is indeed invertible, 
but there is no analytical form of this inverse function. 
Thus, we shall compute the inverse numerically.
\RGI{In particular}, we will first use a bisection method to locate the value $t$ such that $\tilde{\nu}^{-1}(t)\approx s$ 
and then use a Newton's method to compute the more accurate values.
In addition, we shall consider the two cases: 
\begin{align*}
 (a_0,a_1,a_2)=(1,1,5): ~  \nu_0= 2,~~~ \nu_m \approx 0.86, ~~~ \text{and} ~~~  (a_0,a_1,a_2)=(1,6,5): ~ \nu_0= 7,~~~ \nu_m \approx 0.18.
\end{align*}
The first case has better condition number than the second one.
\RGI{In the computation, we \RGI{set the stopping criterion to be} $\| \bar{\bfsigma}_{h,k+1} - \bar{\bfsigma}_{h,k} \|_{l_2}/\sqrt{\mathcal{N}(V_h)} \le \epsilon$,
where the subscripts $k$ and $k+1$ represent the vector solutions at these steps.
Notice that the scaling of $\sqrt{\mathcal{N}(V_h)}$ is due to the total DoFs of the solutions $\bar{\bfsigma}_{h,k+1}$ and $\bar{\bfsigma}_{h,k}$,
resulting in an average per-degree-of-freedom change.
}
We present the numerical results in Table \ref{table:num_v1}.

\RG{
For Cases 1 and 2, we set $\rho$ to be $10^{-6}$ and $10^{-7}$, respectively, as the first case is relatively easier.
In addition, selecting the inner iteration $n_{mg}\approx 13$ 
(regardless of mesh size and the coefficients) 
is sufficient to reduce the residual error to be less than $10^{-8}$,
corresponding to the exact projection,
where the results are reported in the first column ``P'' of each case in Table \ref{table:num_v1}.
As for the original IPPGD methods, we test fixed and variable metrics, for which the results are reported in the second and third columns of each case, denoted by ``IP'' and ``IPv'', respectively.
Moreover, we also consider the IPPGD-$\tau$ method and report the results in the last two columns.
To impose inexactness, we shall consider two strategies: (1) a dynamic approach that starts $n_{mg}$ at a low value 
(e.g., $n_{mg}=1$) and gradually increases it during the optimization process, 
and (2) a simple approach that fixes $n_{mg}$ at $n_{mg}=6$.
Cases 1 and 2 increase $n_{mg}$ every 6 and 60 iterations and reach a maximum of 5 and 6 by the final iterations, respectively.
These two strategies are reported in the last two columns of each case in Table \ref{table:num_v1},
denoted by ``IPv-$\tau$'' and ``IPv-$\tau$f'', respectively.}

Clearly, using variable preconditioning and projection metric can significantly reduce the number of outer iterations.
This is reasonable, as a fixed metric may not adequately capture the behavior of the nonlinear mass matrix.
In addition, the inexactness can largely reduce the number of inner iterations,
which is illustrated by the block of average Wcycles in Table \ref{table:num_v1}.
We also highlight that the projection at the final stage is still inexact, 
where the error is appropriately tailored according to the mesh size,
which can make the IPPGD method even faster.
\RG{Notice that the outer iteration of the column ``IPv-$\tau$'' slightly decreases at the beginning and eventually reaches stable values on fine meshes.
We expect that this phenomenon is due to the gradually-increasing $n_{mg}$,
since the fixed $n_{mg}$ can stabilize the outer iteration, which is indicated by the last column ``IPv-$\tau$f'' of each case. 
However, the fixed $n_{mg}$ slightly increases the overall computational cost.
Moreover, to examine the growth rate of the computational time with respect to the number of DoFs, i.e., the rate $r$ in the relation $T \approx CT_0(\frac{\# DoF}{\# DoF_0})^r$,
where $T_0$ denotes the computational time of the initial mesh with the initial $\# DoF_0$.
Our results indicate that the computational time grows linearly with the number of DoFs, i.e., $r \approx 1$, as expected.}
Overall, the numerical results clearly show that IPPGDv-$\tau$ $>$ IPPGDv $>$ IPPGD $>$ PGD, 
where ``$>$" means faster. 
These findings highlight that variable metrics and inexactness mechanism collectively accelerate the convergence.

\vspace{-0.1in}

\begin{table}[htp]
	\centering
\caption{Comparison of various algorithms. 
 The number of iterations represents the outer iterations,
 and the number of Wcycles represents the average inner iterations of the whole process (Wcycle is the inner iteration for MG) per outer iteration. 
 Below each method, we indicate the corresponding $(\alpha,\tau)$ used in computation. }
\label{table:num_v1}
	\renewcommand{\arraystretch}{1.25}
  \tiny{
	\begin{tabular}{@{} c c | c c c c c | c  c c  c c @{}}
	\toprule
	\hline
 &  & \multicolumn{5}{c|}{$\kappa_f \approx 2.3129$}  
      & \multicolumn{5}{c}{$ \kappa_f \approx 37.2340$}\\ \cline{3-12}
                                            & DoFs       &      \makecell{ P \\  $(0.7,1)$}   & \makecell{ IP \\ $(0.7,1)$}   & \makecell{ IPv \\ $(0.7,1)$} &   \makecell{ \textbf{IPv-$\tau$} \\ $(1.3,0.5)$ }  &    \makecell{ \textbf{IPv-$\tau$f} \\ $(1.3,0.5)$ }  &   \makecell{P \\ $(0.1,1)$ }      &  \makecell{IP \\ $(0.1,1)$ }    &  \makecell{IPv\\ $(0.1,1)$}  &   \makecell{\textbf{IPv-$\tau$}\\ $(0.6,0.2)$}   &   \makecell{\textbf{IPv-$\tau$f} \\ $(0.6,0.2)$ }   
\\  
 \hline 
\multirow{4}{*}{Iteration}                 & 50688         &     43     &     34       &   27   &      {26}       &   24  &    579       &    579      &  221   &    {190}    &    160     \\  
                                           & 399360        &    41      &     31       &   27   &     {23}     &   24    &    466       &    466       &  212  &    {189}   &    162       \\  
                                           & 3170304       &    38      &     30       &   27   &     {21}     &   22    &    364       &    364     &  228  &   {177}    &     165     \\ 
                                           & 10672128   &  --      &     --     &    27  &       {22}     &   22    &    --       &    --     &  230       &   {169}    &     167     \\  
                                           & 25264128   &  --      &     --     &    27  &       {22}     &    22   &    --       &    --     &  232       &   {169}    &     168     \\  \hline
\multirow{4}{*}{\makecell{Ave. Wcycles\\
           (per out. ite.)}  }                & 50688           &  10.9       &     4.2      &   3.7  &    {3.7}      &    5     &    11.0       &    5.0     &  3.4    &    {3.1}       &   6   \\  
                                          & 399360          &   11.9      &     4.0      &   3.7  &   {3.4}       &     5    &    12.0       &    4.7     &  3.3   &  {3.0}     &    6       \\  
                                           & 3170304        &    12.9     &     4.0      &   3.7  &   {3.3}       &     5    &    13.0        &   4.0     &  3.4   &   {3.0}    &   6     \\ 
                                            & 10672128   &  --      &     --     &    3.7  &       {3.4}     &   5    &    --       &    --     &  3.4       &   {2.9}    &     5     \\  
                                            & 25264128   &  --      &     --     &    3.7  &       {3.4}     &    5    &    --       &    --     &  3.4       &   {2.9}    &     5     \\  \hline
\multirow{4}{*}{\makecell{ CPU time \\    (seconds)   }}        
                                           & 50688           &    11         &    6.5      &   5     &    {4.8}      &  4.2  &    136        &     77   &  34   &    {30}      &    28.4        \\  
                                           & 399360          &    67         &    29       &   25   &   {20}       &   25    &    756        &    448       & 218   &    {193}         &     177      \\ 
                                           & 3170304         &    444        &    212      &   185 &   {138}   &  167       &    4212       &    2526        & 2082   &   {1130}    &    1339     \\  
                                            & 10672128   &  --      &     --     &    665  &       {522}     &    589   &    --       &    --     &  6055       &   {4295}    &     4444     \\  
                                            & 25264128   &  --      &     --     &    1795  &       {1448}    &   1568    &    --       &    --     &  15991       &   {10657}    &     11826     \\   \hline
\multicolumn{2}{c|}{Growth rate of CPU time} & -- & -- & 0.94 & 0.91 & 0.94   & -- & -- & 0.99 & 0.93 & 0.96 \\
   \hline\bottomrule
\end{tabular} 
  }
\end{table}

\section{Conclusion}
We have introduced a specialized ODE model designed to capture the dynamics of 
{inexact preconditioned projected gradient descent (IPPGD) methods}, demonstrating a particular efficacy. 
Discretization of this ODE not only recovers the original IPPGD method \eqref{PPGD0} but also yields a faster alternative.
A delicate and novel Lyapunov function is designed to address the complexities of inexactness and variable preconditioning metrics, 
{ensuring} independence from the variable metric and effectively managing deviations from the constraint set. 
\RGI{The strong Lyapunov property} is rigorously proved at both the continuous and discrete levels under this general framework. 
Furthermore, our theoretical and numerical analyses reveal that 
IPPGD outperforms PGD, IPPGDv and IPPGDv-$\tau$ outperform IPPGD.
\RGI{The current analysis is restricted to affine constraint sets. An important direction for future research is to extend the theory to nonlinear constraints or more general manifolds \cite{1981Botsaris,2021GaoSonAbsilStykel,2023Boumal,2020HuLiuWenYuan,2009ShikhmanStein}. 
However, we emphasize that projections onto manifolds introduce significant analytical challenges.}


\appendix
\section{Continuity and convexity of the nonlinear elliptic equation (see Section\,\ref{subsex_exam_subs1})}
\label{append_nonlin_ell}
It is not hard to show $\tilde{\nu}^{-1}$ has the Lipschitz constant $\nu_m^{-1}$ and the monotonicity constant $\nu_M^{-1}$.
Notice that $\tilde{\nu}(0) = 0$. 
Then, we can conclude $\nu_m t \le \tilde{\nu}(t) \le \nu_M t$ and $\nu_m^{-1}t \le \tilde{\nu}^{-1}(t) \le \nu_m^{-1}t$.
We first show the energy in \eqref{proj_exam_1_eq2} has Lipschitz continuous derivative.
Using \eqref{proj_exam_1_eq3}, we have
\begin{equation}
\begin{split}
\label{append_nonlin_ell_eq1}
 \langle \nabla f(\bfsigma_{1}), \bfv \rangle - \langle \nabla f(\bfsigma_{2}), \bfv \rangle
 = &  \int_{\Omega}  \frac{\bfsigma_{1} \cdot \bfv}{ \nu(\tilde{\nu}^{-1}(|\bfsigma_{1}|)) }  \dd x -  \int_{\Omega}  \frac{\bfsigma_{2} \cdot \bfv}{\nu(\tilde{\nu}^{-1}(|\bfsigma_{2}|))}  \dd x
 =: \int_{\Omega}  L \cdot \bfv \dd x.
 \end{split}
\end{equation}
We begin with the case that neither $\bfsigma_{1}$ or $\bfsigma_{2}$ is 0.
Applying $\tilde{\nu}^{-1}(t)\le \nu^{-1}_1 t$, we can write down
\begin{equation}
\begin{split}
\label{append_nonlin_ell_eq2}
 L 
  =&  \left( \tilde{\nu}^{-1}(|\bfsigma_{1}|)  - \tilde{\nu}^{-1}(|\bfsigma_{2}|)  \right)  \frac{\bfsigma_{1}}{|\bfsigma_{1}|} 
  + \tilde{\nu}^{-1}(|\bfsigma_{2}|)  \left(  \frac{\bfsigma_{1}}{|\bfsigma_{1}|}  -   \frac{\bfsigma_{2}}{|\bfsigma_{2}|}  \right) \\
  \le & \nu_m^{-1} | \bfsigma_{1} - \bfsigma_{2} |\frac{\bfsigma_{1}}{|\bfsigma_{1}|} +  \nu_m^{-1} |\bfsigma_{2}| \frac{|\bfsigma_{1} - \bfsigma_{2}|}{|\bfsigma_{2}|},
 \end{split}
\end{equation}
which trivially implies $|L| \le 2 \nu_m^{-1} | \bfsigma_{1} - \bfsigma_{2} |$.
Next, we consider the case that one of $\bfsigma_{1}$ or $\bfsigma_{2}$ is $0$, say $\bfsigma_{2}$, without loss of generality.
As $\tilde{\nu}(0) = 0$, we know $\nu(\tilde{\nu}^{-1}(0)) = \nu(0)$. 
 Combining these estimates, we obtain \eqref{proj_exam_1_eq51}.

As for the convexity, noticing $\frac{ \bfsigma_{1} \cdot \bfsigma_{2}}{|\bfsigma_{1}|} \le |\bfsigma_{2}|$ and 
 $\frac{ \bfsigma_{2} \cdot \bfsigma_{1}}{|\bfsigma_{2}|} \le |\bfsigma_{1}|$, we have
 \begin{equation*}
 \begin{split}
\label{append_nonlin_ell_eq4}
&L\cdot ( \bfsigma_{1}  - \bfsigma_{2}  ) = (\bar{\nu}_M)^{-1}|  \bfsigma_{1}  - \bfsigma_{2} |^2 \\ 
& + ( \tilde{\nu}^{-1}(|\bfsigma_{1}|) - (\bar{\nu}_M)^{-1}|\bfsigma_{1}| ) \frac{|\bfsigma_{1}|^2 - \bfsigma_{1} \cdot \bfsigma_{2}}{ |\bfsigma_{1}| } \\
& + ( \tilde{\nu}^{-1}(|\bfsigma_{2}|) - (\bar{\nu}_M)^{-1}|\bfsigma_{2}| ) \frac{|\bfsigma_{2}|^2 - \bfsigma_{1} \cdot \bfsigma_{2}}{ |\bfsigma_{2}| }  \\
 & ~~~ \ge (\bar{\nu}_M)^{-1}|  \bfsigma_{1}  - \bfsigma_{2} |^2 
 + ( \tilde{\nu}^{-1}(|\bfsigma_{1}|) - \tilde{\nu}^{-1}(|\bfsigma_{2}|) ) ( |\bfsigma_{1}| - |\bfsigma_{2}| ) \\
 & - (\bar{\nu}_M)^{-1}( |\bfsigma_{1}| - |\bfsigma_{2}| )^2  \ge  (\bar{\nu}_M)^{-1}|  \bfsigma_{1}  - \bfsigma_{2} |^2 
 \end{split}
\end{equation*}
where the last inequality holds due to the monotonicity property of $\tilde{\nu}^{-1}$.
Hence, \eqref{proj_exam_1_eq52} is obtained.


\section{Proof of Lemma \ref{lem_Pi_est}}
\label{append_lem_Pi_est}
From \eqref{proj_Pi_eq1}, we have $\| \widetilde{P}_{\mM} u \|^2_M = \langle u,  \widetilde{P}^T_{\mM}M \widetilde{P}_{\mM}  u \rangle = \langle u, M \widetilde{P}_{\mM}^2 u \rangle$ 
and $\| P_M u \|^2_M = \langle u,  P^T_M M P_M  u \rangle = \langle u, M P^2_M u \rangle$.
We then write down
\begin{equation}
\label{lem_Pi_est_eq1}
\begin{split}
M \widetilde{P}_{\mM}  &  = M - B^T\widetilde{S}^{-1}B , \\
M \widetilde{P}_{\mM}^2 
& = M - B^T (2 \widetilde{S}^{-1} - \widetilde{S}^{-1} S \widetilde{S}^{-1}) B, \\
 MP^2_M & =  M P_M  = M - B^T{S}^{-1}B.
\end{split}
\end{equation}
Note that \eqref{lem_Pi_est_eq011} is trivial.
We first show \eqref{lem_Pi_est_eq01}.
Notice that $S^{-1} - (2 \widetilde{S}^{-1} - \widetilde{S}^{-1}S \widetilde{S}^{-1}) = (S^{-1} - \widetilde{S}^{-1})S(S^{-1} - \widetilde{S}^{-1})\succcurlyeq 0$.
Hence, we obtain $S^{-1} \succcurlyeq 2 \widetilde{S}^{-1} - \widetilde{S}^{-1}S \widetilde{S}^{-1} $, which then yields \eqref{lem_Pi_est_eq01}.
In addition, \eqref{lem_Pi_est_eq02} follows from
$(2 \widetilde{S}^{-1} - \widetilde{S}^{-1}S \widetilde{S}^{-1}) - \widetilde{S}^{-1} 
= \widetilde{S}^{-1}( \widetilde{S} - S ) \widetilde{S}^{-1} \succcurlyeq 0$ 
due to $\widetilde{S} \succcurlyeq S$. 

\RG{Next, still using \eqref{lem_Pi_est_eq1} and $(1-\epsilon)S^{-1} \preccurlyeq  \widetilde{S}^{-1} \preccurlyeq  {S}^{-1}$ from the assumption, we obtain
$$
(I - \widetilde{P}_{\mM} )^TM(I - \widetilde{P}_{\mM} ) = B^T\widetilde{S}^{-1}S  \widetilde{S}^{-1}B 
\succcurlyeq (1-\epsilon) B^T \widetilde{S}^{-1} B \succcurlyeq  (1-\epsilon)^2 (I - P_M)^T M (I - P_M),
$$}
which leads to the first inequality in \eqref{lem_Pi_est_eq03}.
The second one follows from a similar argument.
As for \eqref{lem_Pi_est_eq04}, 
Lemma \ref{lem_SPD_general} implies
\begin{equation*}
\begin{split}
( \widetilde{P}_{\mM} - P_M)M^{-1} (\widetilde{P}_{\mM} - P_M)^T & = M^{-1}B^T (\widetilde{S}^{-1} - {S}^{-1}) S (\widetilde{S}^{-1} - {S}^{-1}) B M^{-1} \\
&  \preccurlyeq \epsilon^2 M^{-1}B^T S^{-1} BM^{-1} \preccurlyeq  \epsilon^2 M^{-1},
\end{split}
\end{equation*}
where in the last inequality we have used $M^{-1} \succcurlyeq M^{-1}B^T S^{-1} BM^{-1} $ that is standard for exact projections.

As for \eqref{lem_Dproj_new_eq0}, the direct computation yields
\begin{equation}
\begin{split}
\label{lem_Dproj_new_eq1}
& \widetilde{P}^T_{\mM_1} ( I - \widetilde{P}^T_{\mM_2} )  M_2 ( I - \widetilde{P}_{\mM_2} ) \widetilde{P}_{\mM_1} \\
= & (I - B^T \widetilde{S}^{-1}_1B M^{-1}_1) B^T \widetilde{S}^{-1}_2 B M^{-1}_2 B^T \widetilde{S}^{-1}_2 B  (I - M^{-1}_1 B^T \widetilde{S}^{-1}_1B ) \\
= & B^T  (I -  \widetilde{S}^{-1}_1 S_1) \widetilde{S}^{-1}_2 S_2 \widetilde{S}^{-1}_2 (I - S_1 \widetilde{S}^{-1}_1) B.
 \end{split}
\end{equation}
Using the assumption on $\mM_1$ and $\mM_2$, we have $\widetilde{S}^{-1}_2 S_2 \widetilde{S}^{-1}_2\preccurlyeq \widetilde{S}^{-1}_2 \preccurlyeq  {S}^{-1}_2 \preccurlyeq c{S}^{-1}_1$.
Putting this inequality into \eqref{lem_Dproj_new_eq1} and using Lemma \ref{lem_SPD_general},
we obtain
\begin{equation}
\label{lem_Dproj_new_eq2}
\begin{split}
 \widetilde{P}^T_{\mM_1} ( I - \widetilde{P}^T_{\mM_2} )  M_2 ( I - \widetilde{P}_{\mM_2} ) \widetilde{P}_{\mM_1} \preccurlyeq c B^T  (I -  \widetilde{S}^{-1}_1 S_1)  S^{-1}_1 (I - S_1 \widetilde{S}^{-1}_1) B  \preccurlyeq c \RG{\epsilon^2_1} B^T {S}^{-1}_1 B, \\
 \end{split}
\end{equation}
which yields the desired result in \eqref{lem_Dproj_new_eq0}.
Hence, \eqref{lem_Dproj_new_eq01} follows from \eqref{lem_Dproj_new_eq2} and
$$
(I - \widetilde{P}_{\mM_1})^T M_1 (I - \widetilde{P}_{\mM_1}) = B^T \widetilde{S}^{-1}_1 S_1  \widetilde{S}^{-1}_1 B 
\succcurlyeq (1-\epsilon_1) B^T \widetilde{S}^{-1}_1 B \succcurlyeq (1-\epsilon_1)^2 B^T {S}^{-1}_1 B.
$$


\RG{
\section{Verification of the assumptions on metrics and inexactness (cf.\,Section \ref{assum_metric})}
\label{append_assump_verify}
We now verify the assumptions (H1)-(H4) for the nonlinear elliptic example 
in Section\,\ref{subsex_exam_subs1}.
Since the analysis of preconditioners is highly problem-dependent, 
here we only verify \hyperref[asump_M_h2_new]{\textbf{(H3)}} with a Bramble-Pasciak-Xu (BPX) MG preconditioner \cite{2017XuZikatanov}.

As the proof is applicable to both the continuous and discrete cases,
we shall not distinguish them during our discussion for simplicity. 

From \hyperref[asump_A3]{(A3)}, we immediately obtain 
\begin{equation}
\label{append_assump_verify_eq0}
|\tilde{\nu}'(t)| \in [\nu_m, \bar{\nu}_M]. 
\end{equation}
As $\nu'(s)=(\tilde{\nu}'-\nu(s))/s$,
we conclude from \eqref{append_assump_verify_eq0} and \hyperref[asump_A2]{(A2)} that $\lim_{s\rightarrow \infty} \nu'(s) = 0$.
Combining it with \hyperref[asump_A1]{(A1)}, we obtain
\begin{equation}
\label{append_assump_verify_eq01}
|\nu'(s)| \le \hat{\nu}_M ~~~ \forall s\in [0,\infty],
\end{equation}
for some non-negative constant $\hat{\nu}_M$.

Now, let us first verify \hyperref[asump_M_h1]{\textbf{(H1)}} with the constant satisfying \hyperref[asump_M_h3]{\textbf{(H4)}}.
Moreover, to facilitate presentation, we let $N_{\mathcal{T}}$ be the number of elements $\{T\}$ in the triangulation.
Let $M_{\star}=\text{diag}(A(\bfsigma^{\star}_{h,\phi}))$ where $\bfsigma^{\star}_{h,\phi}$ is the fixed point given by Lemma \ref{lem_fixpt}.
For a general $\bfsigma_h$ defined as $\bfsigma_h|_{T} = \bfsigma_T$, $\forall T\in\mathcal{T}_h$, the corresponding metric can be expressed as a diagonal matrix
\begin{equation}
\label{append_assump_verify_eq1}
M(\bfsigma_h) = \left[ (\nu(\tilde{\nu}^{-1}(|\bfsigma_{T_i}|)))^{-1} |T_i| \right]_{i=1}^{N_{\mathcal{T}}}.
\end{equation}
It is not hard to see that the $i$-th element of the diagonal matrix $M(\bfsigma_h)$ is 
$$
m_i(\bfsigma_h) := |T_i|(\nu(\tilde{\nu}^{-1}(|\bfsigma_{T_i}|)))^{-1}.
$$
Noticing that $|(1/\nu)'|=|\nu'/\nu^2| \le \nu_m^{-2} \hat{\nu}_M$ and that $\tilde{\nu}^{-1}$ has the Lipschitz constant $\nu_m^{-1}$,
we have the following estimation:
\begin{equation}
\begin{split}
\label{append_assump_verify_eq4}
& | m_i(\bfsigma_h)  - m_i(\bfsigma^{\star}_{h,\phi})  |  \le |T_i| \nu_m^{-2} \hat{\nu}_M | \tilde{\nu}^{-1}(|\bfsigma_{T_i}|) - \tilde{\nu}^{-1}(|\bfsigma^{\star}_{T_i}|) | \\
 \le & \nu_m^{-3}  \hat{\nu}_M |T_i| | \bfsigma_{T_i} - \bfsigma^{\star}_{T_i} | 
 \le \nu_m^{-3}  \hat{\nu}_M \nu_M | \bfsigma_{T_i} - \bfsigma^{\star}_{T_i} | \underbrace{ |T_i|  (\nu(\tilde{\nu}^{-1}(|\bfsigma^{\star}_{T_i}|)))^{-1} }_{=m_i(\bfsigma^{\star}_{h,\phi})}.
\end{split}
\end{equation}
Then, for any vector $\bfx\in \mathbb{R}^{N_{\mathcal{T}}}$, we obtain
\begin{equation}
\begin{split}
\label{append_assump_verify_eq5}
\bfx^T(M(\bfsigma_h) - M_{\star})\bfx & \le \sum_{i=1}^{N_{\mathcal{T}}}  x^2_i  | m_i(\bfsigma_h)  - m_i(\bfsigma^{\star}_{h,\phi})  | \\
& \le  \sum_{i=1}^{N_{\mathcal{T}}}  x^2_i  \nu_m^{-3}  \hat{\nu}_M \nu_M | \bfsigma_{T_i} - \bfsigma^{\star}_{T_i} | m_i(\bfsigma^{\star}_{h,\phi}) \\
& \le \max_i\{  | \bfsigma_{T_i} - \bfsigma^{\star}_{T_i} | \}  \nu_m^{-3} \hat{\nu}_M \nu_M \bfx^T M_{\star} \bfx
\end{split}
\end{equation}
which implies \eqref{append_assump_verify_eq2} with $\tilde{K}_{\theta} = \nu_m^{-3}  \hat{\nu}_M \nu_M$.
Hence, using \eqref{proj_exam_1_eq52}, we have
\begin{equation}
\label{append_assump_verify_eq2}
M(\bfsigma_h) - M_{\star} \preccurlyeq 
\underbrace{ \sqrt{2} \frac{ \| \bfsigma_h - \bfsigma^{\star}_{h,\phi} \|_{L^\infty(\Omega)} }{ \| \bfsigma_h - \bfsigma^{\star}_{h,\phi} \|_{L^2(\Omega)} } 
\nu_m^{-3} (\bar{\nu}_M)^{1/2} \hat{\nu}_M \nu_M }_{=K_{\theta}}
\underbrace{ (2\bar{\nu}_M)^{-1/2} \| \bfsigma_h - \bfsigma^{\star}_{h,\phi} \|_{L^2(\Omega)} }_{\le \sqrt{D_f(\bfsigma_h, \bfsigma^{\star}_{h,\phi})}} M_{\star}.
\end{equation}
Denote $\gamma = \frac{ \| \bfsigma_h - \bfsigma^{\star}_{h,\phi} \|_{L^\infty(\Omega)} }{ \| \bfsigma_h - \bfsigma^{\star}_{h,\phi} \|_{L^2(\Omega)} }$.
In $K_{\theta}$, all the constants other than $\gamma$ are completely independent of discretization.
Regarding $\gamma$, we may simply write down
$
|\Omega|^{-1/2} \le \gamma \le c_1 h^{-2}
$, 
for a mesh regularity constant $c_1$, which implies
\begin{equation}
\label{append_assump_verify_eq6}
|\Omega|^{-1/2} \le \frac{K_{\theta}}{\sqrt{2}\nu_m^{-3} (\bar{\nu}_M)^{1/2} \hat{\nu}_M \nu_M} \le c_1 h^{-2}.
\end{equation}
This leads to \eqref{asump_M_eq21} in \hyperref[asump_M_h1]{\textbf{(H1)}} with \hyperref[asump_M_h3]{\textbf{(H4)}}.
Here, we mention that if the constant $\gamma$ behaves ``well enough'', indeed observed in our numerical experiments,
then $K_{\theta}$ is almost $O(1)$ such that the step size can be large.

Next, for \eqref{asump_M_eq22}, using \eqref{asump_M_eq21}, we have $- \Theta(t) S \preccurlyeq S - S_{\star} = B(M^{-1}-M^{-1}_{\star})B^T \preccurlyeq \Theta(t) S_{\star}$, which gives
\begin{equation}
\label{append_assump_verify_eq7}
- \frac{\Theta(t)}{1+\Theta(t)} S^{-1}_{\star} \preccurlyeq S^{-1} - S^{-1}_{\star} \preccurlyeq \Theta(t) S^{-1}_{\star}.
\end{equation} 
In addition, \hyperref[asump_M_h2_new]{\textbf{(H3)}} immediately implies
$$
-K_S \Theta_m(t) \delta_{\max} S^{-1}_{\star} + (S^{-1} - S^{-1}_{\star} ) \preccurlyeq 
\widetilde{S}^{-1} - \widetilde{S}^{-1}_{\star} 
\preccurlyeq K_S \Theta_m(t) \delta_{\max} S^{-1}_{\star} + (S^{-1} - S^{-1}_{\star} ).
$$ 
Then, \eqref{asump_M_eq22} follows from the inequality above and \eqref{append_assump_verify_eq7}.
Hence, it remains to verify Assumption \hyperref[asump_M_h2_new]{\textbf{(H3)}} below.}

\RGI{
Here, \hyperref[asump_M_h2_new]{\textbf{(H3)}} is more difficult, and we shall prove it based on a specific BPX MG construction \cite{2017XuZikatanov} of $\widetilde{S}$ with a fixed number of Richardson iterations.
Meanwhile, we will also prove Assumption \hyperref[asump_M_h2]{\textbf{(H2)}} under this setting.
To this end, we connect $M_\star$ and $M$ by the affine path
\begin{equation*}
        M_s := M_\star+s(M-M_\star), \qquad s\in[0,1],
\end{equation*}
and define
\begin{equation*}
        S_s := B M_s^{-1}B^T .
\end{equation*}
Thus $S_0=S_\star$ and $S_1=S$.

To construct the MG preconditioner and the associated approximation to $S_s$,
we introduce $I_i$ as a natural inclusion operator from the $i$-th level space into the finest space.  
Define the level matrices associated with $S_s$
and the corresponding damped Jacobi smoothers on each level:
\begin{equation}
\label{mg_level_i}
        S_{i,s}:=I_i^T S_s I_i ,
        ~~~~~
        D_{i,s}:=\text{diag}(S_{i,s}),
        ~~~~~
         R_{i,s}:=\omega_i D_{i,s}^{-1},
        \qquad \omega_i>0 .
\end{equation}
Define the additive BPX-type MG preconditioner
\begin{equation}
\label{Bs_def}
        \mathcal{B}_s := \sum_{i=1}^L I_i R_{i,s} I_i^T
        =\sum_{i=1}^L \omega_i I_i D_{i,s}^{-1}I_i^T,
\end{equation}
where $L$ is the number of layers for meshes.
The matrix $\mathcal{B}_s$ itself presents an approximate inverse of $S_s$,
which we will prove in Lemma \ref{lem_Bs_equiv} below.

A single $\mathcal{B}_s$ may not be enough, and thus we will use it to construct $\widetilde{S}^{-1}_s$ through an iterative solver.
For any vector \(b\), essentially we aim to approximate $S^{-1}_sb$.
We apply a damped preconditioned Richardson iteration with the multigrid
preconditioner \(\mathcal B_s\):
\[
        x^{k+1}
        =
        x^k+\tau \mathcal B_s(b-S_sx^k).
\]
Equivalently, it can be written as
$
        x^{k+1}
        =
        (I-\tau \mathcal B_sS_s)x^k+\tau \mathcal B_s b
$.
Define the one-step error operator
\begin{equation}
\label{Ss_eq1}
        T_s:=I-\tau\mathcal B_sS_s .
\end{equation}     
Then the iteration can be written as
$
        x^{k+1}=T_sx^k+\tau\mathcal B_s b
$.
Starting from \(x^0=0\), after \(n_{\rm mg}\) iterations,
we obtain
\[
        x^{n_{\rm mg}}
        =
        \sum_{j=0}^{n_{\rm mg}-1}
        T_s^j \tau\mathcal B_s b .
\]
Since the map \(b\mapsto x^{n_{\rm mg}}\) is linear, the corresponding
approximate inverse can be defined as
\begin{equation}
\label{Ss_eq2}
        \widetilde S_s^{-1}
        :=
        \sum_{j=0}^{n_{\rm mg}-1}
        T_s^j \tau\mathcal B_s,
\end{equation}
which is precisely the operator produced by
applying \(n_{\rm mg}\) damped MG-preconditioned Richardson
iteration steps. 
In particular, there hold $\widetilde S_{\star} = \widetilde S_{0}$ and $\widetilde S = \widetilde S_{1}$.
Moreover, it is easy to see that $\widetilde S_s^{-1}$ is symmetric;
in fact, it is also positive definite, see \eqref{thm_verifyH2_eq2} below.

Now, we define the error operator
\begin{equation}
\label{eq_Es}
        E_s := \widetilde S_s^{-1}-S_s^{-1}
\end{equation}
From now on, we use the notation $\dot{A}$ to denote the derivative of a matrix $A$ with respect to $s$.
Assumption \hyperref[asump_M_h3]{\textbf{(H3)}} is to estimate $E_1 - E_0$, and the fundamental idea is to employ the identity
\begin{equation}
\label{eq_E1E0}
E_1 - E_0 = \int_0^1 \dot{E_s} \dd s.
\end{equation}
Thus, it remains to estimate $\dot{E_s}$.
Let us first derive a useful form for $E_s$.
Notice that \eqref{Ss_eq1} can be rewritten into
$
       (I-T_s)S_s^{-1}
        =
        \tau\mathcal B_s
$.
Substituting it into \eqref{Ss_eq2}, we get
$
        \widetilde S_s^{-1}
        =
        \sum_{j=0}^{n_{\rm mg}-1}
        T_s^j(I-T_s)S_s^{-1}
$.
Since \(I-T_s\) is a polynomial in \(T_s\), it commutes with \(T_s\)
and thus
$
        \sum_{j=0}^{n_{\rm mg}-1}T_s^j(I-T_s)
        =
        I-T_s^{n_{\rm mg}}
$.
Therefore, we have
\begin{equation}
\label{eq2_E1E0}
        \widetilde S_s^{-1}
        =
        (I-T_s^{n_{\rm mg}})S_s^{-1}
\end{equation}
Consequently, the inverse error has the exact representation
\begin{equation}
\label{eq3_E1E0}
E_s = 
        \widetilde S_s^{-1}-S_s^{-1}
        =
        -T_s^{n_{\rm mg}}S_s^{-1}.
\end{equation}

The next stage is to study the approximation of $\mathcal{B}_s$ to $S_s$. 
For $s=0$, it immediately follows from the spectral equivalence in standard MG theory \cite{2017XuZikatanov}:
\begin{equation}
\label{lem_Bs_equiv_star}
        a_{\star} S_{\star}^{-1} \preceq \mathcal{B}_{\star} \preceq b_{\star} S_{\star}^{-1}.
\end{equation}
It is the usual MG property needed to guarantee mesh-independent convergence.
With \eqref{asump_M_eq2} in Assumption \hyperref[asump_M_h1]{\textbf{(H1)}}, the following spectral-equivalence estimate holds for such a class of BPX-type MG preconditioners $\mathcal{B}_s$.

\begin{lemma}
\label{lem_Bs_equiv}
Under \eqref{asump_M_eq21} in Assumption \hyperref[asump_M_h1]{\textbf{(H1)}}, there exist constants $0<a\le b<\infty$, independent of $s$ and of the mesh size, such that
\begin{equation}
\label{lem_Bs_equiv_eq1}
        a S_s^{-1} \preceq \mathcal{B}_s \preceq b S_s^{-1},
        \qquad \forall s\in[0,1],
\end{equation}
where $a=a_{\star}(1+\theta)^{-2}$ and $b = b_{\star}(1+\theta)^{2}$.
\end{lemma}
\begin{proof}
For simplicity, we let $C_{\theta} = (1+\theta)$. 
By \eqref{asump_M_eq2}, we have
$
        C_{\theta}^{-1}M_\star
        \preceq
        M
        \preceq
        C_{\theta} M_\star .
$
Since $M_s=(1-s)M_\star+sM$ and $s\in[0,1]$, it follows that
\begin{equation}
\label{lem_Bs_equiv_eq3}
        C_{\theta}^{-1}M_\star
        \preceq
        M_s
        \preceq
        C_{\theta} M_\star ,
        ~~~~
        C_{\theta}^{-1}M
        \preceq
        M_s
        \preceq
        C_{\theta} M.
\end{equation}
Taking inverses and multiplying by $B$ and $B^T$ gives
$
        C_\theta^{-1}S_\star
        \preceq
        S_s
        \preceq
        C_\theta S_\star ,
$
and hence
\begin{equation}
\label{lem_Bs_equiv_eq2}
        C_\theta^{-1}S_\star^{-1}
        \preceq
        S_s^{-1}
        \preceq
        C_\theta S_\star^{-1}.
\end{equation}
Multiplying by $I$ and $I^T$ gives
the estimate on each level $i$:
\[
        C_\theta^{-1}S_{i,\star}
        \preceq
        S_{i,s}
        \preceq
        C_\theta S_{i,\star}.
\]
As diagonals preserve the Lowener order for diagonal matrices, we obtain
$
        C_\theta^{-1}D_{i,\star}
        \preceq
        D_{i,s}
        \preceq
        C_\theta D_{i,\star}.
$
Taking inverses gives
$
        C_\theta^{-1}D_{i,\star}^{-1}
        \preceq
        D_{i,s}^{-1}
        \preceq
        C_\theta D_{i,\star}^{-1}.
$
Multiplying by $\omega_iI_i$ and $I_i^T$ and summing over $i$ leads to
\[
        C_\theta^{-1}\mathcal B_\star
        \preceq
        \mathcal B_s
        \preceq
        C_\theta \mathcal B_\star .
\]
Using \eqref{lem_Bs_equiv_star} and \eqref{lem_Bs_equiv_eq2}, we obtain
\[
        \mathcal B_s
        \succeq
        C_\theta^{-1}a_\star S_\star^{-1}
        \succeq
        \frac{a_\star}{C_\theta^2}S_s^{-1},
~~~~
\text{and}
~~~~
        \mathcal B_s
        \preceq
        C_\theta b_\star S_\star^{-1}
        \preceq
        b_\star C_\theta^2 S_s^{-1}.
\]
This proves the desired estimate.
\end{proof}


Next, we shall prove that $\widetilde S$ defined in \eqref{Ss_eq2} can indeed fulfill Assumptions \hyperref[asump_M_h2]{\textbf{(H2)}} and \hyperref[asump_M_h3]{\textbf{(H3)}}.
For simplicity, we introduce the following matrix norm for an arbitrary matrix $A$ with a given SPD matrix $B$:
\[
        \|A\|_{B}
        :=\left\|B^{1/2}A B^{-1/2}\right\|_2 .
\]
We shall frequently use the following inequality:
\begin{equation}
\label{AB_ineq}
\| AC \|_B = \| B^{1/2} AC B^{-1/2} \|_2 = \| B^{1/2} A B^{-1/2} B^{1/2} C B^{-1/2} \|_2 \le \|A\|_B \|C\|_B.
\end{equation}

\begin{theorem}
\label{thm_verifyH2}
Under \eqref{asump_M_eq21} in Assumption \hyperref[asump_M_h1]{\textbf{(H1)}} and \eqref{lem_Bs_equiv_eq1},
Assumption \hyperref[asump_M_h2]{\textbf{(H2)}} holds with $\delta = \rho^{n_{\text{mg}}}$ and $\rho = 1 - a\tau < 1$
for sufficiently small $\tau \le 1/b$.
\end{theorem}
\begin{proof}
Lemma \ref{lem_Bs_equiv} immediately implies $\lambda(S_s^{1/2}\mathcal{B}_sS_s^{1/2}) \in [a,b]$.  
Then, noticing the identity
$
        S_s^{1/2}T_sS_s^{-1/2}
        = I-\tau S_s^{1/2}\mathcal{B}_sS_s^{1/2}
$,
we have
\begin{equation}
\label{thm_verifyH2_eq1}
      \lambda(S_s^{1/2}T_sS_s^{-1/2}) \in [0,\rho].
\end{equation}
Using the identity \eqref{eq2_E1E0}, we write down
\[
        S_s^{1/2}\widetilde S_s^{-1}S_s^{1/2}
        =I- (S_s^{1/2}T_sS_s^{-1/2})^{n_{\rm mg}}.
\]
Using \eqref{thm_verifyH2_eq1} we have $\lambda(S_s^{1/2}\widetilde S_s^{-1}S_s^{1/2}) \in [1-\rho^{n_{\rm mg}},1]$. 
which is equivalent to
\begin{equation}
\label{thm_verifyH2_eq2}
        (1-\delta)\widetilde S_s\preceq S_s\preceq \widetilde S_s.
\end{equation}
At $s=0,1$, this is precisely Assumption \hyperref[asump_M_h2]{\textbf{(H2)}}.  The proof is complete.
\end{proof}


\begin{theorem}
\label{thm_verifyH3}
Under \eqref{asump_M_eq21} in Assumption \hyperref[asump_M_h1]{\textbf{(H1)}} and \eqref{lem_Bs_equiv_eq1},
Assumption \hyperref[asump_M_h3]{\textbf{(H3)}} holds
with the explicit constant
\begin{equation}
\label{thm_verifyH3_eq0}
        K_S
        =
        \left(1+\frac{2b \tau n_{\rm mg}}{\rho}\right)(1+\theta)^2.
\end{equation}
\end{theorem}

\begin{proof}
Recall from \eqref{eq3_E1E0} that
$
 E_s=-T_s^{n_{\rm mg}}S_s^{-1}
$.
In addition, we trivially have
$T_s^{n_{\rm mg}}S_s^{-1}=S_s^{-1/2} (S_s^{1/2}T_sS_s^{-1/2})^{n_{\rm mg}}S_s^{-1/2}$ is symmetric.  Hence $E_s$ and $\dot E_s$ are symmetric.
Thus
\begin{equation*}
\begin{split}
\label{thm_verifyH3_eq7}
        \dot E_s
        =-\frac{\dd}{\dd s}\left(T_s^{n_{\rm mg}}\right)S_s^{-1}
        +T_s^{n_{\rm mg}}S_s^{-1}\dot S_sS_s^{-1}.
 \end{split}
\end{equation*}
Noticing $\frac{\dd}{\dd s}\left(T_s^{n_{\rm mg}}\right)
        =\sum_{j=0}^{n_{\rm mg}-1}
        T_s^j\dot T_sT_s^{n_{\rm mg}-1-j}$,
we use \eqref{AB_ineq} to obtain
\begin{align}
\label{thm_verifyH3_eq8}
        \left\|S_s^{1/2}\dot E_sS_s^{1/2}\right\|_2
        &\le
       \underbrace{ \sum_{j=0}^{n_{\rm mg}-1}
        \|T_s^j\|_{S_s}\,\|\dot T_s\|_{S_s}\,
        \|T_s^{n_{\rm mg}-1-j}\|_{S_s}}_{(I)}
        + \underbrace{ \|T_s^{n_{\rm mg}}\|_{S_s}
        \left\|S_s^{-1/2}\dot S_sS_s^{-1/2}\right\|_2}_{(II)}.
\end{align}
We divide its estimation into several steps.

\textbf{Step 1}. Show
\begin{equation}
\label{thm_verifyH3_eq1}
        \left\|S_s^{-1/2}\dot S_s S_s^{-1/2}\right\|_2 \le \Gamma ,
        ~~~~
        \Gamma:=(1+\theta)\Theta .
\end{equation}
To see this, we first notice
$
        \dot M_s : = M-M_\star
$.
Then, \eqref{asump_M_eq21} yields
$\dot M_s \preceq \Theta M_\star$
and
$-\dot M_s = M_\star-M \preceq \Theta M$.
Combining these estimates with \eqref{lem_Bs_equiv_eq3}, we obtain
\begin{equation}
\label{thm_verifyH3_eq2}
        -\Gamma M_s \preceq \dot M_s \preceq \Gamma M_s,
        \qquad
        \Gamma:=(1+\theta)\Theta .
\end{equation}
Since $S_s=BM_s^{-1}B^T$,
we have
$
        \dot S_s
        = -B M_s^{-1}\dot M_s M_s^{-1}B^T
$.
Applying \eqref{thm_verifyH3_eq2} to this identity, we obtain
\begin{equation}
\label{thm_verifyH3_eq3}
        -\Gamma S_s \preceq \dot S_s \preceq \Gamma S_s
\end{equation}
which is \eqref{thm_verifyH3_eq1}.

\textbf{Step 2}.
Show
\begin{equation}
\label{thm_verifyH3_eq4}
        \left\|S_s^{1/2}\dot{\mathcal{B}}_sS_s^{1/2}\right\|_2
        \le b\Gamma .
\end{equation}
Restricting \eqref{thm_verifyH3_eq3} to level $i$ through \eqref{mg_level_i} gives
\begin{equation}
\label{thm_verifyH3_eq5}
        -\Gamma S_{i,s}\preceq \dot S_{i,s}\preceq \Gamma S_{i,s}
~~~~
        -\Gamma D_{i,s}\preceq \dot D_{i,s}\preceq \Gamma D_{i,s}.
\end{equation}
Since $R_{i,s}=\omega_iD_{i,s}^{-1}$,
we have
$        \dot R_{i,s}
        = -\omega_iD_{i,s}^{-1}\dot D_{i,s}D_{i,s}^{-1}
$.
Thus, we obtain
$ -\Gamma R_{i,s}\preceq \dot R_{i,s}\preceq \Gamma R_{i,s}$.
Summing over levels through \eqref{Bs_def} yields
$
        -\Gamma \mathcal{B}_s\preceq \dot{\mathcal{B}}_s\preceq \Gamma \mathcal{B}_s
$.
Combining it with the upper bound in Lemma \ref{lem_Bs_equiv} yields \eqref{thm_verifyH3_eq4}.

\textbf{Step 3}.
Show
\begin{equation}
\label{thm_verifyH3_eq9}
\| \dot{T_s} \|_{S_s} \le 2b\tau \Gamma.
\end{equation}
Differentiating $T_s=I-\tau\mathcal{B}_sS_s$ gives
$\dot T_s=-\tau\dot{\mathcal{B}}_sS_s-\tau\mathcal{B}_s\dot S_s$,
which then leads to
\begin{align*}
        \|\dot T_s\|_{S_s}
        &\le
        \tau\left\|S_s^{1/2}\dot{\mathcal{B}}_sS_s^{1/2}\right\|_2
        +\tau\left\|S_s^{1/2}\mathcal{B}_sS_s^{1/2}\right\|_2
        \left\|S_s^{-1/2}\dot S_sS_s^{-1/2}\right\|_2.
\end{align*}
Using \eqref{thm_verifyH3_eq1} and \eqref{thm_verifyH3_eq4}, we obtain the desired estimate \eqref{thm_verifyH3_eq9}.

\textbf{Step 4}. Show the estimate of $\dot{E_s}$
\begin{equation}
\label{thm_verifyH3_eq10}
        -K_S\delta_{\max}\Theta S_\star^{-1}
        \preceq \dot E_s
        \preceq
        K_S\delta_{\max}\Theta S_\star^{-1},
\end{equation}
where $K_S$ is given by \eqref{thm_verifyH3_eq0}. 

At last, using \eqref{thm_verifyH2_eq1}, we have $\|T_s\|_{S_s} \le \rho$ and thus conclude from \eqref{thm_verifyH3_eq8} that
\begin{align*}
        \left\|S_s^{1/2}\dot E_sS_s^{1/2}\right\|_2
        &\le
        2b\tau n_{\rm mg}\rho^{n_{\rm mg}-1}\Gamma
        +\rho^{n_{\rm mg}}\Gamma  
        =
        \left(1+\frac{2b \tau n_{\rm mg}}{\rho}\right)\delta \Gamma .
\end{align*}
Since $E_s$ is symmetric, this inequality is equivalent to
\begin{equation*}
        -c\,S_s^{-1}\preceq \dot E_s\preceq c\,S_s^{-1},
        \qquad
        c:=\left(1+\frac{2b \tau n_{\rm mg}}{\rho}\right)\delta \Gamma .
\end{equation*}
Finally, applying \eqref{lem_Bs_equiv_eq2} and $\Gamma=(1+\theta)\Theta$,
we obtain the desired estimate for $\dot{E_s}$ in \eqref{thm_verifyH3_eq10}.

At last, integrating over $s\in[0,1]$ through \eqref{eq_E1E0} finishes the proof. 
\end{proof}

}


\commentout{
\RGI{Here, \hyperref[asump_M_h2_new]{\textbf{(H3)}} is more difficult,
and we shall show that it will hold for large $t$.} 
First, it is not hard to see that \hyperref[asump_M_h2_new]{\textbf{(H3)}} relies on estimating the spectrum of the matrix 
$$
S^{1/2}_{\star} (\widetilde{S}^{-1}(\bfsigma_h)-S^{-1}(\bfsigma_h)) S^{1/2}_{\star} - 
S^{1/2}_{\star} (\widetilde{S}^{-1}_{\star}(\bfsigma^{\star}_{h,\phi})-S^{-1}_{\star}(\bfsigma^{\star}_{h,\phi})) S^{1/2}_{\star}.
$$
Define a function $\Gamma(\bfsigma_h;\bfx) = \bfx^T S^{1/2}_{\star} (\widetilde{S}^{-1}(\bfsigma_h)-S^{-1}(\bfsigma_h)) S^{1/2}_{\star}\bfx$.
For each given unit vector $\bfx$, there exists a Lipschitz constant $\hat{K}_1$ such that
\begin{equation}
\label{append_assump_verifyH3_eq3}
\Gamma(\bfsigma_h;\bfx) - \Gamma(\bfsigma^{\star}_{h,\phi};\bfx) \le \hat{K}_1 \| \sigma_h - \sigma^{\star}_{h,\phi} \|_{L^{\infty}},
\end{equation}
where the constant $\hat{K}_1$ depends on the bound of $\partial_{\sigma_{T_i}}\Gamma(\bfsigma_h;\bfx)$.
Without loss of generality, we only need to estimate the spectrum of $S^{1/2}_{\star} \partial_t (\widetilde{S}^{-1}(t)-S^{-1}(t)) S^{1/2}_{\star}$, 
which is equivalent to the one of $\partial_t (\widetilde{S}^{-1}(t)-S^{-1}(t)) S_{\star}$.
This demands the following estimate 
\begin{equation}
\label{append_assump_verifyH3_eq4}
\partial_t (\widetilde{S}^{-1}(t)-S^{-1}(t)) \preccurlyeq \hat{K}_2 \delta_{\max} S^{-1}_{\star}
\end{equation}
for some constant $\hat{K}_2$ to be specified.
Based on \eqref{append_assump_verifyH3_eq4}, as $S_{\star}$ is SPD, by classic results, we know 
$$
\lambda(\partial_t (\widetilde{S}^{-1}(t)-S^{-1}(t))  S_{\star}) \le \hat{K}_2 \delta_{\max} \lambda_{\min}(S^{-1}_{\star}S_{\star}) \le \hat{K}_2 \delta_{\max} .
$$
Then, it yields $\hat{K}_1 = \hat{K}_2 \delta_{\max}$, which finishes the proof:
\begin{equation}
\label{append_assump_verifyH3_eq6}
\Gamma(\bfsigma_h;\bfx) - \Gamma(\bfsigma^{\star}_{h,\phi};\bfx) \le \hat{K}_2 \delta_{\max} \gamma \| \sigma_h - \sigma^{\star}_{h,\phi} \|_{L^{2}(\Omega)},
\end{equation}
~\\
Now, we focus on \eqref{append_assump_verifyH3_eq4}.
By the derivative formula for matrices, we have
\begin{equation}
\label{append_assump_verifyH3_eq5}
\partial_t {S}^{-1}(t) = {S}^{-1} (\partial_t {S} ) {S}^{-1}
= {S}^{-1} B (\partial_t {M}^{-1} ) B^T {S}^{-1}
= {S}^{-1} B{M}^{-1} (\partial_t {M} ) {M}^{-1} B^T {S}^{-1}.
\end{equation}
\RGI{Using \eqref{append_assump_verify_eq0_1},
we know that $\lim_{t\rightarrow \infty } \partial_t {M} =0$, 
and thus for any $\delta_{\max}$, there exists a constant $\hat{K}_3$ such that
$\partial_t {M}  \preccurlyeq \hat{K}_3 \delta_{\max} M_{\star}$.
Then, Assumption \hyperref[asump_M_h1]{\textbf{(H1)}} leads to 
$\partial_t {M} \preccurlyeq \hat{K}_3 \delta_{\max} (1+\theta)M$,
and this estimate together with \eqref{append_assump_verifyH3_eq5} implies
$\partial_t {S}^{-1} \preccurlyeq  \hat{K}_3 (1+\theta)  \delta_{\max} S^{-1} \preccurlyeq \hat{K}_3 (1+\theta)^2  \delta_{\max}  S^{-1}_{\star}$.
It gives the estimate of $\partial_t {S}^{-1} $ in \eqref{append_assump_verifyH3_eq4}.
In addition, as for $\widetilde{S}$ in \eqref{append_assump_verifyH3_eq4}, notice that it is constructed by MG methods,
and thus we can express it as
$\widetilde{S} = \sum_{i=1}^LI_i R_i I_i^T$ with $I_i$ being a natural inclusion and $R_i$ being a smoother.
For widely-used Jacobi or Gauss-Seidel smoothers, we all have $\partial_t R_i = 0$.
With a similar argument above, we still have the similar estimate.}
}



\begin{thebibliography}{10}

\bibitem{2023AguiarFerreiraPrudente}
{\sc A.~A. Aguiar, O.~P. Ferreira, and L.~F. Prudente}, {\em Inexact gradient
  projection method with relative error tolerance}, Comput. Optim. Appl., 84
  (2023), pp.~363--395.

\bibitem{1987Amaya}
{\sc J.~Amaya}, {\em A differential equations approach to function
  minimization}, Rev. Mat. Apl., 4 (1987), pp.~1--7.

\bibitem{2005AndreaniBirginMartnez}
{\sc R.~Andreani, E.~G. Birgin, J.~M. Mart{\'\i}nez, and J.~Yuan}, {\em
  Spectral projected gradient and variable metric methods for optimization with
  linear inequalities}, IMA J. Numer. Anal., 25 (2005), pp.~221--252.

\bibitem{1989Antipin}
{\sc A.~Antipin}, {\em Continuous and iterative processes with projection and
  projection-type operators}, Problems in Cybernetics: Computational problems
  in the analysis of large systems,  (1989), pp.~5--43.

\bibitem{2003Antipin}
{\sc A.~S. Antipin}, {\em Minimization of convex functions on convex sets by
  means of differential equations}, Differ. Equ., 30 (1994).

\bibitem{2006Bacuta}
{\sc C.~Bacuta}, {\em A unified approach for {U}zawa algorithms}, SIAM J.
  Numer. Anal., 44 (2006), pp.~2633--2649.

\bibitem{1976Bertsekas}
{\sc D.~P. Bertsekas}, {\em On the {G}oldstein-{L}evitin-{P}olyak gradient
  projection method}, IEEE Trans. Autom. Control., 21 (1976), pp.~174--184.

\bibitem{2009BirginMartnezRaydan}
{\sc E.~G. Birgin, J.~Mart{\'\i}nez, and M.~Raydan}, {\em Spectral projected
  gradient methods}, Enc. Optim., 2 (2009).

\bibitem{2000BirginMartnezRaydan}
{\sc E.~G. Birgin, J.~M. Mart{\'\i}nez, and M.~Raydan}, {\em Nonmonotone
  spectral projected gradient methods on convex sets}, SIAM J. Optim., 10
  (2000), pp.~1196--1211.

\bibitem{2003BirginMartnezRaydan}
{\sc E.~G. Birgin, J.~M. Mart{\'\i}nez, and M.~Raydan}, {\em Inexact spectral
  projected gradient methods on convex sets}, IMA J. Numer. Anal., 23 (2003),
  pp.~539--559.

\bibitem{2014BirginMartnezRaydan}
{\sc E.~G. Birgin, J.~M. Mart{\'\i}nez, and M.~Raydan}, {\em Spectral projected
  gradient methods: Review and perspectives}, J. Stat. Softw., 60 (2014),
  pp.~1--21.

\bibitem{2013BoffiBrezziFortin}
{\sc D.~Boffi, F.~Brezzi, and M.~Fortin}, {\em {Mixed finite element methods
  and applications}}, vol.~44 of Springer Series in Computational Mathematics,
  Springer, Heidelberg, 2013.

\bibitem{1978Botsaris}
{\sc C.~Botsaris}, {\em Differential gradient methods}, J. Math. Anal. Appl.,
  63 (1978), pp.~177--198.

\bibitem{1981Botsaris}
\leavevmode\vrule height 2pt depth -1.6pt width 23pt, {\em A class of
  differential descent methods for constrained optimization}, J. Math. Anal.
  Appl., 79 (1981), pp.~96--112.

\bibitem{2023Boumal}
{\sc N.~Boumal}, {\em An introduction to optimization on smooth manifolds},
  Cambridge University Press, 2023.

\bibitem{1997BramblePasciakVassilev}
{\sc J.~H. Bramble, J.~E. Pasciak, and A.~T. Vassilev}, {\em Analysis of the
  inexact {U}zawa algorithm for saddle point problems}, SIAM J. Numer. Anal.,
  34 (1997), pp.~1072--1092.

\bibitem{2000BramblePasciakVassilev}
\leavevmode\vrule height 2pt depth -1.6pt width 23pt, {\em Uzawa type
  algorithms for nonsymmetric saddle point problems}, Math. Comput., 69 (2000),
  pp.~667--689.

\bibitem{1990BramblePasciakXu}
{\sc J.~H. Bramble, J.~E. Pasciak, and J.~Xu}, {\em Parallel multilevel
  preconditioners}, Math. Comp., 55 (1990).

\bibitem{1989BrownBartholomew}
{\sc A.~A. Brown and M.~C. Bartholomew-Biggs}, {\em {ODE} versus {SQP} methods
  for constrained optimization}, J. Optim. Theory Appl., 62 (1989),
  pp.~371--386.

\bibitem{2024ChenGuoWei}
{\sc L.~Chen, R.~Guo, and J.~Wei}, {\em Transformed primal-dual methods with
  variable-preconditioners}, SIAM J. Sci. Comput.,  (2026).

\bibitem{chen2023transformed}
{\sc L.~Chen and J.~Wei}, {\em Transformed primal-dual methods for nonlinear
  saddle point systems}, J. Numer. Math.,  (2023).

\bibitem{chen1998global}
{\sc X.~Chen}, {\em Global and superlinear convergence of inexact {Uzawa}
  methods for saddle point problems with nondifferentiable mappings}, SIAM J.
  Numer. Anal., 35 (1998), pp.~1130--1148.

\bibitem{1994ChenZhangxin}
{\sc Z.~Chen}, {\em {BDM} mixed methods for a nonlinear elliptic problem}, J.
  Comput. Appl. Math., 53 (1994), pp.~207--223.

\bibitem{cheng2003inexact}
{\sc X.-L. Cheng and J.~Zou}, {\em An inexact {U}zawa-type iterative method for
  solving saddle point problems}, Int. J. Comput. Math., 80 (2003), pp.~55--64.

\bibitem{1998ClarkeLedyaevStern}
{\sc F.~H. Clarke, Y.~S. Ledyaev, and R.~J. Stern}, {\em Asymptotic stability
  and smooth lyapunov functions}, J. Differential Equations, 149 (1998),
  pp.~69--114.

\bibitem{2014DevolderGlineurNesterov}
{\sc O.~Devolder, F.~Glineur, and Y.~Nesterov}, {\em First-order methods of
  smooth convex optimization with inexact oracle}, Math. Program., 146 (2014),
  pp.~37--75.

\bibitem{2025PaulJesusKonstantinos}
{\sc P.~Dobson, J.~M. Sanz-Serna, and K.~C. Zygalakis}, {\em On the connections
  between optimization algorithms, lyapunov functions, and differential
  equations: Theory and insights}, SIAM J. Optim., 35 (2025), pp.~537--566.

\bibitem{1981Dunn}
{\sc J.~C. Dunn}, {\em Global and asymptotic convergence rate estimates for a
  class of projected gradient processes}, SIAM J. Control Optim., 19 (1981),
  pp.~368--400.

\bibitem{1983Dykstra}
{\sc R.~L. Dykstra}, {\em An algorithm for restricted least squares
  regression}, J. Am. Stat. Assoc., 78 (1983), pp.~837--842.

\bibitem{1994EvtushenkoZhadan}
{\sc Y.~G. Evtushenko and V.~G. Zhadan}, {\em Stable barrier-projection and
  barrier-newton methods in linear programming}, Comput. Optim. Appl., 3
  (1994), pp.~289--303.

\bibitem{2022FerreiraLemesPrudente}
{\sc O.~P. Ferreira, M.~Lemes, and L.~F. Prudente}, {\em On the inexact scaled
  gradient projection method}, Comput. Optim. Appl., 81 (2022), pp.~91--125.

\bibitem{2021GaoSonAbsilStykel}
{\sc B.~Gao, N.~T. Son, P.-A. Absil, and T.~Stykel}, {\em Riemannian
  optimization on the symplectic stiefel manifold}, SIAM J. Optim., 31 (2021),
  pp.~1546--1575.

\bibitem{1974Goldstein}
{\sc A.~Goldstein}, {\em On gradient projection}, in Proc. 12th Ann. Allerton
  Conference and Circuits and Systems, Allerton Park, IL, 1974, pp.~38--40.

\bibitem{1964Goldstein}
{\sc A.~A. Goldstein}, {\em Convex programming in {H}ilbert space}, Bull. Amer.
  Math. Soc., 70 (1964).

\bibitem{2009GomesSantos}
{\sc M.~A. Gomes-Ruggiero, J.~M. Mart\'{\i}nez, and S.~A. Santos}, {\em
  Spectral projected gradient method with inexact restoration for minimization
  with nonconvex constraints}, SIAM J. Sci. Comput., 31 (2009), pp.~1628--1652.

\bibitem{2022DouglasMaxTiago}
{\sc D.~S. Gon{\c c}alves, M.~L.~N. Gon{\c c}alves, and T.~C. Menezes}, {\em
  Inexact variable metric method for convex-constrained optimization problems},
  Optimization, 71 (2022), pp.~145--163.

\bibitem{2016Han}
{\sc Q.~Han}, {\em Nonlinear Elliptic Equations of the Second Order}, Amer.
  Math. Soc., 2016.

\bibitem{2020PatrickDaniel}
{\sc P.~Henning and D.~Peterseim}, {\em Sobolev gradient flow for the
  {G}ross--{P}itaevskii eigenvalue problem: Global convergence and
  computational efficiency}, SIAM J. Numer. Anal., 58 (2020), pp.~1744--1772.

\bibitem{2006HiptmairWidmerZou}
{\sc R.~Hiptmair, G.~Widmer, and J.~Zou}, {\em {Auxiliary space preconditioning
  in $H_0(curl; \Omega)$}}, Numer. Math., 103 (2006), pp.~435--459.

\bibitem{2020HuLiuWenYuan}
{\sc J.~Hu, X.~Liu, Z.-W. Wen, and Y.-X. Yuan}, {\em A brief introduction to
  manifold optimization}, Journal of the Operations Research Society of China,
  8 (2020), pp.~199--248.

\bibitem{2001HuZou}
{\sc Q.~Hu and J.~Zou}, {\em An iterative method with variable relaxation
  parameters for saddle-point problems}, SIAM J. Matrix Anal. Appl., 23 (2001),
  pp.~317--338.

\bibitem{2002HuZou}
{\sc Q.~Hu and J.~Zou}, {\em Two new variants of nonlinear inexact {U}zawa
  algorithms for saddle-point problems}, Numer. Math., 93 (2002), pp.~333--359.

\bibitem{2007HuZou}
\leavevmode\vrule height 2pt depth -1.6pt width 23pt, {\em Nonlinear inexact
  {U}zawa algorithms for linear and nonlinear saddle-point problems}, SIAM J.
  Optim., 16 (2006), pp.~798--825.

\bibitem{2018HuangChenRui}
{\sc J.~Huang, L.~Chen, and H.~Rui}, {\em Multigrid methods for a mixed finite
  element method of the {D}arcy--{F}orchheimer model}, J. Sci. Comput., 74
  (2018), pp.~396--411.

\bibitem{2012JiangSunToh}
{\sc K.~Jiang, D.~Sun, and K.-C. Toh}, {\em An inexact accelerated proximal
  gradient method for large scale linearly constrained convex {SDP}}, SIAM J.
  Optim., 22 (2012), pp.~1042--1064.

\bibitem{2010PARIMAHMICHAEL}
{\sc P.~KAZEMI and M.~ECKART}, {\em Minimizing the {G}ross-{P}itaevskii energy
  functional with the sobolev gradient --- analytical and numerical results},
  Int. J. Comput. Methods, 07 (2010), pp.~453--475.

\bibitem{1998MuralidharanHanan}
{\sc M.~S. Kodialam and H.~Luss}, {\em Algorithms for separable nonlinear
  resource allocation problems}, Operations Research, 46 (1998), pp.~272--284.

\bibitem{1966LevitinPolyak}
{\sc E.~Levitin and B.~Polyak}, {\em Constrained minimization methods}, USSR
  Comput. Math. \& Math. Phys., 6 (1966), pp.~1--50.

\bibitem{2023Luo}
{\sc H.~Luo}, {\em Accelerated differential inclusion for convex optimization},
  Optimization, 72 (2023), pp.~1139--1170.

\bibitem{2022LuoChen}
{\sc H.~Luo and L.~Chen}, {\em From differential equation solvers to
  accelerated first-order methods for convex optimization}, Math. Program., 195
  (2022), pp.~735--781.

\bibitem{2019May}
{\sc R.~May}, {\em On the convergence of the continuous gradient projection
  method}, Optimization, 68 (2019), pp.~1791--1806.

\bibitem{2004MazencNesic}
{\sc F.~Mazenc and D.~Ne{\v{s}}i{\'c}}, {\em Strong lyapunov functions for
  systems satisfying the conditions of la salle}, IEEE Transactions on
  Automatic Control, 49 (2004), pp.~1026--1030.

\bibitem{2003Nesterov}
{\sc Y.~Nesterov}, {\em Introductory lectures on convex optimization: A basic
  course}, vol.~87, Springer Science \& Business Media, 2003.

\bibitem{2018PatrascuNecoara}
{\sc A.~Patrascu and I.~Necoara}, {\em On the convergence of inexact projection
  primal first-order methods for convex minimization}, IEEE Trans. Autom.
  Control., 63 (2018), pp.~3317--3329.

\bibitem{2015PatrikssonStrmberg}
{\sc M.~Patriksson and C.~Str{\"o}mberg}, {\em Algorithms for the continuous
  nonlinear resource allocation problem---new implementations and numerical
  studies}, European Journal of Operational Research, 243 (2015), pp.~703--722.

\bibitem{1967Polyak}
{\sc B.~Polyak}, {\em A general method for solving extremum problems}, Dokl.
  Akad. Nauk SSSR, 174 (1967).

\bibitem{2017PolyakShcherbakov}
{\sc B.~Polyak and P.~Shcherbakov}, {\em Lyapunov functions: An optimization
  theory perspective}, IFAC, 50 (2017), pp.~7456--7461.
\newblock 20th IFAC World Congress.

\bibitem{1960Rosen}
{\sc J.~B. Rosen}, {\em The gradient projection method for nonlinear
  programming. part i. linear constraints}, Journal of the Society for
  Industrial and Applied Mathematics, 8 (1960), pp.~181--217.

\bibitem{1961Rosen}
\leavevmode\vrule height 2pt depth -1.6pt width 23pt, {\em The gradient
  projection method for nonlinear programming. part ii. nonlinear constraints},
  Journal of the Society for Industrial and Applied Mathematics, 9 (1961),
  pp.~514--532.

\bibitem{2021SanzZygalakis}
{\sc J.~M. Sanz~Serna and K.~C. Zygalakis}, {\em The connections between
  {L}yapunov functions for some optimization algorithms and differential
  equations}, SIAM J. Numer. Anal., 59 (2021), pp.~1542--1565.

\bibitem{1977Scheurer}
{\sc B.~Scheurer}, {\em Existence et approximation de points selles pour
  certains probl{\`e}mes non lin{\'e}aires}, RAIRO. Anal. num{\'e}r., 11
  (1977), pp.~369--400.

\bibitem{2000SchroppSinger}
{\sc J.~Schropp and I.~Singer}, {\em A dynamical systems approach to
  constrained minimization}, Numer. Funct. Anal., 21 (2000), pp.~537--551.

\bibitem{2006Micheal}
{\sc M.~{\'O}. Searc{\'o}id}, {\em Metric Spaces}, Springer-Verlag, 2006.

\bibitem{2022ShiDuJordanSu}
{\sc B.~Shi, S.~S. Du, M.~I. Jordan, and W.~J. Su}, {\em Understanding the
  acceleration phenomenon via high-resolution differential equations}, Math.
  Program., 195 (2022), pp.~79--148.

\bibitem{2009ShikhmanStein}
{\sc V.~Shikhman and O.~Stein}, {\em Constrained optimization: Projected
  gradient flows}, J. Optim. Theory Appl., 140 (2009), pp.~117--130.

\bibitem{2016SuBoydCandes}
{\sc W.~Su, S.~Boyd, and E.~J. Cand`es}, {\em A differential equation for
  modeling {N}esterov's accelerated gradient method: Theory and insights}, J.
  Mach. Learn. Res.,  (2016).

\bibitem{1973TANABE}
{\sc K.~TANABE}, {\em An algorithm for the constrained maximization in
  nonlinear programming}, J. Operations Research Soc. of Japan, 17 (1973).

\bibitem{1980Tanabe}
{\sc K.~Tanabe}, {\em A geometric method in nonlinear programming}, J. Optim.
  Theory Appl., 30 (1980), pp.~181--210.

\bibitem{uzawa1958iterative}
{\sc H.~Uzawa}, {\em Iterative methods for concave programming}, Studies in
  linear and nonlinear programming, 6 (1958), pp.~154--165.

\bibitem{2006WangLiu}
{\sc C.~Wang and Q.~Liu}, {\em Convergence properties of inexact projected
  gradient methods}, Optimization, 55 (2006), pp.~301--310.

\bibitem{2021WilsonRechtJordan}
{\sc A.~C. Wilson, B.~Recht, and M.~I. Jordan}, {\em A {L}yapunov analysis of
  accelerated methods in optimization}, J. Mach. Learn. Res., 22 (2021),
  pp.~1--34.

\bibitem{2017XuZikatanov}
{\sc J.~Xu and L.~Zikatanov}, {\em Algebraic multigrid methods}, Acta Numer.,
  26 (2017), pp.~591--721.

\bibitem{2020XuYouseptZou}
{\sc Y.~Xu, I.~Yousept, and J.~Zou}, {\em An adaptive edge element
  approximation of a quasilinear {H}(curl)-elliptic problem}, Math. Models
  Methods Appl. Sci., 30 (2020), pp.~2799--2826.

\bibitem{1980Hiroshi}
{\sc H.~Yamashita}, {\em A differential equation approach to nonlinear
  programming}, Math. Program., 18 (1980), pp.~155--168.

\bibitem{2026ZengZhanGuoWei}
{\sc J.~Zeng, D.~Zhan, R.~Guo, and C.~Wei}, {\em Variable-preconditioned
  transformed primal--dual method for generalized wasserstein gradient flows},
  J. Comput. Phys., 563 (2026), p.~115103.

\end{thebibliography}

\end{document}